\documentclass[11pt,a4paper]{amsart}

\usepackage[utf8]{inputenc}
\usepackage[english]{babel}
\usepackage{amsmath,amssymb,amsthm}
\usepackage[colorlinks,bookmarks,linkcolor=black,citecolor=black]{hyperref}
\usepackage{graphicx}
\usepackage{tikz}
\usepackage{enumitem}
\usepackage{subcaption}
\usepackage{fullpage}

\newtheorem{defn}{Definition}
\newtheorem{conj}{Conjecture}
\newtheorem{theorem}{Theorem}[section]
\newtheorem{lemma}[theorem]{Lemma}

\newtheorem{corollary}[theorem]{Corollary}
\newtheorem{claim}[theorem]{Claim}
\theoremstyle{remark}

\DeclareMathOperator{\intr}{int} 
\DeclareMathOperator{\extr}{ext} 
\DeclareMathOperator{\pp}{pp} 
\DeclareMathOperator{\pc}{pc} 
\DeclareMathOperator{\ckf}{CKF} 

\newcommand{\NN}{\mathbb{N}}
\newcommand{\eps}{\varepsilon}
\newcommand{\nth}{\varnothing}
\newcommand{\config}[3]{$\dag_{#1}(#2,#3)$}

\def\endofClaim{\hfill\scalebox{.6}{$\blacksquare$}}
\newcommand{\oldqed}{}
\newenvironment{claimproof}[1][Proof]{
  \renewcommand{\oldqed}{\qedsymbol}
  \renewcommand{\qedsymbol}{\endofClaim}
  \begin{proof}[#1]
}{
  \end{proof}
  \renewcommand{\qedsymbol}{\oldqed}
}

\newcommand{\By}[2]{\overset{\mbox{\tiny{#1}}}{#2}} 
\newcommand{\ByRef}[2]{   \By{\eqref{#1}}{#2} }

\newcommand{\leByRef}[1]{ \ByRef{#1}{\le} } 
\newcommand{\geByRef}[1]{ \ByRef{#1}{\ge} }

\title{Minimum degrees for powers of paths and cycles}
\author{Eng Keat Hng}
\date{}

\thanks{Department of Mathematics, London School of Economics, Houghton Street, London, WC2A 2AE, United Kingdom, \texttt{e.hng@lse.ac.uk}, supported by an LSE PhD Studentship.}

\begin{document}

\begin{abstract}
We study minimum degree conditions under which a graph $G$ contains $k$th powers of paths and cycles of arbitrary specified lengths. We determine precise thresholds, assuming that the order of $G$ is large. This extends a result of Allen, B\"ottcher and Hladk\'y [J.\ Lond.\ Math.\ Soc.\ (2) 84(2) (2011), 269--302] concerning the containment of squares of paths and squares of cycles of arbitrary specified lengths and settles a conjecture of theirs in the affirmative.
\end{abstract}

\maketitle

\section{Introduction} \label{section:intro}

The study of conditions on vertex degrees in a host graph $G$ for the appearance of a target graph $H$ is a major theme in extremal graph theory. A classical result in this area is the following theorem of Dirac about the existence of a Hamiltonian cycle.

\begin{theorem}[Dirac~\cite{Dirac}] \label{thm:dirac}
Every graph on $n\geq3$ vertices with minimum degree at least $\frac{n}{2}$ has a Hamiltonian cycle.
\end{theorem}

We write $C_{\ell}$ (resp.\ $P_{\ell}$) for a cycle (resp.\ path) of \emph{length} $\ell$, that is, a cycle (resp.\ path) on $\ell$ vertices. The $k$th power of a graph $G$, denoted by $G^k$, is obtained from $G$ by joining every pair of vertices at distance at most $k$. In 1962, P\'osa conjectured an analogue of Dirac's theorem for the containment of the square of a Hamiltonian cycle. This conjecture was extended in 1974 by Seymour to general powers of a Hamiltonian cycle.

\begin{conj}[P\'osa--Seymour Conjecture~\cite{Seymour}]  \label{conj:posa-seymour}
Let $k \in \NN$. A graph $G$ on $n \ge 3$ vertices with minimum degree $\delta(G)\geq\frac{kn}{k+1}$ contains the $k$th power of a Hamiltonian cycle.
\end{conj}

Fan and Kierstead made significant progress, proving an approximate version of this conjecture for the square of paths and the square of cycles in sufficiently large graphs~\cite{FanKierstead-square-path-cycle} and determining the best-possible minimum degree condition for the square of a Hamiltonian path~\cite{FanKierstead-square-path}. Koml\'os, S\'ark\"ozy and Szemer\'edi confirmed the truth of the P\'osa--Seymour Conjecture for sufficiently large graphs.

\begin{theorem}[Koml\'os, S\'ark\"ozy and Szemer\'edi~\cite{KomlosSarkozySzemeredi-seymour}] \label{thm:komlos-sarkozy-szemeredi-seymour}
For every positive integer $k$, there exists an integer $n_0=n_0(k)$ such that for all integers $n\geq n_0$, any graph $G$ on $n$ vertices with minimum degree at least $\frac{kn}{k+1}$ contains the $k$th power of a Hamiltonian cycle.
\end{theorem}

In fact, their proof asserts a stronger result, guaranteeing $k$th powers of cycles of all lengths divisible by $k+1$ between $k+1$ and $n$, in addition to the $k$th power of a Hamiltonian cycle. The divisibility condition is necessary as balanced complete $(k+1)$-partite graphs contain $k$th powers of cycles of no other length.

\begin{theorem}[Koml\'os--S\'ark\"ozy--Szemer\'edi~\cite{KomlosSarkozySzemeredi-seymour}] \label{thm:komlos-sarkozy-szemeredi-seymour-gen}
For every positive integer $k$, there exists an integer $n_0=n_0(k)$ such that for all integers $n\geq n_0$, any graph $G$ on $n$ vertices with minimum degree $\delta(G)\geq\frac{kn}{k+1}$ contains the $k$th power of a cycle $C_{(k+1)\ell}^k$ for any $1\leq\ell\leq\frac{n}{k+1}$.
\end{theorem}

Recently there has been an interest in generalising the P\'osa-Seymour Conjecture. Allen, B\"ottcher and Hladk\'y~\cite{AllenBoettcherHladky} determined the exact minimum degree threshold for a large graph to contain the square of a cycle of a given length. Staden and Treglown~\cite{StadenTreglown} proved a degree sequence analogue for the square of a Hamiltonian cycle. Ebsen, Maesaka, Reiher, Schacht and Sch\"ulke~\cite{EbsenMaesakaReiherSchachtSchulke} showed that inseparable graphs which are sufficiently uniformly dense contain powers of Hamiltonian cycles. Recently, Lang and Sanhueza-Matamala~\cite{LangSanhueza-Matamala} introduced the concept of Hamilton frameworks and proved that robust aperiodic Hamilton frameworks contain powers of Hamiltonian cycles. There has also been related work in the hypergraph setting for tight cycles and tight components. R\"odl, Ruci\'nski and Szemer\'edi~\cite{RodlRucinskiSzemeredi} established the minimum codegree threshold for a tight Hamiltonian cycle in $k$-uniform hypergraphs. Allen, B\"ottcher, Cooley and Mycroft~\cite{AllenBoettcherCooleyMycroft} proved an asymptotically tight result on the minimum codegree threshold for a tight cycle of a given length in $k$-partite $k$-uniform hypergraphs. The problem of minimum codegree thresholds for tight components of a given size has also been studied by Georgakopoulos, Haslegrave and Montgomery~\cite{GeorgakopoulosHaslegraveMontgomery}.

In this paper we are interested in exact minimum degree thresholds for the appearance of the $k$th power of a path $P_\ell^k$ and the $k$th power of a cycle $C_\ell^k$. One possible guess as to what minimum degree $\delta=\delta(G)$ will guarantee which length $\ell=\ell(n,\delta)$ of $k$th power of a path (or longest $k$th power of a cycle) is the following. Since the minimum degree threshold for the $k$th power of a \emph{Hamiltonian} cycle (or path) is roughly the same as that for a spanning $K_{k+1}$-factor, perhaps this remains true for smaller $\ell$. If this were true, it would mean that one could expect that $\ell(n,\delta)$ would be roughly $(k+1)(k\delta-(k-1)n)$. This is characterised by $(k+1)$-partite extremal examples, which are exemplified by the $k=3$ example in Figure~\ref{fig:partite-extremal}.

\begin{figure}
    \centering
    \includegraphics[height=6.5cm,width=10cm]{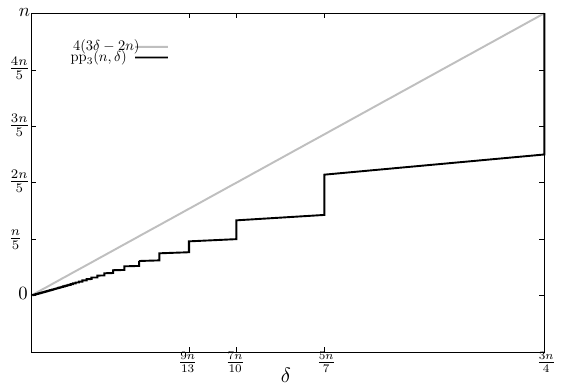} 
    \caption{The behaviour of $\pp_3(n,\delta)$} 
    \label{fig:cube-path}
\end{figure}

However, this was shown not to give the correct answer by Allen, B\"ottcher and Hladk\'y~\cite{AllenBoettcherHladky}. For the case $k=2$ (see Theorem~\ref{thm:cycle-path-square}), they determined sharp thresholds attained by a family of extremal graphs which exhibit not a linear dependence between the length of the longest square of a path and the minimum degree, but rather piecewise linear dependence with jumps at certain points. In order to state the result of~\cite{AllenBoettcherHladky} as well as our result, we first introduce the following functions. Given positive integers $k,n,\delta$ with $\delta\in\left(\frac{(k-1)n}{k},n-1\right]$, we define
\begin{equation} \label{eqn:r-expression}
\begin{split}
r_p(k,n,\delta)&:=\max\left\lbrace r\in\NN:\left\lfloor\tfrac{(k-1)\delta-(k-2)n}{r}\right\rfloor>k\delta-(k-1)n\right\rbrace, \\
r_c(k,n,\delta)&:=\max\left\lbrace r\in\NN:\left\lceil\tfrac{(k-1)\delta-(k-2)n}{r}\right\rceil>k\delta-(k-1)n\right\rbrace.
\end{split}
\end{equation}
Note that $r_p(k,n,\delta)$ and $r_c(k,n,\delta)$ are almost always the same, differing only for a very small number of values of $\delta$. Setting $s_p(k,n,\delta):=\left\lceil\tfrac{(k-1)\delta-(k-2)n}{r_p(k,n,\delta)}\right\rceil$ and $s_c(k,n,\delta):=\left\lceil\tfrac{(k-1)\delta-(k-2)n}{r_c(k,n,\delta)}\right\rceil$, we define
\begin{equation} \label{eqn:pp-pc-expression}
\begin{split}
\pp_k(n,\delta)&:=\min\left\lbrace(k-1)\left(\left\lfloor\tfrac{s_p(k,n,\delta)}{2}\right\rfloor+1\right)+s_p(k,n,\delta),n\right\rbrace, \\
\pc_k(n,\delta)&:=\min\left\lbrace(k-1)\left\lfloor\tfrac{s_c(k,n,\delta)}{2}\right\rfloor+s_c(k,n,\delta),n\right\rbrace.
\end{split}
\end{equation}
Note that the functions $\pc_k(n,\delta)$ and $\pp_k(n,\delta)$ satisfy $\pc_k(n,\delta)\leq\pp_k(n,\delta)$. They also behave very similarly and differ only by a constant (dependent only on $k$) when $r_p$ and $r_c$ are equal. The behaviour of $\pp_3(n,\delta)$ is illustrated in Figure~\ref{fig:cube-path}.

Before we discuss the result of Allen, B\"ottcher and Hladk\'y~\cite{AllenBoettcherHladky} and our result, we shall define two closely related families of graphs which will serve as examples of extremal graphs. We obtain the $n$-vertex graph $G_p(k,n,\delta)$\label{Gpkndelta} by starting with the disjoint union of $k-1$ independent sets $I_1,\dots,I_{k-1}$ and $r:=r_p(k,n,\delta)$ cliques $X_1,\dots,X_r$ with $|I_1|=\dots=|I_{k-1}|=n-\delta$ and $|X_1|\geq\dots\geq|X_r|\geq|X_1|-1$. Then, insert all edges between $X_i$ and $I_j$ for each $(i,j)\in[r]\times[k-1]$ and all edges between $I_i$ and $I_j$ for each $(i,j)\in\binom{[k-1]}{2}$. This is a natural generalisation of the construction in~\cite{AllenBoettcherHladky}. Figure~\ref{fig:component-extremal} shows an example with $k=3$. Construct the graph $G_c(k,n,\delta)$ in the same way as $G_p(k,n,\delta)$ but with $r:=r_c(k,n,\delta)$ and with arbitrary selection of a vertex $v\in X_1$ and insertion of all edges between $v$ and $X_i$ for each $i\in[r]$ such that $|X_i|\neq|X_1|$.

\begin{figure}      
 	\begin{subfigure}[b]{0.5\textwidth}       
 		\centering          
 		\begin{tikzpicture}
 			\node [draw, circle, radius=2, black, above, below] (ind1) at (-2,0) {$n-\delta$};
 			\node [draw, circle, radius=2, black, above, below] (ind2) at (0,1) {$n-\delta$};
 			\node [draw, circle, radius=2, black, above, below] (ind3) at (2,0) {$n-\delta$};
 			\node [draw, circle, radius=2, black, above, below] (ind4) at (0,-1) {$3\delta-2n$};
 			\draw [ultra thick] (ind1) to (ind2);
 			\draw [ultra thick] (ind1) to (ind3);
 			\draw [ultra thick] (ind2) to (ind3);
 			\draw [ultra thick] (ind1) to (ind4);
 			\draw [ultra thick] (ind2) to (ind4);
 			\draw [ultra thick] (ind3) to (ind4);
 		\end{tikzpicture}
 		\caption{$K_{n-\delta,n-\delta,n-\delta,3\delta-2n}$}
 		\label{fig:partite-extremal}
     \end{subfigure}      
 	\begin{subfigure}[b]{0.4\textwidth}      
 		\centering          
 		\begin{tikzpicture}
 			\node [draw, circle, radius=2, black, above, below] (ind1) at (-1,0) {$n-\delta$};
 			\node [draw, circle, radius=2, black, above, below] (ind2) at (1,0) {$n-\delta$};
 			\node [draw, circle, radius=2, black, fill=black, above, below] (ext1) at (-2,-2) {$x$};
 			\node [draw, circle, radius=2, black, fill=black, above, below] (ext2) at (0,-2) {$x$};
 			\node [draw, circle, radius=2, black, fill=black, above, below] (ext3) at (2,-2) {$x$};
 			\draw [ultra thick] (ind1) to (ind2);
 			\draw [ultra thick] (ind1) to (ext1);
 			\draw [ultra thick] (ind1) to (ext2);
 			\draw [ultra thick] (ind1) to (ext3);
 			\draw [ultra thick] (ind2) to (ext1);
 			\draw [ultra thick] (ind2) to (ext2);
 			\draw [ultra thick] (ind2) to (ext3);
 		\end{tikzpicture}
 		\caption{$G_p(3,n,\delta)$}
 		\label{fig:component-extremal}
     \end{subfigure}
     \caption{Graphs for $k=3$}
     \label{fig:extremal-graphs}
\end{figure}
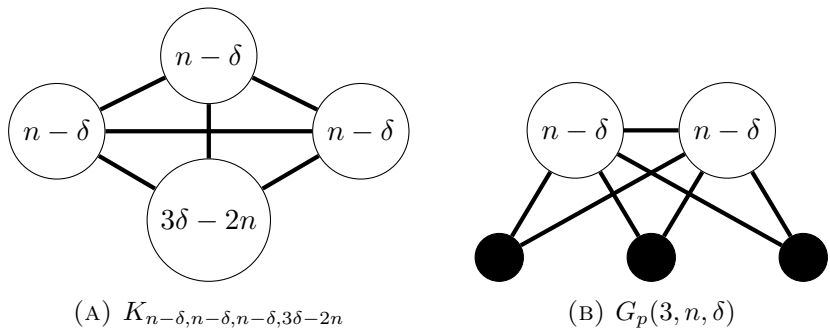

Let us now discuss $k$th powers of paths and cycles in $G_p(k,n,\delta)$ and $G_c(k,n,\delta)$ respectively. We focus on the former as the discussion for the latter is analogous. Consider an arbitrary $k$th power of a path $P^k_\ell\subseteq G_p(k,n,\delta)$ with its vertices in a natural order. Any $k+1$ consecutive vertices form a clique, so any $k+1$ consecutive vertices contain vertices from at most one clique $X_i$. Therefore, $P^k_\ell$ contains vertices from at most one clique $X_i$. Since $P^k_\ell$ has independence number $\left\lceil\frac{\ell}{k+1}\right\rceil$ and $I_i$ is independent for each $i\in[k-1]$, we have $\ell-(k-1)\left\lceil\frac{\ell}{k+1}\right\rceil\leq\left\lceil\tfrac{(k-1)\delta-(k-2)n}{r_p(k,n,\delta)}\right\rceil$ and thus we deduce $\ell\leq\pp_k(n,\delta)$. Finally, observe that we can construct a copy of $P^k_{\pp_k(n,\delta)}$ in $G_p(k,n,\delta)$ as follows. Repeatedly take an unused vertex from $I_i$ for each $i\in[k-1]$ and two unused vertices from $X_1$ in turn, until all vertices of $X_1$ are used and skipping $I_i$ for each $i\in[k-1]$ if they become entirely used before $X_1$ does.

The following is the result of Allen, B\"ottcher and Hladk\'y~\cite{AllenBoettcherHladky} for the case $k=2$. It states that $\pp_2(n,\delta)$ and $\pc_2(n,\delta)$ are the maximal lengths of squares of paths and squares of cycles, respectively, guaranteed in an $n$-vertex graph $G$ with minimum degree $\delta$. Furthermore, $G$ also contains any shorter square of a cycle with length divisible by $3$. These results are tight with $G_p(2,n,\delta)$ and $G_p(2,n,\delta)$ serving as extremal examples. In fact, both graphs contain the squares of cycles $C^2_\ell$ for all lengths $3\leq\ell\leq\pc_2(n,\delta)$ such that $\chi(C^2_\ell)\leq4$. If $G$ does not contain any one of these squares of cycles with chromatic number $4$, then~\ref{item:cycle-square-extra-longer} of Theorem~\ref{thm:cycle-path-square} guarantees even longer squares of cycles $C^2_\ell$ in $G$, where $\ell$ is divisible by $3$.

\begin{theorem}[Allen, B\"ottcher and Hladk\'y \cite{AllenBoettcherHladky}] \label{thm:cycle-path-square}
For any $\nu>0$ there exists an integer $n_0$ such that for all integers $n>n_0$ and $\delta\in[(\frac{1}{2}+\nu)n,\frac{2n-1}{3}]$ the following holds for all graphs $G$ on $n$ vertices with minimum degree $\delta(G)\geq\delta$.
\begin{enumerate}[label=(\roman*)]
    \item \label{item:cycle-path-square-div} $P_{\pp_2(n,\delta)}^2\subseteq G$ and $C_\ell^2\subseteq G$ for every $\ell\in\NN$ with $\ell\in[3,\pc_2(n,\delta)]$ such that $3$ divides $\ell$.
    \item \label{item:cycle-square-extra-longer} Either $C_\ell^2\subseteq G$ for every $\ell\in\NN$ with $\ell\in[3,\pc_2(n,\delta)]$ and $\chi(C^2_\ell)\leq4$, or $C_\ell^2\subseteq G$ for every $\ell\in\NN$ with $\ell\in[3,6\delta-3n-\nu n]$ such that $3$ divides $\ell$.
\end{enumerate}
\end{theorem}

It was conjectured by Allen, B\"ottcher, and Hladk\'y~\cite[Conjecture 24]{AllenBoettcherHladky} that their result can be naturally generalised to higher powers. Our main result is that their conjecture is indeed true. Note that $\chi(C^k_\ell)\leq k+2$ holds for all $\ell\geq k^2+k$, so this condition excludes only a number of lengths which is a function of $k$.

\begin{theorem} \label{thm:cycle-path-power}
Given an integer $k\geq3$ and $0<\nu<1$ there exists an integer $n_0$ such that for all integers $n\geq n_0$ and $\delta\in\left[\left(\frac{k-1}{k}+\nu\right)n,\frac{kn}{k+1}\right)$ the following holds for all graphs $G$ on $n$ vertices with minimum degree $\delta(G)\geq\delta$.
\begin{enumerate}[label=(\roman*)]
\item \label{item:cycle-path-power-min} $P^k_{\pp_k(n,\delta)}\subseteq G$ and $C^k_\ell\subseteq G$ for every $\ell\in\NN$ with $\ell\in[k+1,\pc_k(n,\delta)]$ such that $k+1$ divides $\ell$.
\item \label{item:cycle-power-parity cases} Either $C^k_\ell\subseteq G$ for every $\ell\in\NN$ with $\ell\in[k+1,\pc_k(n,\delta)]$ such that $\chi(C^k_\ell)\leq k+2$, or $C^k_\ell\subseteq G$ for every $\ell\in\NN$ with $\ell\in[k+1,(k+1)(k\delta-(k-1)n)-\nu n]$ such that $k+1$ divides $\ell$.
\end{enumerate}
\end{theorem}

As with the result of Allen, B\"ottcher and Hladk\'y~\cite{AllenBoettcherHladky}, our result is also tight with $G_p(k,n,\delta)$ and $G_c(k,n,\delta)$ serving as extremal examples. Note that~\ref{item:cycle-power-parity cases} implies the $k$th powers of cycles case of~\ref{item:cycle-path-power-min}, as the latter is precisely the common part of the two cases in~\ref{item:cycle-power-parity cases}. Hence, it will be sufficient to prove~\ref{item:cycle-power-parity cases} and the first part of~\ref{item:cycle-path-power-min}.

We remark that while our proof uses the same basic strategy as used in~\cite{AllenBoettcherHladky} for the proof of Theorem~\ref{thm:cycle-path-square} (that is, combining the regularity method and the stability method), our proof is not merely a generalisation of the proof of Theorem~\ref{thm:cycle-path-square}. In particular, the proof of our Stability Lemma turns out to be much more complex than in~\cite{AllenBoettcherHladky}, and the analysis requires new insights.

The remainder of the paper is organised as follows. In Section~\ref{section:prelim} we introduce our notation and tools. In Section~\ref{section:lem-thmproof} we outline our proof strategy and state the key lemmas in our proof. Then, we provide a proof of Theorem~\ref{thm:cycle-path-power} which applies these lemmas. The main difficulty in our proof is proving a Stability Lemma (that is, Lemma~\ref{lem:stability}). In Section~\ref{section:stability} we provide proofs for two special cases of our Stability Lemma and introduce a family of configurations which enables analysis of the general case. In Section~\ref{section:structure-methods} we analyse the aforementioned family of configurations and develop greedy-type methods, which we will subsequently use in Section~\ref{section:stability-general-cmpnt} in the proof of the general case of our Stability Lemma. Finally, we provide a proof of our Extremal Lemma in Section~\ref{section:extremal}.

\section{Preliminaries} \label{section:prelim}

In this section we introduce the notation we will use and provide various tools we will need. We will also establish some useful properties of the functions introduced in~\eqref{eqn:r-expression} and~\eqref{eqn:pp-pc-expression}.

\subsection{Notation} \label{subsection:notation}

Write $\NN$ for the set of positive integers and $\NN_0$ for the set $\NN\cup\{0\}$. For $m\in\NN_0$ write $[m]$ for the set $\{1,\dots,m\}$ and $[m]_0$ for $[m]\cup\{0\}$. For a graph $G$ denote its vertex set and edge set by $V(G)$ and $E(G)$ respectively. Set $v(G):=|V(G)|$ and $e(G):=|E(G)|$. For sets $X,Y\subseteq V(G)$, set $E(X,Y):=\{xy\in E(G):x\in X,y\in Y\}$ and $e(X,Y):=|E(X,Y)|$. Let $G[X]$ denote the subgraph of $G$ {\em induced} by $X$. For a vertex $v\in V(G)$ and a subset $A\subseteq V(G)$ we denote by $\Gamma_G(v;A)$ the {\em neighbourhood} in $A$ of $v$ in $G$ and write $\deg_G(v;A)$ for its cardinality $|\Gamma_G(v;A)|$. Given $X\subseteq V(G)$ let $\Gamma_G(X;A):=\bigcap_{v\in X}\Gamma_G(v;A)$ denote the {\em common neighbourhood} in $A$ of vertices from $X$ in $G$ and write $\deg_G(X;A)$ for its cardinality $|\Gamma_G(X;A)|$. We will omit the set brackets in $\Gamma_G(\{v_1,\dots,v_\ell\};A)$ and $\deg_G(\{v_1,\dots,v_\ell\};A)$, and write $\Gamma_G(v_1,\dots,v_\ell;A)$ and $\deg_G(v_1,\dots,v_\ell;A)$ respectively instead. We omit the graph $G$ in the subscripts if it is clear from context. Furthermore, we omit the set $A$ if we intend $A=V(G)$. Denote the minimum degree of a graph $G$ by $\delta(G)$. We write $v_1\cdots v_\ell$ to denote a clique in $G$ with vertices $v_1,\dots,v_\ell$. For an event $\mathcal{A}$ we write $\mathbf{1}_\mathcal{A}$ to denote its indicator function.

\subsection{Tools} \label{subsection:tools}

We will need the following simple observations about matchings in graphs with given minimum degree.

\begin{lemma} \label{lem:min-deg-matching} \leavevmode
\begin{enumerate}[label=(\roman*)]
\item \label{lem:min-deg-matching-gen} A graph $G$ contains a matching with $\min\{\delta(G),\left\lfloor\frac{|V(G)|}{2}\right\rfloor\}$ edges.
\item \label{lem:min-deg-matching-bipartite} Let $G=(U\cup V,E)$ be a bipartite graph with vertex classes $U$ and $V$ such that every vertex in $U$ has degree at least $u$ and every vertex in $V$ has degree at least $v$. Then $G$ contains a matching with $\min\{u+v,|U|,|V|\}$ edges.
\end{enumerate}
\end{lemma}

\begin{proof}
For~\ref{lem:min-deg-matching-gen}, let $M$ be a maximum matching in $G$. We are done unless there are two vertices $u,v\in V(G)$ not contained in $M$. $M$ is maximal so all neighbours of $u$ and $v$ are contained in $M$. There cannot be an edge $u'v'$ in $M$ with $uv',vu'\in E(G)$ by maximality of $M$, since then $uv'u'v$ would be an $M$-augmenting path. But this means that $\deg(u;e)+\deg(v;e)\leq2$ for each $e\in M$, which implies that
\[\delta(G)+\delta(G)\leq\deg(u)+\deg(v)=\sum_{e\in M}\deg(u;e)+\deg(v;e)\leq2|M|\]
and hence $|M|\geq\delta(G)$.

For~\ref{lem:min-deg-matching-bipartite}, let $M$ be a maximum matching in $G$. We are done unless there are  vertices $u\in U$ and $v\in V$ not contained in $M$. There cannot be an edge $u'v'$ in $M$ with $uu',vv'\in E$ by maximality of $M$, since then $uu'v'v$ would be an $M$-augmenting path. Now $M$ is maximal so all neighbours of $u$ and $v$ are contained in $M$. This means that $\deg(u;e)+\deg(v;e)\leq1$ for each $e\in M$, which implies that
\[u+v\leq\deg(u)+\deg(v)=\sum_{e\in M}\deg(u;e)+\deg(v;e)\leq|M|\]
and hence $|M|\geq u+v$.
\end{proof}

It will be useful to have the following simple observations about sizes of common neighbourhoods and maximal cliques.

\begin{lemma} \label{lem:common-nbrhood-size}
Let $k\in\NN$ be an integer, $u_1,\dots,u_k$ be vertices of a graph $G$ and $U\subseteq V(G)$. Then $\deg(u_1,\dots,u_k;U)\geq\sum_{i=1}^k\deg(u_i;U)-(k-1)|U|$. In particular, if $\delta(G)\geq\delta$ then $\deg(u_1,\dots,u_k)\geq k\delta-(k-1)n$.
\end{lemma}

\begin{proof}
Let $X:=\{u_i \mid i\in[k]\}$. Count $\rho:=\sum_{i\in[k],v\in U}\mathbf{1}_{\{vu_i\in E(G)\}}$ in two ways. On the one hand, $\rho=\sum_{i\in[k]}\sum_{v\in U}\mathbf{1}_{\{vu_i\in E(G)\}} =\sum_{i\in[k]}\deg(u_i;U)$. On the other hand, noting that vertices in $U\setminus\Gamma(X)$ have at most $k-1$ neighbours in $X$, we obtain
\begin{align*}
\rho&=\sum_{v\in U}\sum_{i\in[k]}\mathbf{1}_{\{vu_i\in E(G)\}}=\sum_{v\in U}\deg(v;X) \\
&=\sum_{v\in\Gamma(X;U)}\deg(v;X) + \sum_{v\in U\setminus\Gamma(X)}\deg(v;X) \\
&\leq k\deg(X;U)+(k-1)(|U|-\deg(X;U))=\deg(X;U)+(k-1)|U|.
\end{align*}
It follows that $\deg(X;U)\geq\sum_{i=1}^k\deg(u_i;U)-(k-1)|U|$. Furthermore, if $\delta(G)\geq\delta$ then $\deg(u_i)\geq\delta$ for each $i\in[k]$, so it follows immediately that $\deg(X)\geq k\delta-(k-1)n$.
\end{proof}

\begin{lemma} \label{lem:clique-extn}
Let $j,k$ and $\ell$ be integers satisfying $1\le j\le\ell\le k+1$ and let $G$ be a graph on $n$ vertices with minimum degree $\delta(G)>\frac{(k-1)n}{k}$. Then every copy of $K_j$ in $G$ can be extended to a copy of $K_\ell$ in $G$.
\end{lemma}

\begin{proof}
Fix integers $k,\ell$ satisfying $1\leq\ell\leq k+1$ and proceed by backwards induction on $j$. The case $j=\ell$ is trivial. For $1\leq j<\ell$, note that by Lemma~\ref{lem:common-nbrhood-size} a copy of $K_j$ has common neighbourhood of size at least $j\delta-(j-1)n>0$. Therefore, we can extend it to a copy of $K_{j+1}$ by adding to it a vertex in its common neighbourhood. The resultant copy of $K_{j+1}$ can be extended to a copy of $K_\ell$ by the induction hypothesis.
\end{proof}

The following is a classical result of Hajnal and Szemer\'{e}di~\cite{HajnalSzemeredi}.

\begin{theorem}[Hajnal and Szemer\'{e}di~\cite{HajnalSzemeredi}] \label{thm:hajnal-szemeredi}
For any graph $G$ on $n$ vertices with maximum degree $\Delta(G)$ and any integer $r\geq\Delta(G)+1$, there is a partition of $V(G)$ into $r$ independent sets which are each of size $\left\lceil\frac{n}{r}\right\rceil$ or $\left\lfloor\frac{n}{r}\right\rfloor$.
\end{theorem}

For our purposes we will need the following corollary of Theorem~\ref{thm:hajnal-szemeredi}.

\begin{corollary} \label{cor:hajnal-szemeredi-cor}
Let $k\in\NN$. Let $G$ be a graph on $n\geq k(k+1)$ vertices with $\delta:=\delta(G)\geq\frac{(k-1)n}{k}$. Then $G$ contains $\min\left\lbrace k\delta-(k-1)n,\left\lfloor\frac{n}{k+1}\right\rfloor\right\rbrace$ vertex-disjoint copies of $K_{k+1}$.
\end{corollary}

\begin{proof} [Proof of Corollary \ref{cor:hajnal-szemeredi-cor}]
First consider $\delta\in\left(\frac{(k-1)n}{k},\frac{kn}{k+1}\right)$. Apply Theorem~\ref{thm:hajnal-szemeredi} to $\overline{G}$ with $r:=\Delta(\overline{G})+1=n-\delta\in\left(\frac{n}{k+1},\frac{n}{k}\right)$. Each part in the resultant partition has size $\left\lceil\frac{n}{r}\right\rceil=k+1$ or $\left\lfloor\frac{n}{r}\right\rfloor=k$, so there are $n-rk=k\delta-(k-1)n$ pairwise disjoint independent sets of size $k+1$. These correspond to $k\delta-(k-1)n$ vertex-disjoint copies of $K_{k+1}$ in $G$.

Now consider $\delta\ge\frac{kn}{k+1}$. Apply Theorem \ref{thm:hajnal-szemeredi} to $\overline{G}$ with $r:=\left\lfloor\frac{n}{k+1}\right\rfloor>\Delta(\overline{G})$. Each part in the resultant partition has size $\left\lceil\frac{n}{r}\right\rceil\geq k+1$ or $\left\lfloor\frac{n}{r}\right\rfloor\geq k+1$, so there are $r=\left\lfloor\frac{n}{k+1}\right\rfloor$ pairwise disjoint independent sets of size at least $k+1$ in $\overline{G}$. These correspond to $\left\lfloor\frac{n}{k+1}\right\rfloor$ vertex-disjoint copies of $K_{k+1}$ in $G$.
\end{proof}

For our purposes the following corollary of Theorem~\ref{thm:komlos-sarkozy-szemeredi-seymour} will be useful.

\begin{corollary} \label{cor:komlos-sarkozy-szemeredi-seymour-path}
For every integer $k\in\NN$, there exists an integer $n_0=n_0(k)$ such that for all integers $n\geq n_0$, any graph $G$ on $n$ vertices with minimum degree at least $\frac{kn-1}{k+1}$ contains the $k$th power of a Hamiltonian path.
\end{corollary}

\begin{proof}
Fix an integer $k\in\NN$. Theorem~\ref{thm:komlos-sarkozy-szemeredi-seymour} produces an integer $n_0$. Let $G$ be a graph on $n\geq n_0$ vertices with minimum degree at least $\frac{kn-1}{k+1}$. Obtain a new graph $G^*$ by adding to $G$ a vertex adjacent to all other vertices. Note that $\delta(G^*)\geq\frac{k(n+1)}{k+1}$, so we can appeal to Theorem~\ref{thm:komlos-sarkozy-szemeredi-seymour} to find a copy of $C^k_{n+1}$ in $G^*$. Deleting the additional vertex from this copy of $C^k_{n+1}$ in $G^*$ yields the desired copy of $P^k_n$ in $G$.
\end{proof}

The following theorem of Andr\'{a}sfai, Erd\H{o}s and S\'{o}s gives a sufficient condition for a $K_k$-free graph to be in fact $(k-1)$-partite.

\begin{theorem}[Andr\'{a}sfai, Erd\H{o}s and S\'{o}s \cite{AndrasfaiErdosSos}] \label{thm:andrasfai-erdos-sos}
Let $k\geq3$ be an integer. A $K_k$-free graph $G$ on $n$ vertices with minimum degree $\delta(G)>\frac{3k-7}{3k-4}n$ is $(k-1)$-partite.
\end{theorem}

Denote by $d_G(u,v)$ the distance between vertices $u,v\in V(G)$ in a connected graph $G$. A connected graph $G$ on $n$ vertices is \emph{panconnected} if for each pair $u,v\in V(G)$ of vertices and each $d_G(u,v) < \ell \le n$ there is a path in $G$ with $\ell$ vertices which has $u$ and $v$ as endpoints. The following theorem of Williamson gives a sufficient minimum degree condition for a graph to be panconnected.

\begin{theorem}[Williamson~\cite{Williamson}] \label{thm:williamson}
Every graph $G$ on $n\ge4$ vertices with $\delta(G)\ge\tfrac{n}{2}+1$ is panconnected.
\end{theorem}

The following theorem of Erd\H{o}s and Stone gives a sufficient condition for a graph to contain $K_{t,t,t}$, the complete tripartite graph on three sets of vertices of size $t$.

\begin{theorem}[Erd\H{o}s and Stone~\cite{ErdosStone}] \label{thm:erdos-stone}
Given $t \in \NN$ and $\rho>0$, there exists $n_0=n_0(t,\rho)$ such that every graph on $n\geq n_0$ vertices with at least $\left(\frac{1}{2}+\rho\right)\binom{n}{2}$ edges contains a copy of $K_{t,t,t}$.
\end{theorem}

\subsection{Properties of some functions}

In this subsection, we collect some analytical data about the functions $r_p$, $r_c$, $\pp_k$ and $\pc_k$. Note that for fixed $k,n\in\NN$ the functions $r_p(k,n,\cdot)$ and $r_c(k,n,\cdot)$ are monotone decreasing while the functions $\pp_k(n,\cdot)$ and $\pc_k(n,\cdot)$ are monotone increasing. Note that the definition of $r:=r_p(k,n,\delta)$ in~\eqref{eqn:r-expression} is equivalent to $r=\left\lfloor\frac{(k-1)\delta-(k-2)n}{k\delta-(k-1)n+1}\right\rfloor$, from which we obtain
\begin{align}
\frac{n-\delta-1}{k\delta-(k-1)n+1}&<r\leq\frac{(k-1)\delta-(k-2)n}{k\delta-(k-1)n+1},\label{eqn:r-delta-ineq1} \\
\frac{[(k-1)r+1]n-(r+1)}{kr+1}&<\delta\leq\frac{[(k-1)(r-1)+1]n-r}{k(r-1)+1}. \label{eqn:r-delta-ineq2}
\end{align}

The following lemma gives bounds on the functions $r_p$, $\pp_k$ and $\pc_k$.

\begin{lemma} \label{lem:extr-fn-bounds}
Given an integer $k\ge3$ and $\mu>0$, there exists $\eta_0>0$ such that for every $0<\eta<\eta_0$ there exists $n_2\in\NN$ such that the following hold for all $n\ge n_2$. Let $r_0\in\NN$ satisfy $r_p(n,\gamma) \le r_0$ for all $\gamma\ge\left(\frac{k-1}{k}+\mu\right)n$. For $\delta\in\left[\left(\frac{k-1}{k}+\mu\right)n,\left(\frac{k}{k+1}-2\eta\right)n\right]$ we have
\begin{equation} \label{eqn:r'-bound}
r_p(k,n,\delta+\eta n) \ge 2,
\end{equation}
\begin{equation} \label{eqn:eta-bump-approx-bound}
\pp_k(n,\delta) \le \left(1-\tfrac{\eta}{10r_0}\right)\pp_k(n,\delta+\tfrac{\eta}{2}n),
\end{equation}
\begin{equation} \label{eqn:extr-approx-bound}
\pp_k(n,\delta+\eta n)\leq\frac{k+1}{2}\left(\frac{(k-1)(\delta+3\eta n)-(k-2)n}{r_p(k,n,\delta+\eta n)}-2\right),
\end{equation}
\begin{equation} \label{eqn:partite-intr-bound}
\delta-\frac{(k-1)(\delta+3\eta n)-(k-2)n}{r_p(k,n,\delta+\eta n)}>\frac{3k-4}{3}(n-\delta), \quad\text{and}
\end{equation}
\begin{equation} \label{eqn:spacious-bound}
\begin{split}
\pp_k(n,\delta+\eta n) & \le \frac{19}{20}(k+1)(k\delta-(k-1)n)-2 \\
& \le (k+1)(k\delta-(k-1)n)-10k^2\eta n.
\end{split}
\end{equation}
For $\delta'\in\left[\left(\frac{k-1}{k}+\mu\right)n,\left(\frac{2k-1}{2k+1}-2\eta\right)n\right]\cup\left[\left(\frac{2k-1}{2k+1}+\eta\right)n,\left(\frac{k}{k+1}-2\eta\right)n\right]=:A$ we have
\begin{equation} \label{eqn:three-quarters-bound}
\pp_k(n,\delta'+\eta n)\le\frac{3}{4}(k+1)(k\delta'-(k-1)n).
\end{equation}
For $\delta''\in\left[\left(\frac{k-1}{k}+\mu\right)n,\left(\frac{3k-2}{3k+1}-2\eta\right)n\right]\cup\left[\left(\frac{3k-2}{3k+1}+\eta\right)n,\left(\frac{2k-1}{2k+1}-2\eta\right)n\right]=:B$ we have
\begin{equation} \label{eqn:two-thirds-bound}
\pp_k(n,\delta''+\eta n)\leq
\frac{2}{3}(k+1)(k\delta''-(k-1)n).
\end{equation}
For $\delta'''\ge\left(\frac{k}{k+1}-2\eta\right)n$ we have
\begin{equation} \label{eqn:r'-bound-nearHAM}
r_p(k,n,\delta'''+\eta n) \le 2.
\end{equation}
For $\delta_1\in\left[\left(\frac{k-1}{k}+\mu\right)n,\frac{kn-1}{k+1}\right)$ we have
\begin{equation} \label{eqn:spacious-bound-threshold-hug}
\pp_k(n,\delta_1) \le \min\left\{(k+1)(k\delta_1-(k-1)n)-10k^2\eta n-(k+1),\tfrac{11n}{20}\right\}.
\end{equation}
For $\delta_2\in\left[\left(\frac{k-1}{k}+\mu\right)n,\frac{kn-1}{k+1}\right]$ we have
\begin{equation} \label{eqn:cycle-half-bound}
\pc_k(n,\delta_2) \le \tfrac{11n}{20}.
\end{equation}
\end{lemma}

\begin{proof}
Let $k\ge3$ be an integer and $\mu>0$. Pick $\eta_0=\frac{\mu}{200k^2}$. For $0 < \eta < \eta_0$, pick $n_2=\max\left\{\frac{10r_0}{\eta},100k\right\}$. Let $n \ge n_2$ be an integer.

Let $\delta\in\left[\left(\frac{k-1}{k}+\mu\right)n,\left(\frac{k}{k+1}-2\eta\right)n\right]$. Set $\delta_{+}:=\delta+\eta n$, $r:=r_p(n,\delta)$ and $r':=r_p(n,\delta_{+})$. If $r' = 1$, then by~\eqref{eqn:r-delta-ineq2} we have $\delta_{+} \ge \frac{kn-1}{k+1} > \left(\frac{k}{k+1}-\eta\right)n$. This gives a contradiction, so we have $r'\ge2$, i.e.\ \eqref{eqn:r'-bound}. By~\eqref{eqn:pp-pc-expression} we have
\begin{align*}
\pp_k(n,\delta) & \le \frac{k+1}{2}\left(\frac{(k-1)\delta-(k-2)n}{r}\right)+\frac{3k-1}{2} \\
& \le \left(1-\frac{\eta}{10r_0}\right)\left(\frac{k+1}{2}\left(\frac{(k-1)(\delta+\frac{\eta n}{2})-(k-2)n}{r'}\right)+\frac{k-1}{2}\right) \\
& \le \left(1-\frac{\eta}{10r_0}\right)\pp_k\left( n,\delta+\frac{\eta n}{2}\right),
\end{align*}
so we have~\eqref{eqn:eta-bump-approx-bound}. By~\eqref{eqn:pp-pc-expression} we have
\begin{align*}
\pp_k(n,\delta_{+}) & \le \frac{k+1}{2}\left(\frac{(k-1)\delta_{+}-(k-2)n}{r'}\right)+\frac{3k-1}{2} \\
& \le \frac{k+1}{2}\left(\frac{(k-1)(\delta+3\eta n)-(k-2)n}{r'}-2\right),
\end{align*}
so we have~\eqref{eqn:extr-approx-bound}. By~\eqref{eqn:r-expression}, for some $q\in[r',r'+1]$ we have
\[(k-1)\delta_{+}-(k-2)n = q(k\delta_{+}-(k-1)n+1),\]
so we have
\[\frac{(k-1)(\delta+3\eta n)-(k-2)n}{r'} = (k\delta_{+}-(k-1)n+1)\frac{q}{r'}+\frac{(k-1)\eta n}{r'}.\]
Since $n-\delta-1 \ge (q-1)(k\delta-(k-1)n+1)$ and $r'\ge2$, we have~\eqref{eqn:partite-intr-bound}.

By~\eqref{eqn:r-expression} we have 
\begin{equation} \label{eqn:extr-upper-bound}
\frac{(k-1)\delta_{+}-(k-2)n}{r'+1} < k\delta_{+}-(k-1)n+1.
\end{equation}
Hence, by~\eqref{eqn:pp-pc-expression} we have $\pp_k(n,\delta_{+}) \le \frac{k+1}{2}(k\delta_{+}-(k-1)n+1)\frac{r'+1}{r'}+\frac{3k-1}{2}$. Since $r'\ge2$, we have~\eqref{eqn:spacious-bound} because
\begin{align*}
\pp_k(n,\delta_{+}) & \le \frac{3(k+1)}{4}(k\delta_{+}-(k-1)n+1)+\frac{3k-1}{2} \\
& \le \frac{19}{20}(k+1)(k\delta-(k-1)n)-2 \\
& \le (k+1)(k\delta-(k-1)n)-10k^2\eta n.
\end{align*}

Let $\delta'\in A$. By~\eqref{eqn:r-delta-ineq2} we have $r'\ge2$. If $r'\ge3$, then we have 
\begin{align*}
\pp_k(n,\delta'_{+}) & \le \frac{2(k+1)}{3}(k\delta'_{+}-(k-1)n+1)+\frac{3k-1}{2} \\
& \le \frac{3}{4}(k+1)(k\delta'-(k-1)n),
\end{align*}
which gives~\eqref{eqn:three-quarters-bound}. If $r'=2$, we have $\delta'\in\left[\left(\frac{2k-1}{2k+1}+\eta\right)n,\left(\frac{k}{k+1}-2\eta\right)n\right]=:A'$. By considering~\eqref{eqn:extr-upper-bound} for $\delta_{+}=\frac{(2k-1)n}{2k+1}$ and the coefficient of $\delta_{+}$ on both sides of the inequality, for $\delta'\in A'$ we can strengthen the inequality to
\[\frac{(k-1)\delta'_{+}-(k-2)n}{3} < k\delta'-(k-1)n-4.\]
From this, we obtain $\pp_k(n,\delta'+\eta n)\le\frac{3}{4}(k+1)(k\delta'-(k-1)n)$ by an argument analogous to that for~\eqref{eqn:spacious-bound}. For $\delta''\in B$ we obtain~\eqref{eqn:two-thirds-bound} by an argument analogous to that for $\delta'\in A$. Let $\delta'''\ge\left(\frac{k}{k+1}-2\eta\right)n$. If $r' \ge 3$, then by~\eqref{eqn:r-delta-ineq2} we have $\delta'''_{+} \le \frac{(2k-1)n-3}{2k+1} < \left(\frac{k}{k+1}-\eta\right)n$. This gives a contradiction, so we have $r' \le 2$.

Let $\delta_1\in\left[\left(\frac{k-1}{k}+\mu\right)n,\frac{kn-1}{k+1}\right)$. Since $\pp_k(n,\cdot)$ is monotone increasing, we have $\pp_k(n,\delta_1) \le \pp_k(n,\frac{kn-2}{k+1}) \le \frac{n}{2}+k \le \frac{11n}{20}$. By~\eqref{eqn:r-delta-ineq2} we have $r \ge 2$ and by~\eqref{eqn:r-expression} we have $\frac{(k-1)\delta_1-(k-2)n}{r+1} < k\delta_1-(k-1)n+1$. Then, by~\eqref{eqn:pp-pc-expression} we have
\begin{align*}
\pp_k(n,\delta_1) & \le \frac{3(k+1)}{4}(k\delta_1-(k-1)n+1)+\frac{3k-1}{2} \\
& \le (k+1)(k\delta_1-(k-1)n)-10k^2\eta n-(k+1),
\end{align*}
so we obtain~\eqref{eqn:spacious-bound-threshold-hug}. Let $\delta_2\in\left[\left(\frac{k-1}{k}+\mu\right)n,\frac{kn-1}{k+1}\right]$. Since $\pc_k(n,\cdot)$ is monotone increasing, we have $\pc_k(n,\delta_2) \le \pc_k(n,\frac{kn-1}{k+1}) \le \frac{n}{2}+k \le \frac{11n}{20}$, which gives~\eqref{eqn:cycle-half-bound}.
\end{proof}

\section{Main lemmas and proof of Theorem~\ref{thm:cycle-path-power}} \label{section:lem-thmproof}

Our proof of Theorem~\ref{thm:cycle-path-power} uses the well-established technique that combines the regularity method, which involves the joint application of Szemer\'edi's regularity lemma~\cite{Szemeredi} and the blow-up lemma of Koml\'os, S\'ark\"ozy and Szemer\'edi~\cite{KomlosSarkozySzemeredi-blowup}, and the stability method. However, this provides only a loose framework for proofs of this kind. In particular, the proof of our stability lemma is significantly more involved than in~\cite{AllenBoettcherHladky} and is the main contribution of this paper. For our application we will define the concept of a connected $K_{k+1}$-component of a graph, which generalises the concept of a connected triangle component of a graph introduced by Allen, B\"ottcher and Hladk\'y~\cite{AllenBoettcherHladky}.

In this section we explain how we utilise connected $K_{k+1}$-components, the regularity method and the stability method. We first introduce the necessary definitions and formulate our main lemmas. Then, we detail how these lemmas imply Theorem~\ref{thm:cycle-path-power} at the end of this section.

\subsection{Connected $K_{k+1}$-components and $K_{k+1}$-factors} \label{section:component-factor}

Fix $k\in\NN$ and let $G$ be a graph. A {\em $K_{k+1}$-walk} is a sequence $t_1,\dots, t_p$ of copies of $K_k$ in $G$ such that for every $i\in[p-1]$ there is a copy $c_i$ of $K_{k+1}$ in $G$ such that $t_i,t_{i+1}\subseteq c_i$. We say that $t_1$ and $t_p$ are {\em $K_{k+1}$-connected} in $G$. A {\em $K_{k+1}$-component} of $G$ is a maximal set $C$ of copies of $K_k$ in $G$ such that every pair of copies of $K_k$ in $C$ is $K_{k+1}$-connected. Observe that this induces an equivalence relation on the copies of $K_k$ of $G$ whose equivalence classes are the $K_{k+1}$-components of $G$. The {\em vertices of a $K_{k+1}$-component} $C$ are all vertices $v$ of $G$ such that $v$ is a vertex of a copy of $K_k$ in $C$. The {\em size of a $K_{k+1}$-component} $C$ is the number of vertices of $C$, which we denote by $|C|$.

A {\em $K_{k+1}$-factor} $F$ in $G$ is a collection of vertex-disjoint copies of $K_{k+1}$ in $G$. $F$ is a {\em connected $K_{k+1}$-factor} if all copies of $K_k$ of $F$ are contained in the same $K_{k+1}$-component of $G$. The \emph{size} of a $K_{k+1}$-factor $F$ is the number of vertices covered by $F$. Let $\ckf_{k+1}(G)$ denote the maximum size of a connected $K_{k+1}$-factor in $G$. For $\ell\in[k]$ the {\em copies of $K_{\ell}$ of a $K_{k+1}$-component} $C$ are all copies of $K_{\ell}$ of $G$ which can be extended to a copy of $K_k$ in $C$. For $\ell>k$ the {\em copies of $K_{\ell}$ of a $K_{k+1}$-component} $C$ are all copies of $K_{\ell}$ of $G$ to which a copy of $K_k$ in $C$ can be extended. In a (slight) abuse of notation, we shall write $K_{\ell}\subseteq C$ to mean that there is such a copy of $K_{\ell}$.

\subsection{Regularity method} \label{subsection:regularity}

Our proof uses a combination of Szemer\'edi's regularity lemma~\cite{Szemeredi} and the blow-up lemma by Koml\'os, S\'ark\"ozy and Szemer\'edi~\cite{KomlosSarkozySzemeredi-blowup}. Generally, the regularity lemma produces a partition of a dense graph that is suitable for an application of the blow-up lemma, which enables us to embed a target graph in a large host graph. We first introduce some terminology to formulate the versions of these two lemmas we will use.

Let $G=(V,E)$ be a graph and $d,\eps\in(0,1]$. For disjoint nonempty $U,W\subseteq V$ the {\em density} of the pair $(U,W)$ is $d(U,W):=\frac{e(U,W)}{|U||W|}$. A pair $(U,W)$ is {\em $\eps$-regular} if $|d(U',W')-d(U,W)|\leq\eps$ for all $U'\subseteq U$ and $W'\subseteq W$ with $|U'|\geq\eps|U|$ and $|W'|\geq\eps|W|$. An {\em $\eps$-regular partition} of $G$ is a partition $V_0\dot{\cup}V_1\dot{\cup}\dots\dot{\cup}V_\ell$ of $V$ with $|V_0|\leq\eps|V|$, $|V_i|=|V_j|$ for all $i,j\in[\ell]$, and such that for all but at most $\eps k^2$ pairs $(i,j)\in[\ell]^2$, the pair $(V_i,V_j)$ is $\eps$-regular.

Given $d\in(0,1)$, a pair of disjoint vertex sets $(V_i,V_j)$ in a graph $G$ is {\em $(\eps,d)$-regular} if it is $\eps$-regular and has density at least $d$. An {\em $\eps$-regular partition} $V_0\dot{\cup}V_1\dot{\cup}\dots\dot{\cup}V_\ell$ of a graph $G$ is an $(\eps,d)$-regular partition if the following is true. For every $i\in[k]$ and every vertex $v\in V_i$, there are at most $(d+\eps)n$ edges incident to $v$ which are not contained in $(\eps,d)$-regular pairs of the partition. Given an $(\eps,d)$-partition $V_0\dot{\cup}V_1\dot{\cup}\dots\dot{\cup}V_\ell$ of a graph $G$, we define a graph $R$, which we call the {\em reduced graph} of the partition of $G$, where $R=(V(R),E(R))$ has $V(R)=\{V_1,\dots,V_\ell\}$ and $V_iV_j\in E(R)$ whenever $(V_i,V_j)$ is an $(\eps,d)$-regular pair. We say that $G$ has {\em $(\eps,d)$-reduced graph} $R$ and call the partition classes $V_i$ with $i\in[\ell]$ {\em clusters} of $G$.

The classical Szemer\'edi regularity lemma~\cite{Szemeredi} states that every large graph has an $\eps$-regular partition with a bounded number of parts. Here we state the so-called minimum degree form of Szemer\'edi's regularity lemma (see, e.g., Lemma 7 in conjunction with Proposition 9 in~\cite{KuehnOsthusTaraz}).

\begin{lemma}[Regularity Lemma, minimum degree form] \label{lem:regularity-min-deg}
Given $\eps\in(0,1)$ and $m_0\in\NN$, there is an integer $m_1\in\NN$ such that the following holds for all $d,\gamma\in(0,1)$ such that $\gamma\ge2d+4\eps$. Every graph $G$ on $n\ge m_1$ vertices with $\delta(G)\ge\gamma n$ has an $(\eps,d)$-reduced graph $R$ on $m$ vertices with $m_0\le m\le m_1$ and $\delta(R)\ge(\gamma-d-2\eps)m$.
\end{lemma}

This lemma asserts the existence of a reduced graph $R$ of $G$ which `inherits' the high minimum degree of $G$. We shall use this property to reduce our original problem of finding the $k$th power of a path (or cycle) in a graph on $n$ vertices with minimum degree $\gamma n$ to the problem of finding an {\em arbitrary} connected $K_{k+1}$-factor of a desired size in a graph $R$ on $m$ vertices with minimum degree $(\gamma-d-\eps)n$ (see Lemma~\ref{lem:embedding}). The new problem seeks a much less specific subgraph (connected $K_{k+1}$-factor) than the original problem and is therefore easier to tackle.

This kind of problem reduction is possible due to the Blow-up Lemma, which enables the embedding of a large bounded degree target graph $H$ into a graph $G$ with reduced graph $R$ if there is a homomorphism from $H$ to a subgraph $T$ of $R$ which does not `overuse' any cluster of $T$. For our purposes we apply this lemma with $T=K_{k+1}$ and obtain for each copy of $K_{k+1}$ in a connected $K_{k+1}$-factor $F$ the $k$th power of a path which almost fills up the corresponding clusters of $G$. The $K_{k+1}$-connectedness of $F$ then enables us to link up these $k$th power of paths and obtain $k$th powers of paths or cycles of the desired length (see Lemma~\ref{lem:embedding}~\ref{item:embed-divisible} and~\ref{item:embed-(k+2)-clique}). In addition, the Blow-up Lemma allows for some control of the start-vertices and end-vertices of the $k$th power of a path constructed in this manner (see Lemma~\ref{lem:embedding}~\ref{item:embed-path-fixed-ends}).

The following lemma sums up what we obtain from this embedding strategy. This is an application of standard methods and we will provide its proof in Appendix~\ref{section:embedding}.

\begin{lemma}[Embedding Lemma] \label{lem:embedding}
For any integer $k\geq2$ and any $d>0$ there exists $\eps_{EL}>0$ with the following property. Given $0<\eps<\eps_{EL}$, for every $m_{EL}\in\NN$ there exists $n_{EL}\in\NN$ such that the following hold for any graph $G$ on $n\geq n_{EL}$ vertices with $(\eps,d)$-reduced graph $R$ of $G$ on $m\leq m_{EL}$ vertices.
\begin{enumerate}[label=(\roman*)]
\item \label{item:embed-divisible} $C_{(k+1)\ell}^k\subseteq G$ for every $\ell\in\NN$ with $(k+1)\ell\leq(1-d)\ckf_{k+1}(R)\frac{n}{m}$.
\item \label{item:embed-(k+2)-clique} If $K_{k+2}\subseteq C$ for each $K_{k+1}$-component $C$ of $R$, then $C_\ell^k\subseteq G$ for every $\ell\in\NN$ with $k+1\leq\ell\leq(1-d)\ckf_{k+1}(R)\frac{n}{m}$ and $\chi(C_{\ell}^k)\leq k+2$.
\end{enumerate}
Furthermore, let $\mathcal{T}'$ be a connected $K_{k+1}$-factor in a $K_{k+1}$-component $C$ of $R$ which contains a copy of $K_{k+2}$, let $u_{1,1}\dots u_{1,k}$ and $u_{2,1}\dots u_{2,k}$ be vertex-disjoint copies of $K_k$ in $G$, and suppose that $X_{1,1}\dots X_{1,k}$ and $X_{2,1}\dots X_{2,k}$ are (not necessarily disjoint) copies of $K_k$ in $C$ in $R$ such that $u_{i,j}\dots u_{i,k}$ has at least $\frac{2dn}{m}$ common neighbours in the cluster $X_{i,j}$ for each $(i,j)\in[2]\times[k]$. Let $A$ be a set of at most $\frac{\eps n}{m}$ vertices of $G$ disjoint from $\{u_{1,1},\dots,u_{1,k},u_{2,1},\dots,u_{2,k}\}$. Then 
\begin{enumerate}[resume,label=(\roman*)]
\item \label{item:embed-path-fixed-ends} $P_\ell^k\subseteq G$ for every $\ell\in\NN$ with $3m^{k+1}\leq\ell\leq(1-d)|\mathcal{T}'|\frac{n}{m}$, such that $P_\ell^k$ starts in $u_{1,1}\dots u_{1,k}$ and ends in $u_{2,k}\dots u_{2,1}$ (in those orders), $P_\ell^k$ contains no element of $A$, and at most $\frac{\eps n}{m}$ vertices of $P_\ell^k$ are not in $\bigcup\mathcal{T}'$.
\end{enumerate}
\end{lemma}

Note that the copies of $K_{k+2}$ required in~\ref{item:embed-(k+2)-clique} are essential to the embedding $k$th powers of cycles which not $(k+1)$-chromatic.

\subsection{Stability method} \label{subsection:stability}

The regularity method as just described leaves us with the task of finding a sufficiently large connected $K_{k+1}$-factor $F$ in a reduced graph $R$ of $G$. However, this is insufficient on its own. The Embedding Lemma (Lemma~\ref{lem:embedding}) gives the $k$th power of a path which covers a proportion of $G$ roughly the same as the proportion $\lambda$ of $R$ covered by $F$. Furthermore, the extremal graphs for $k$th powers of paths and connected $K_{k+1}$-factors are the same, but the relative minimum degree $\gamma_R=\delta(R)/|V(R)|$ of $R$ is in general slightly smaller than $\gamma_G=\delta(G)/|V(G)|$. Consequently we cannot expect that $\lambda$ is larger than the proportion $\pp_k\big(v(R),\gamma_Rv(R)\big)/v(R)$ a maximum $k$th power of a path covers in a graph with relative minimum degree $\gamma_R$, and hence $\lambda$ is smaller than the proportion $\pp_k\big(v(G),\gamma_Gv(G)\big)/v(G)$ we would like to cover for relative minimum degree $\gamma_G$. This is where our stability approach comes into the picture.

Roughly speaking, we will be more ambitious and aim for a connected $K_{k+1}$-factor in $R$ larger than guaranteed by the relative minimum degree (see Lemma~\ref{lem:stability}~\ref{item:stability-large-factor} and~\ref{item:stability-target-factor}). We prove that we will fail to find this larger connected $K_{k+1}$-factor only if $R$ (and hence $G$) is `very close' to the extremal graph $G_p(k,v(R),\delta(R))$, in which case we will say that $R$ is {\em near-extremal} (see Lemma~\ref{lem:stability}~\ref{item:stability-extremal}). The following lemma, which we call Stability Lemma and prove in Section~\ref{section:stability}, does precisely this. Note that this lemma guarantees the copies of $K_{k+2}$ required in the Embedding Lemma (Lemma~\ref{lem:embedding}). We remark that the proof of Stability Lemma is our main contribution as it is significantly more involved than in~\cite{AllenBoettcherHladky} and the analysis requires new insights.

\begin{lemma}[Stability Lemma] \label{lem:stability}
Given an integer $k\geq3$ and $\mu>0$, for any sufficiently small $\eta>0$ there exists an integer $m_0$ such that if $\delta\in\left[(\frac{k-1}{k}+\mu)n,\frac{kn}{k+1}\right]$ and $G$ is a graph on $n\geq m_0$ vertices with minimum degree $\delta(G)\geq\delta$, then either 
\begin{enumerate}[label=(C\arabic{*})]
\item \label{item:stability-large-factor} $\ckf_{k+1}(G)\geq(k+1)(k\delta-(k-1)n)$, or
\item \label{item:stability-target-factor} $\ckf_{k+1}(G)\geq\pp_k(n,\delta+\eta n)$, or
\item \label{item:stability-extremal} $G$ has $k-1$ vertex-disjoint independent sets of combined size at least $(k-1)(n-\delta)-5k\eta n$ whose removal disconnects $G$ into components which are each of size at most $\frac{19}{10}(k\delta-(k-1)n)$ and for each such component $X$ all copies of $K_k$ in $G$ containing at least one vertex of $X$ are $K_{k+1}$-connected in $G$.
\end{enumerate}
Moreover, in~\ref{item:stability-target-factor} and~\ref{item:stability-extremal} each $K_{k+1}$-component of $G$ contains a copy of $K_{k+2}$.
\end{lemma}

Note that~\ref{item:stability-large-factor} gives a connected $K_{k+1}$-factor whose size is significantly larger than $\pp_k(n,\delta)$, which is the maximum size we can guarantee in general (see Figure~\ref{fig:cube-path} for an illustration in the case $k=3$). We also remark that since $\pp_k(n,\delta)$ is a function with relatively large jumps at certain points, around these points~\ref{item:stability-target-factor} gives a connected $K_{k+1}$-factor whose size is significantly larger than $\pp_k(n,\delta)$. Furthermore, we clarify that the `components' in~\ref{item:stability-extremal} refer to the usual connected components of a graph rather than $K_{k+1}$-components. It is notable that the two functions $\pp_k(n,\delta)$ and $\pc_k(n,\delta)$ are sufficiently similar that Stability Lemma handles both. We will only need to distinguish between the $k$th powers of paths and the $k$th powers of cycles when we consider near-extremal graphs.

It remains to handle graphs with near-extremal reduced graphs. We have a great deal of structural information about these graphs, which we use to directly find the desired $k$th powers of paths and cycles. The following lemma, which we call Extremal Lemma, handles the near-extremal case. We will provide a proof of this lemma in Section~\ref{section:extremal}. Note that in this proof we will make use of our Embedding Lemma (Lemma~\ref{lem:embedding}). Accordingly Lemma~\ref{lem:extremal} inherits the upper bound $m_{EL}$ on the number of clusters from Lemma~\ref{lem:embedding}.

\begin{lemma}[Extremal Lemma] \label{lem:extremal}
Given an integer $k\ge3$, $0<\nu<1$ and $0<\eta,d\le\frac{\nu^4}{(k+1)^{13}10^8}$, there exists $\eps_0>0$ such that for every $0<\eps<\eps_0$ and every integer $m_{EL}$, there is an integer $N$ such that the following holds. Let $G$ be a graph on $n\ge N$ vertices with $\delta(G)\ge\delta\in\left[\left(\frac{k-1}{k}+\nu\right)n,\frac{kn}{k+1}\right)$ and $R$ be an $(\eps,d)$-reduced graph of $G$ on $m\le m_{EL}$ vertices with a partition $\left(\bigsqcup_{i=1}^{k-1}I_i\right)\sqcup\left(\bigsqcup_{j=1}^{\ell}B_j\right)$ of $V(R)$ with $\ell\ge2$. Suppose that
\begin{enumerate}[label=(\roman*)]
\item \label{item:extr-lem-component-clique} each $K_{k+1}$-component of $R$ contains a copy of $K_{k+2}$,
\item \label{item:extr-lem-int} $I_1,\dots,I_{k-1}$ are independent sets in $R$ with $\left|\bigcup_{i=1}^{k-1}I_i\right|\geq((k-1)(n-\delta)-5k\eta n)\frac{m}{n}$,
\item \label{item:extr-lem-ext} for each $i\in[\ell]$ we have $0<|B_i|\leq\frac{19m}{10n}(k\delta-(k-1)n)$, all copies of $K_k$ in $R$ containing at least one vertex of $B_i$ are $K_{k+1}$-connected in $R$, and for $j\in[\ell]\setminus\{i\}$ there are no edges between $B_i$ and $B_j$ in $R$.
\end{enumerate}
Then $G$ contains $P^k_{\pp_k(n,\delta)}$ and $C^k_\ell$ for each $\ell\in[k+1,\pc_k(n,\delta)]$ such that $\chi(C^k_\ell)\leq k+2$.
\end{lemma}

We now have all the ingredients for the proof of our main theorem. The Regularity Lemma (Lemma~\ref{lem:regularity-min-deg}) provides a regular partition with reduced graph $R$ of the host graph $G$ and Stability Lemma (Lemma~\ref{lem:stability}) tells us that $R$ either contains a large connected $K_{k+1}$-factor or is near-extremal. We find the $k$th powers of long paths and cycles in $G$ by applying Embedding Lemma (Lemma~\ref{lem:embedding}) in the first case and Extremal Lemma (Lemma~\ref{lem:extremal}) in the second case.

\begin{proof}[Proof of Theorem~\ref{thm:cycle-path-power}]
We first set up the necessary constants. Let $k\geq3$ and $0<\nu<1$. Set $\mu:=\left(1-\frac{2}{2k^2+2k+1}\right)\nu$ and choose $\eta>0$ to be small enough for Lemmas~\ref{lem:extr-fn-bounds}, \ref{lem:stability} and~\ref{lem:extremal}. In particular, $\eta\leq\frac{\nu^4}{(k+1)^{13}10^8}$. Given $k,\mu,\eta$ from above as input, Lemmas~\ref{lem:extr-fn-bounds} and~\ref{lem:stability} produce positive integers $m'_0$ and $m_2$ respectively. Let $r_0\in\NN$ satisfy $r_p(n,\gamma) \le r_0$ for all $\gamma\ge\left(\frac{k-1}{k}+\mu\right)n$. Set $d:=\frac{\eta}{5r_0}$ and $m_0:=\max\{m'_0,m_2,d^{-1}\}$. With $k$ and $d$ as input, Lemma~\ref{lem:embedding} then produces $\eps_{EL}>0$. For $\nu,\eta$ and $d$, Lemma~\ref{lem:extremal} produces $\eps_0>0$. Set $\eps:=\frac{1}{2}\min\{\eps_{EL},\eps_0,\frac{\eta}{5r_0}\}$ and choose an integer $m_{EL}$ such that Lemma~\ref{lem:regularity-min-deg} guarantees the existence of an $(\eps,d)$-regular partition with at least $m_0$ and at most $m_{EL}$ parts. With the constants $\nu,\eta,d,\eps,m_{EL}$ from above as input, Lemma~\ref{lem:embedding} and Lemma~\ref{lem:extremal} produce $n_{EL}$ and $N$ respectively. Corollary~\ref{cor:komlos-sarkozy-szemeredi-seymour-path} returns $n_{HP}$ with $k$ as input. Finally, set $n_0:=\max\{n_{HP},m_{EL},n_{EL},N\}$.

Let $\delta\in\left[\left(\frac{k-1}{k}+\nu\right)n,\frac{kn}{k+1}\right)$ and let $G$ be a graph on $n\geq n_0$ vertices with minimum degree $\delta(G)\ge\delta$. As remarked after Theorem~\ref{thm:cycle-path-power}, observe that it suffices to show that $P^k_{\pp_k(n,\delta)}\subseteq G$ and that~\ref{item:cycle-power-parity cases} of Theorem~\ref{thm:cycle-path-power} holds. Furthermore, we will need to treat the case $\delta=\frac{kn-1}{k+1}$ separately from the rest because in this case $P^k_{\pp_k(n,\delta)}$ is actually Hamiltonian.

Let us consider the case when $\delta\in\left[\left(\frac{k-1}{k}+\nu\right)n,\frac{kn-1}{k+1}\right)$. We first apply Lemma~\ref{lem:regularity-min-deg} to $G$ to obtain an $(\eps,d)$-reduced graph $R$ on $m_0\leq m\leq m_{EL}$ vertices with $\delta(R)\geq\delta':=\left(\frac{\delta}{n}-d-2\eps\right)m$. Note that $\delta'\in\left[\left(\frac{k-1}{k}+\mu\right)m,\frac{km}{k+1}\right]$. Then we apply Lemma~\ref{lem:stability} to $R$. According to this lemma, there are three possibilities.

Firstly, we could have $\ckf_{k+1}(R)\ge(k+1)(k\delta'-(k-1)m)$. Now Lemma~\ref{lem:embedding}\ref{item:embed-divisible} guarantees that $G$ contains $C^k_{\ell}$ for each positive integer $\ell\le(1-d)\ckf_{k+1}(R)\frac{n}{m}$ divisible by $k+1$; since by choice of $d$ and $\eps$ we have $(1-d)(k+1)(k\delta'-(k-1)m)\frac{n}{m}\ge(k+1)(k\delta-(k-1)n)-10k^2\eta n$, $G$ contains $C^k_\ell$ for each positive integer $\ell\le(k+1)(k\delta-(k-1)n)-\nu n$ divisible by $k+1$, i.e.\ the second case of Theorem~\ref{thm:cycle-path-power}\ref{item:cycle-power-parity cases} holds. Since $P^k_\ell\subseteq C^k_\ell$ and we have~\eqref{eqn:spacious-bound-threshold-hug}, it follows that $G$ contains $P^k_{\pp_k(n,\delta)}$.

Secondly, we could have that $\ckf_{k+1}(R)\ge\pp_k(m,\delta'+\eta m)$ and every $K_{k+1}$-component of $R$ contains a copy of $K_{k+2}$. By Lemma~\ref{lem:embedding}\ref{item:embed-(k+2)-clique} we have that $G$ contains $C^k_\ell$ for each integer $k+1\le\ell\le(1-d)\ckf_{k+1}(R)\frac{n}{m}$ such that $\chi(C^k_\ell)\le k+2$; since by~\eqref{eqn:eta-bump-approx-bound}, $P^k_\ell\subseteq C^k_\ell$ and choice of $\eta,d$ and $\eps$ we have $(1-d)\ckf_{k+1}(R)\frac{n}{m}\ge(1-d)\pp_k(n,\delta+\tfrac{\eta n}{2})\ge\pp_k(n,\delta)\ge\pc_k(n,\delta)$, we conclude that $G$ contains $P^k_{\pp_k(n,\delta)}$ and $C^k_\ell$ for each integer $k+1\le\ell\le\pc_k(n,\delta)$ such that $\chi(C^k_\ell)\le k+2$, i.e.\ the first case of Theorem~\ref{thm:cycle-path-power}\ref{item:cycle-power-parity cases} holds.

Thirdly, we could have that $R$ is near-extremal. In this case each $K_{k+1}$-component of $R$ contains a copy of $K_{k+2}$ and $R$ contains $k-1$ vertex-disjoint independent sets of combined size at least $(k-1)(m-\delta')-5k\eta m\ge((k-1)(n-\delta)-5k\eta n)\frac{m}{n}$ whose removal disconnects $R$ into components which are each of size at most $\frac{19}{10}(k\delta'-(k-1)m)\le\frac{19m}{10n}(k\delta-(k-1)n)$ and for each such component $X$ all copies of $K_k$ in $R$ containing at least one vertex of $X$ are $K_{k+1}$-connected in $R$. But now $G$ and $R$ satisfy the conditions of Lemma~\ref{lem:extremal}, so it follows that $G$ contains $P^k_{\pp_k(n,\delta)}$ and $C^k_\ell$ for each integer $k+1\le\ell\le\pc_k(n,\delta)$ such that $\chi(C^k_\ell)\le k+2$, i.e.\ the first case of Theorem~\ref{thm:cycle-path-power}\ref{item:cycle-power-parity cases} holds.

Now it remains to deal with the special case $\delta = \frac{kn-1}{k+1}$. As with the main case, we apply Lemma~\ref{lem:regularity-min-deg} to $G$ to obtain a reduced graph $R$, apply Lemma~\ref{lem:stability} to $R$ with three possible outcomes and then apply Lemma~\ref{lem:embedding}\ref{item:embed-divisible}, Lemma~\ref{lem:embedding}\ref{item:embed-(k+2)-clique} and Lemma~\ref{lem:extremal} in the first, second and third cases respectively to obtain $k$th powers of cycles of the appropriate lengths. Finally, by Corollary~\ref{cor:komlos-sarkozy-szemeredi-seymour-path} $G$ contains a copy of $P^k_n = P^k_{\pp_k(n,\delta)}$.
\end{proof}

\section{Proving our stability lemma} \label{section:stability}

In this section we provide a proof of our stability lemma for connected $K_{k+1}$-factors, Lemma \ref{lem:stability}. We divide the proof of Lemma~\ref{lem:stability} into three lemmas, which correspond to three different cases as follows.
\begin{enumerate}[label=(\arabic{*})]
    \item $G$ has just one $K_{k+1}$-component (see Lemma~\ref{lem:single-cmpnt}),
    \item $G$ has a $K_{k+1}$-component $C$ which does not contain a copy of $K_{k+2}$ (see Lemma~\ref{lem:large-clique-free-cmpnt}),
    \item $G$ has at least two $K_{k+1}$-components and each $K_{k+1}$-component contains a copy of $K_{k+2}$ (see Lemma~\ref{lem:general-cmpnt}).
\end{enumerate}
In the first case, the result follows from an application of a classical result of Hajnal and Szemer\'{e}di~\cite{HajnalSzemeredi} in the form of Lemma~\ref{lem:single-cmpnt}. In the second case, the result follows from an inductive argument in the form of Lemma~\ref{lem:large-clique-free-cmpnt}. Finally, we handle the third case in the form of Lemma~\ref{lem:general-cmpnt}. This turns out to be the main work and we will provide a sketch of its proof at the end of this section.

We now state Lemmas~\ref{lem:single-cmpnt},~\ref{lem:large-clique-free-cmpnt} and~\ref{lem:general-cmpnt}, and provide a proof of Lemma~\ref{lem:stability} applying these lemmas. We will provide proofs of Lemmas~\ref{lem:single-cmpnt} and~\ref{lem:large-clique-free-cmpnt} right after our proof of Lemma~\ref{lem:stability}. Finally, we will introduce a family of configurations in Section~\ref{subsection:configs} to prepare for the substantially more involved proof of Lemma~\ref{lem:general-cmpnt}. We will analyse this family of configurations and develop greedy-type methods for the construction of connected $K_{k+1}$-factors in Section~\ref{section:structure-methods}. These will be applied in the proof of Lemma~\ref{lem:general-cmpnt}, which will be provided in Section~\ref{section:stability-general-cmpnt}.

\begin{lemma} \label{lem:single-cmpnt}
Let $k\in\NN$ and $\delta\in\left[\frac{(k-1)n}{k},\frac{kn}{k+1}\right]$. Let $G$ be a graph on $n\geq k(k+1)$ vertices with minimum degree $\delta(G)\geq\delta$ and exactly one $K_{k+1}$-component. Then $\ckf_{k+1}(G)\geq(k+1)(k\delta-(k-1)n)$.
\end{lemma}

\begin{lemma} \label{lem:large-clique-free-cmpnt}
Let $k\in\NN$. Let $G$ be a graph on $n$ vertices with minimum degree $\delta(G)\geq\delta\geq\frac{(k-1)n}{k}$. Suppose that $G$ has a $K_{k+1}$-component $C$ which does not contain a copy of $K_{k+2}$. Then there is a set of $k\delta-(k-1)n$ vertex-disjoint copies of $K_{k+1}$ which are all in $C$.
\end{lemma}

\begin{lemma} \label{lem:general-cmpnt}
Given an integer $k\geq3$ and $\mu>0$, for any sufficiently small $\eta>0$ there exists an integer $m_1$ such that if $\delta\geq(\frac{k-1}{k}+\mu)n$ and $G$ is a graph on $n\geq m_1$ vertices with minimum degree $\delta(G)\geq\delta$ such that $G$ has at least two $K_{k+1}$-components and every $K_{k+1}$-component of $G$ contains a copy of $K_{k+2}$, then either
\begin{enumerate}[label=(D\arabic{*})]
\item \label{item:sub-stability-target-factor} $\ckf_{k+1}(G)\geq\pp_k(n,\delta+\eta n)$, or
\item \label{item:sub-stability-extremal} $G$ has $k-1$ vertex-disjoint independent sets of combined size at least $(k-1)(n-\delta)-5k\eta n$ whose removal disconnects $G$ into components which are each of size at most $\frac{19}{10}(k\delta-(k-1)n)$ and for each component $X$ all copies of $K_k$ in $G$ containing at least one vertex of $X$ are $K_{k+1}$-connected in $G$.
\end{enumerate}
\end{lemma}

\begin{proof}[Proof of Lemma~\ref{lem:stability}]
Given an integer $k\geq3$, $\mu>0$ and any $\eta>0$ sufficiently small for application of Lemma~\ref{lem:general-cmpnt}, Lemma~\ref{lem:general-cmpnt} produces an integer $m_1$. Set $m_0:=\max\{m_1,k(k+1)\}$. Let $\delta\in\left[(\frac{k-1}{k}+\mu)n,\frac{kn}{k+1}\right]$ and let $G$ be a graph on $n\geq m_0$ vertices with minimum degree $\delta(G)\geq\delta$.

If $G$ has only one $K_{k+1}$-component then Lemma~\ref{lem:single-cmpnt} implies $\ckf_{k+1}(G)\geq(k+1)(k\delta-(k-1)n)$. If $G$ has a $K_{k+1}$-component $C$ which does not contain a copy of $K_{k+2}$ then Lemma~\ref{lem:large-clique-free-cmpnt} implies $\ckf_{k+1}(G)\geq(k+1)(k\delta-(k-1)n)$. In both cases we are in~\ref{item:stability-large-factor}. If $G$ has at least two $K_{k+1}$-components and every $K_{k+1}$-component of $G$ contains a copy of $K_{k+2}$, then Lemma~\ref{lem:general-cmpnt} implies that we are in~\ref{item:stability-target-factor} or~\ref{item:stability-extremal}.
\end{proof}

Next, we provide proofs for Lemmas~\ref{lem:single-cmpnt} and~\ref{lem:large-clique-free-cmpnt}.

\begin{proof}[Proof of Lemma~\ref{lem:single-cmpnt}]
Fix $k\in\NN$ and $\delta\in\left[\frac{(k-1)n}{k},\frac{kn}{k+1}\right]$. Let $G$ be a graph on $n\geq k(k+1)$ vertices with minimum degree $\delta(G)\geq\delta$ and exactly one $K_{k+1}$-component. Corollary~\ref{cor:hajnal-szemeredi-cor} implies $\ckf_{k+1}(G)\geq(k+1)(k\delta-(k-1)n)$.
\end{proof}

\begin{proof}[Proof of Lemma~\ref{lem:large-clique-free-cmpnt}]
We proceed by induction on $k$. For $k=1$, let $x$ be a vertex of $C$ and define $U:=\Gamma(x)\subseteq C$. Note that a $K_2$-component is a connected component, so in particular vertices with a neighbour in $C$ are also in $C$. $C$ contains no triangle, so $U$ is an independent set. Pick a set $S$ of $\delta$ vertices from $U$. Choose greedily for each $u\in S$ a distinct vertex $v\in V(G)$ such that $uv$ is an edge. Since $S\subseteq U$ is an independent set, all these vertices are not elements of $S$. Since $\deg(u)\ge\delta$, we can find a distinct vertex for each $u\in S$. This yields a set $M$ of $\delta$ vertex-disjoint edges all in $C$.

Now suppose $k\geq2$. Let $x$ be a vertex of $C$. Define $H:=G[\Gamma(x)]$ and $C_1:=\{x_1\dots x_{k-1}:x_1\dots x_{k-1}x\in C\}$. Note that $H$ is a graph on $m:=|\Gamma(x)|\geq\delta$ vertices with minimum degree $\delta(H)\geq\delta-n+m\geq\frac{k-2}{k-1}m$ and $C_1$ is a nonempty union of some $K_k$-components of $H$. Since $C$ does not contain a copy of $K_{k+2}$, any $K_k$-component $C'\subseteq C_1$ of $H$ does not contain a copy of $K_{k+1}$. Let $C'$ be such a $K_k$-component. Applying the induction hypothesis with $H$ and $C'$, we obtain a set $F''$ of $(k-1)\delta(H)-(k-2)m\geq(k-1)(\delta-n+m)-(k-2)m\geq k\delta-(k-1)n$ vertex-disjoint copies of $K_k$ which are all in $C'\subseteq C_1$; since these copies of $K_k$ lie in $\Gamma(x)$, they are also in $C$. Let $F'\subseteq F''$ be a subset of $F''$ containing $k\delta-(k-1)n$ vertex-disjoint copies of $K_k$. Choose greedily for each $f\in F'$ a distinct vertex $v\in V(G)$ such that $fv$ is a copy of $K_{k+1}$. Since $C'$ is $K_{k+1}$-free, all these vertices are not neighbours of $x$ and in particular are not vertices of elements of $F'$. Since $|\Gamma(f)|\geq k\delta-(k-1)n$ by Lemma~\ref{lem:common-nbrhood-size}, we can find a distinct vertex for each $f\in F'$. This yields a set $F$ of $k\delta-(k-1)n$ vertex-disjoint copies of $K_{k+1}$ which are all in $C$.
\end{proof}

\subsection{Configurations} \label{subsection:configs}

To prepare for the proof of Lemma~\ref{lem:general-cmpnt}, we first introduce some definitions useful for the analysis of the graph structure. Let $G$ be a graph with $K_{k+1}$-components $C_1,\dots,C_r$. The \emph{$K_{k+1}$-interior} $\intr_k(G)$ of $G$ is the set of vertices of $G$ which are in more than one of the $K_{k+1}$-components. For a $K_{k+1}$-component $C_i$, the \emph{interior} $\intr(C_i)$ of $C_i$ is the set of vertices of $C_i$ which are in $\intr_k(G)$. The \emph{exterior} $\extr(C_i)$ of $C_i$ is the set of vertices of $C_i$ which are in no other $K_{k+1}$-component of $G$. To give an example, by definition the graph $G_p(k,n,\delta)$ has $r_p(k,n,\delta)$ $K_{k+1}$-components; its $K_{k+1}$-interior is the disjoint union of the $k-1$ independent sets $I_1,\dots,I_{k-1}$ (using notation from the definition of $G_p(k,n,\delta)$ on page~\pageref{Gpkndelta} in Section~\ref{section:intro}) and its component exteriors are the cliques $X_1,\dots,X_{r_p(k,n,\delta)}$. Note that $\intr_k(G_p(k,n,\delta))$ induces a complete $(k-1)$-partite graph and in particular contains no copy of $K_k$.

A key part of our proof of Lemma~\ref{lem:general-cmpnt} involves the analysis of the case in which $\intr_k(G)$ contains a copy of $K_k$, which corresponds to the case in which $\intr_2(G)$ contains an edge in~\cite{AllenBoettcherHladky}. However, unlike in the case in~\cite{AllenBoettcherHladky}, in our case the graph structure is not immediately amenable to the construction of connected $K_{k+1}$-factors. To overcome this, we introduce a family of configurations in this subsection which gives graph structures that facilitate the construction of connected $K_{k+1}$-factors.

\begin{defn}[Configurations]\label{defn:config}
Let $j,k,\ell \in \NN$ satisfy $1 \le j < \ell \le k$. We say that a graph $G$ \emph{contains} the configuration \config{k}{\ell}{j} if there is a (multi)set
\[ \{ u_i : i \in [k] \} \sqcup \{ v_i : j < i \le \ell \} \sqcup \{ w_{i,h} : j < i \le \ell, h \in [\ell-1] \} \]
of (not necessarily distinct) vertices in $V(G)$ such that
\begin{enumerate}[label=(CG\arabic{*})]
\item \label{item:config-central} $u_1\dots u_k$ is a copy of $K_k$ in a $K_{k+1}$-component $C$ of $G$,
\item \label{item:config-adjacent} $u_1\dots u_j v_{j+1}\dots v_\ell u_{\ell+1}\dots u_k$ is a copy of $K_k$ of $G$ not in $C$, and
\item \label{item:config-dangle} $u_{\ell+1}\dots u_ku_p w_{p,1}\dots w_{p,\ell-1}$ is a copy of $K_k$ of $G$ not in $C$ for every $j<p\leq\ell$.
\end{enumerate}
We say that $G$ \emph{does not contain} the configuration \config{k}{\ell}{j} if there is no such (multi)set of vertices in $V(G)$.
\end{defn}

One may regard the configuration \config{k}{\ell}{j} as a collection of copies of $K_k$ which satisfies the following.
\begin{enumerate}[label=(\roman{*})]
\item There are $k-\ell$ vertices common to all the copies of $K_k$.
\item There is a `central' copy of $K_k$ in some $K_{k+1}$-component $C$ and all the other copies of $K_k$ do not belong to $C$.
\item After deleting the $k-\ell$ common vertices from the copies of $K_k$, we obtain a collection of copies of $K_\ell$. The `central' copy of $K_\ell$ shares $j$ vertices with one other copy of $K_\ell$ and each of its remaining vertices has one copy of $K_\ell$ `dangling' off it.
\end{enumerate}
Note that these configurations are by no means distinct, since `non-central' copies of $K_k$ and vertices not on the same copy of $K_k$ need not be distinct. For example, a graph that contains \config{3}{3}{2} also contains \config{3}{3}{1} -- set $v_2:=u_2, w_{2,1}:=u_1, w_{2,2}:=v_3$. The family of configurations for $k=3$ can be found in Figure~\ref{fig:config-eg}.

\begin{figure}
\centering
\begin{tikzpicture}
\draw [fill=blue] (0,0) -- (-1,-1) -- (1,-1) -- (0,0);
\draw [fill=red] (0,0) -- (-1,1) -- (1,1) -- (0,0);
\draw [fill=red] (-1,-1) -- (-2,-2) -- (-3,-1) -- (-1,-1);
\draw [fill=red] (1,-1) -- (2,-2) -- (3,-1) -- (1,-1);
\node [draw, circle, radius=1, fill=black, black, label=left:{\small $w_{2,2}$}] (ind1) at (-2,-2) {};
\node [draw, circle, radius=1, fill=black, black, label=above:{\small $w_{2,1}$}] (ind2) at (-3,-1) {};
\node [draw, circle, radius=1, fill=black, black, label=above:{\small $u_2$}] (ind3) at (-1,-1) {};
\node [draw, circle, radius=1, fill=black, black, label=left:{\small $w_{3,2}$}] (ind4) at (2,-2) {};
\node [draw, circle, radius=1, fill=black, black, label=above:{\small $w_{3,1}$}] (ind5) at (3,-1) {};
\node [draw, circle, radius=1, fill=black, black, label=above:{\small $u_3$}] (ind6) at (1,-1) {};
\node [draw, circle, radius=1, fill=black, black, label=above:{\small $v_2$}] (ind7) at (-1,1) {};
\node [draw, circle, radius=1, fill=black, black, label=above:{\small $v_3$}] (ind8) at (1,1) {};
\node [draw, circle, radius=1, fill=black, black, label=left:{\small $u_1$}] (ind9) at (0,0) {};
\draw (ind1) to (ind2);
\draw (ind1) to (ind3);
\draw (ind2) to (ind3);
\draw (ind4) to (ind5);
\draw (ind4) to (ind6);
\draw (ind5) to (ind6);
\draw (ind7) to (ind8);
\draw (ind7) to (ind9);
\draw (ind8) to (ind9);
\draw (ind3) to (ind6);
\draw (ind3) to (ind9);
\draw (ind6) to (ind9);
\end{tikzpicture}
\begin{tikzpicture}
\draw [fill=blue] (0,0) -- (-1,-1) -- (1,-1) -- (0,0);
\draw [fill=red] (0,0) -- (-1,1) -- (1,1) -- (0,0);
\draw [fill=red] (0,-2) -- (-1,-1) -- (1,-1) -- (0,-2);
\node [draw, circle, radius=1, fill=black, black, label=above:{\small $w_{3,1}$}] (ind1) at (-1,1) {};
\node [draw, circle, radius=1, fill=black, black, label=above:{\small $w_{3,2}$}] (ind2) at (1,1) {};
\node [draw, circle, radius=1, fill=black, black, label=left:{\small $u_3$}] (ind3) at (0,0) {};
\node [draw, circle, radius=1, fill=black, black, label=above:{\small $u_1$}] (ind4) at (-1,-1) {};
\node [draw, circle, radius=1, fill=black, black, label=above:{\small $u_2$}] (ind5) at (1,-1) {};
\node [draw, circle, radius=1, fill=black, black, label=left:{\small $v_3$}] (ind6) at (0,-2) {};
\draw (ind1) to (ind2);
\draw (ind1) to (ind3);
\draw (ind2) to (ind3);
\draw (ind3) to (ind4);
\draw (ind3) to (ind5);
\draw (ind4) to (ind5);
\draw (ind4) to (ind6);
\draw (ind5) to (ind6);
\end{tikzpicture}
\begin{tikzpicture}
\draw [fill=blue] (0,0) -- (-1,-1) -- (1,-1) -- (0,0);
\draw [fill=red] (0,0) -- (-1,-1) -- (-2,0) -- (0,0);
\draw [fill=red] (0,0) -- (1,-1) -- (2,0) -- (0,0);
\node [draw, circle, radius=1, fill=black, black, label=above:{\small $v_2$}] (ind1) at (-2,0) {};
\node [draw, circle, radius=1, fill=black, black, label=left:{\small $u_1$}] (ind2) at (-1,-1) {};
\node [draw, circle, radius=1, fill=black, black, label=above:{\small $u_3$}] (ind3) at (0,0) {};
\node [draw, circle, radius=1, fill=black, black, label=right:{\small $u_2$}] (ind4) at (1,-1) {};
\node [draw, circle, radius=1, fill=black, black, label=above:{\small $w_{2,1}$}] (ind5) at (2,0) {};
\draw (ind1) to (ind2);
\draw (ind1) to (ind3);
\draw (ind2) to (ind3);
\draw (ind2) to (ind4);
\draw (ind3) to (ind4);
\draw (ind3) to (ind5);
\draw (ind4) to (ind5);
\end{tikzpicture}
\caption{\config{3}{3}{1}, \config{3}{3}{2} and \config{3}{2}{1}}
\label{fig:config-eg}
\end{figure}

Now let us sketch the proof of Lemma~\ref{lem:general-cmpnt}. We will distinguish two cases as follows.
\begin{enumerate}[label=(\roman{*})]
    \item \label{item:intr-has-K_k} $\intr_k(G)$ contains a copy of $K_k$,
    \item \label{item:intr-no-K_k} $\intr_k(G)$ does not contain a copy of $K_k$.
\end{enumerate}
Case~\ref{item:intr-has-K_k} is equivalent to $G$ containing \config{k}{k}{1}, so $G$ contains a member of our family of configurations. By Lemma~\ref{lem:inductive-hangers-high-min-deg-layer-int}, $G$ in fact contains a configuration of the form \config{k}{\ell+1}{\ell}. Consider the configuration of this form contained in $G$ with minimal $\ell$. We will distinguish two cases: when $\ell=1$ and when $\ell>1$. In the first case, Lemma~\ref{lem:high-min-deg-k-intclique-edge-int-nbrhood-indset} will tell us that common neighbourhoods of a certain form are independent sets, which will enable us to apply Lemma~\ref{lem:high-min-deg-gen-//-intclique-clique-factor} to obtain the desired large connected $K_{k+1}$-factor. In the second case, we know that $G$ does not contain \config{k}{2}{1}. Lemma~\ref{lem:high-min-deg-intclique-edge-int-nbrhood-indset-umbr}~\ref{lem:high-min-deg-intclique-edge-int-nbrhood-indset} will tell us that common neighbourhoods of a certain form are independent sets and we will be able to apply Lemma~\ref{lem:high-min-deg-gen-intclique-clique-factor} to obtain the desired large connected $K_{k+1}$-factor. We remark that the argument presented above for the second case is inadequate when $\delta$ is close to $\frac{(2k-1)n}{2k+1}$. We will use an essentially similar but more tailored approach in the form of Lemmas~\ref{lem:high-min-deg-intclique-edge-int-nbrhood-indset-umbr}~\ref{lem:high-min-deg-intclique-edge-int-nbrhood-indset-matching} and~\ref{lem:high-min-deg-gen-intclique-clique-factor-matching}.

In Case~\ref{item:intr-no-K_k}, $G$ resembles our extremal graphs and has enough structure for the application of our construction methods to obtain the desired large connected $K_{k+1}$-factor. This approach works for most values of $\delta$ below $\frac{(2k-1)n}{2k+1}$. For $\delta\geq\frac{(2k-1)n}{2k+1}$ however, we find that our greedy-type methods are insufficient. To overcome this, we will employ a Hall-type argument in the form of Lemma~\ref{lem:very-high-min-deg-non-extr-partite-int-clique-factor}.

\section{Structure and methods} \label{section:structure-methods}

In this section we develop useful techniques for our proof of Lemma~\ref{lem:general-cmpnt}. These include structural results pertaining to the family of configurations defined in Section~\ref{subsection:configs} and procedures for constructing connected $K_{k+1}$-factors.

\subsection{Configurations and structure} \label{subsection:configs-structure}

In this subsection we prove structural facts about our family of configurations which are useful for our proof of Lemma~\ref{lem:general-cmpnt}.

A key argument in our proof of Lemma~\ref{lem:general-cmpnt} is that a graph without a sufficiently large connected $K_{k+1}$-factor in fact contains no member of the family of configurations defined previously in Section~\ref{subsection:configs}. The following lemma establishes an inductive-like relationship between the members of our family of configurations.

\begin{lemma} \label{lem:inductive-hangers-high-min-deg-layer-int}
Let $j,k,\ell \in \NN$ satisfy $3\leq j+2\leq\ell\leq k$ and $G$ be a graph on $n$ vertices with minimum degree $\delta>\frac{(k-1)n}{k}$ and at least two $K_{k+1}$-components. Suppose that $G$ does not contain \config{k}{\ell}{j+1}, \config{k}{\ell}{\ell-1} or \config{k}{\ell-j}{1}. Then $G$ does not contain \config{k}{\ell}{j}.
\end{lemma}

\begin{proof}
Suppose that $G$ contains \config{k}{\ell}{j}. By Definition~\ref{defn:config}, there is a (multi)set
\[ \{ u_i : i \in [k] \} \sqcup \{ v_i : j < i \le \ell \} \sqcup \{ w_{i,h} : j < i \le \ell, h \in [\ell-1] \} \]
of vertices in $V(G)$ such that~\ref{item:config-central}--\ref{item:config-dangle} hold. Observe that the vertices $u_1,\dots,u_k,v_{j+1},\dots$, $v_\ell$ are all distinct: if $u_a = v_b$ for some $j < a,b \le \ell$, the copy of $K_k$ containing $v_b$ would share at least $k-\ell+j+1$ vertices with the `central' copy of $K_k$, thereby yielding \config{k}{\ell}{j+1}.

For each $j<i\leq\ell$ define $S_i:=\Gamma(v_i,u_1,\dots,u_{i-1},u_{i+1},\dots,u_k)\backslash\{u_i\}$ and $f_i:=u_1\dots u_{i-1}u_{i+1}$ $\dots u_k$. Set $f':=u_1\dots u_ju_{\ell+1}\dots u_k$. Let $j<i\leq\ell$. Observe that $v_i$ has at least two non-neighbours in $\{u_{j+1},\dots,u_\ell\}$. Indeed, without loss of generality, suppose that $v_i$ is adjacent to the vertices $u_{j+1},\dots,u_{\ell-1}$. Here $B := f'u_{j+1}\dots u_{\ell-1}v_i$ is a copy of $K_k$ sharing at least $k-1$ vertices with $u_1\dots u_k$ and at least $k-\ell+j+1$ vertices with $f'v_{j+1}\dots v_\ell$. If $B \in C$, then by taking $B$ as the `central' copy of $K_k$ we have \config{k}{\ell}{j+1}. If $B \notin C$, then with $B$ replacing $f'v_{j+1}\dots v_\ell$ we have \config{k}{\ell}{\ell-1}.

Now applying Lemma~\ref{lem:common-nbrhood-size} with $U=V(G)\backslash\{u_1,\dots,u_k,v_i\}$, we have
\[|S_i|\geq (k-1)(\delta-k)+(\delta-k+2)-(k-1)(n-k-1)=k\delta-(k-1)n+1>0.\]
Pick $s_i\in S_i$ and complete $f's_iv_i$ to a copy $Z_i := f's_iv_iz_{i,1}\dots z_{i,\ell-j-2}$ of $K_k$ by Lemma~\ref{lem:clique-extn}. Observe that $f_is_i\in C$: if not, then $u_i w_{i,1}\dots w_{i,k-1}\notin C$, $f_iu_i\in C$ and $f_is_i\notin C$ would yield \config{k}{\ell}{\ell-1} with $f_iu_i$ as the `central' copy of $K_k$. Furthermore, we have $Z_i\in C$: if not, then $u_p w_{p,1}\dots w_{p,k-1}\notin C$ for $j<p\leq\ell$, $f_is_i\in C$ and $Z_i\notin C$ would yield \config{k}{\ell-j}{1} with $f_is_i$ as the `central' copy of $K_k$. But now $Z_i\in C$ for $j<i\leq\ell$ with $f'v_{j+1}\dots v_\ell\notin C$ as the `central' copy of $K_k$ yields \config{k}{\ell-j}{1}, which is a contradiction.
\end{proof}

The following lemma collects structural properties useful for the construction of connected $K_{k+1}$-factors in graphs which contain \config{k}{2}{1}.

\begin{lemma} \label{lem:high-min-deg-k-intclique-edge-int-nbrhood-indset}
Let $k\geq2$. Let $f=u_1\dots u_{k-1}$ be a copy of $K_{k-1}$ in a graph $G$ which lies in distinct $K_{k+1}$-components $C_1$ and $C_2$ of $G$. Let $uv$ and $wu_k$ be edges of $G$ such that $fu,fv\in C_1$ and $fw,fu_k\in C_2$. Then $\Gamma(u_1,\dots,u_{i-1},u_{i+1},\dots,u_{k},w,u,v)$ is an independent set for each $i\in[k-1]$.
\end{lemma}

\begin{proof}
Fix $i\in[k-1]$. Let $U:=u_1\dots u_{i-1}u_{i+1}\dots u_{k-1}$. Suppose that $X:=\Gamma(u_1,\dots,u_{i-1},u_{i+1}$, $\dots,u_{k},w,u,v)$ contains an edge~$u'v'$. Note that $Uu_iu\in C_1$ and $Uu_iuv$ is a copy of $K_{k+1}$ in $G$ so $Uuv\in C_1$; also $Uuvu'v'$ is a copy of $K_{k+2}$ in $G$ so $Uu'v'\in C_1$. On the other hand, $Uu_kw\in C_2$ and $Uu_{k}wu'v'$ is a copy of $K_{k+2}$ in $G$ so $Uu'v'\in C_2$. Since no copy of $K_k$ is in more than one $K_{k+1}$-component, this is a contradiction. Hence, $X$ contains no edge and is therefore an independent set.
\end{proof}

The following lemma provides structural properties useful for the construction of connected $K_{k+1}$-factors in graphs containing \config{k}{\ell}{\ell-1} for some $3\leq\ell\leq k$ but not \config{k}{2}{1}.

\begin{lemma} \label{lem:high-min-deg-intclique-edge-int-nbrhood-indset-umbr}
Let $k\geq2$ and $i\in[k-1]$ be integers. Let $G$ be a graph which does not contain \config{k}{2}{1} and $f=u_1\dots u_{k-1}$ be a copy of $K_{k-1}$ in $G$ which lies in distinct $K_{k+1}$-components $C_1$ and $C_2$ of $G$. Let $uv$ be an edge of $G$ such that $fu,fv\in C_1$ and $w$ be a vertex of $G$ such that $fw\in C_2$. Then 
\begin{enumerate}[label=(\roman*)]
\item \label{lem:high-min-deg-intclique-edge-int-nbrhood-indset} $\Gamma(u_1,\dots,u_{i-1},u_{i+1},\dots,u_{k-1},w,u,v)$ is an independent set.
\item \label{lem:high-min-deg-intclique-edge-int-nbrhood-indset-matching} $\Gamma(x_1,\dots,x_{i-1},u_{i+1},\dots,u_{k-1},w,u,v)$ is an independent set for any copy $g:=x_1\dots x_{i-1}$ of $K_{i-1}$ in $G$ such that we have $x_j\in\Gamma(u_{j+1},\dots,u_{k-1},w,u,v)$ for each $j<i$.
\end{enumerate} 
\end{lemma}

\begin{proof}
Note that $u_1\dots u_{i-1}$ is a copy of $K_{i-1}$ such that for each $j<i$ we have $u_j\in\Gamma(u_{j+1},\dots$, $u_{k-1},w,u,v)$, so~\ref{lem:high-min-deg-intclique-edge-int-nbrhood-indset} follows from~\ref{lem:high-min-deg-intclique-edge-int-nbrhood-indset-matching}. Hence, it remains to prove~\ref{lem:high-min-deg-intclique-edge-int-nbrhood-indset-matching}.

Fix a copy $g:=x_1\dots x_{i-1}$ of $K_{i-1}$ such that for each $j<i$ we have $x_j\in\Gamma(u_{j+1},\dots,u_{k-1},w$, $u,v)$. For each $j\in[i]$ set $U'_j:=u_1\dots u_{j-1}$, $g_j:=x_1\dots x_{j-1}$, $U_j:=u_{j+1}\dots u_{k-1}$ and $f_j:=U'_jU_j$. Set $U_0:=f$. We shall prove by induction that $g_jU_juv\in C_1$ for each $j\in[i]$. For $j=1$, note that $fu\in C_1$ and $fuv$ is a copy of $K_{k+1}$ in $G$ so $f_1uv=g_1U_1uv\in C_1$. For $j\geq2$, note that $g_{j-1}U_{j-1}uv\in C_1$ by the induction hypothesis and $g_jU_{j-1}uv$ is a copy of $K_{k+1}$ in $G$ so $g_jU_juv\in C_1$, completing the proof by induction. In particular, we have $gU_iuv\in C_1$.

Set $X:=\Gamma(x_1,\dots,x_{i-1},u_{i+1},\dots,u_{k-1},w,u,v)$ and suppose that there is an edge $u'v'$ with $u',v' \in X$. Now by the definitions of $g$ and $X$ we have that $gU_iuvu'v'$ is a copy of $K_{k+2}$ in $G$, so we have $gU_iu'v'\in C_1$. Furthermore, since $g_jU_juv\in C_1$ and $g_jU_{j-1}uv$ is a copy of $K_{k+1}$ in $G$ for each $j\in[i]$, we have $g_jU_{j-1}u\in C_1$ for each $j\in[i]$.

Now we shall prove that $g_jU_{j-1}w\in C_2$ for each $j\in[i]$. For $j=1$, we have $fw=g_1U_0w\in C_2$. For $j\geq2$, observe that the induction hypothesis implies that $g_jU_{j-1}w\in C_2$: if not, then $g_{j-1}U_{j-2}u\in C_1$ from before, $g_{j-1}U_{j-2}w\in C_2$ by the induction hypothesis and $g_jU_{j-1}w\notin C_2$ would yield \config{k}{2}{1} with $g_{j-1}U_{j-2}w\in C_2$ as the `central' copy of $K_k$ and $g_{j-1}U_{j-1}$ as the common vertices. This completes the proof by induction. In particular, we have $gU_{i-1}w\in C_2$. Now observe that $gU_iwu'\in C_2$: if not, then $gU_{i-1}u\in C_1$ from before, $gU_{i-1}w\in C_2$ and $gU_iwu'\notin C_2$ would yield \config{k}{2}{1} with $gU_{i-1}w\in C_2$ as the `central' copy of $K_k$ and $gU_i$ as the common vertices. Finally, $gU_iwu'v'$ is a copy of $K_{k+1}$ in $G$ so $gU_iu'v'\in C_2$, which contradicts our earlier deduction that $gU_iu'v'\in C_1$. Hence, $X$ contains no edge and is therefore an independent set.
\end{proof}

\subsection{Constructing connected $K_{k+1}$-factors} \label{subsection:clique-factor-constr}

In this subsection we develop greedy-type procedures for constructing connected $K_{k+1}$-factors which exploit certain structures in graphs of interest, including those proved in Section~\ref{subsection:configs-structure}. Lemmas~\ref{lem:high-min-deg-gen-//-intclique-clique-factor}, \ref{lem:high-min-deg-gen-intclique-clique-factor}, and~\ref{lem:high-min-deg-gen-intclique-clique-factor-matching} serve to formalise the achievable outcomes of these procedures.

Lemma~\ref{lem:high-min-deg-gen-//-intclique-clique-factor} represents a greedy-type procedure for constructing connected $K_{k+1}$-factors in a graph using two parallel processes following two closely related partitions of the vertex set. The purpose of this lemma is to obtain sufficiently large connected $K_{k+1}$-factors in graphs containing \config{k}{2}{1}. The sets $A$ and $A'$ in Lemma~\ref{lem:high-min-deg-gen-//-intclique-clique-factor} contain the vertices avoided by the two parallel processes. Note that the larger $A$ and $A'$ are, the smaller $s_1$ and $t_1$ are. Since the sizes of $s_1$ and $t_1$ will often be the key determinants of the attainable size of a connected $K_{k+1}$-factor, we will think of $A$ and $A'$ as `bad' sets. We remark that while we formally allow the quantities $s_1,s_2,t_1$ and $t_2$ to be negative to reduce the overall proof complexity, in practice they will always be non-negative.

\begin{lemma} \label{lem:high-min-deg-gen-//-intclique-clique-factor}
Let $2\leq b\leq c\leq k$ be integers. Let $G$ be a graph on $n$ vertices with minimum degree $\delta=\delta(G)>\frac{(k-1)n}{k}$. Suppose there are two partitions of $V(G)$, one with vertex classes $U_1,U_2,X_1,\dots,X_{k-1},A$ and another with vertex classes $V_1,V_2 $, $X_1,\dots,X_{k-2},X',A'$, such that
\begin{enumerate}[label=(\alph*)]
\item \label{item:constr-parallel-sym-disjoint} $U_1\cap V_1=\nth$,
\item \label{item:constr-parallel-sym-no-edge} there are no edges between $U_1$ and $U_2$ and between $V_1$ and $V_2$,
\item \label{item:constr-parallel-sym-connected} all copies of $K_k$ in $G$ with at least two vertices from $U_1$ and all other vertices from $\bigcup_{i=1}^{k-2}X_i$, or at least two vertices from $V_1$ and all other vertices from $\bigcup_{i=1}^{k-2}X_i$, are $K_{k+1}$-connected,
\item \label{item:constr-parallel-sym-size-bound} $|X_i|\leq n-\delta$ for $i\in[k-1]$ and $|X'|\leq n-\delta$, and
\item \label{item:constr-parallel-sym-ind-set} $X_{i}\cap\Gamma(g)$ is an independent set for each $(i,g)$ where $i\in[k-2]$ and $g$ is a clique of order at least $i$ with at least two vertices from $U_1$ and all other vertices from $\bigcup_{j=1}^{i-1}X_j$, or at least two vertices from $V_1$ and all other vertices from $\bigcup_{j=1}^{i-1}X_j$.
\end{enumerate}
Let $F^U$ be a collection of vertex-disjoint copies of $K_b$ in $U_1$ and $F^V$ be a collection of vertex-disjoint copies of $K_c$ in $V_1$. Set $s_1:=\tfrac{k\delta-(k-1)n+(b-1)|U_2|-|A|}{2b-1}$ and $s_2:=\tfrac{k\delta-(k-1)n+(b-1)|U_2|-|V_1|}{2b-1}$. Set $t_1:=\tfrac{k\delta-(k-1)n+(c-1)|V_2|-|A'|-|U_1|(c-1)/b}{2c-1}$ and $t_2:=\tfrac{k\delta-(k-1)n+(c-1)|V_2|-|U_1|-|U_1|(c-1)/b}{2c-1}$.  Let $d_1,d_2\geq0$ satisfy $|V_2|\geq2d_2+d_1$. Then $G$ contains a connected $K_{k+1}$-factor of size at least
\begin{equation*}
(k+1)\min\left\lbrace|F^U|,\left\lfloor\frac{|U_2|}{2}\right\rfloor,d_1,s_1,s_2\right\rbrace
 + (k+1)\min\left\lbrace|F^V|,d_2,t_1,t_2\right\rbrace.
\end{equation*}
Moreover, if $F^V$ is empty then $G$ contains a connected $K_{k+1}$-factor of size at least
\begin{equation*}
(k+1)\min\left\lbrace|F^U|,\left\lfloor\frac{|U_2|}{2}\right\rfloor,d_1,s_1,s_2\right\rbrace.
\end{equation*}
\end{lemma}

The proof of this lemma proceeds as follows. We first describe the greedy-type procedure used to construct a connected $K_{k+1}$-factor. Then, we prove that our procedure does indeed produce a connected $K_{k+1}$-factor of the desired size. This will turn out to be an inductive argument where we will need to justify that we can make `good' choices at each step and the quantities $s_1,s_2,t_1,t_2$ are chosen to ensure success. For example, a copy of $K_k$ extending a copy of $K_b$ in $F^U$ has at least $k\delta-(k-1)n+(b-1)|U_2|-|A|$ common neighbours not in some `bad' set $A$; on the other hand, each copy of $K_b$ in $F^U$ may render up to $2b-1$ of these common neighbours unavailable, so $|F^U|\leq s_1$ ensures that there is still an available common neighbour.

\begin{proof}
Let $F_b^U\subseteq F^U$ and $F_c^V\subseteq F^V$ satisfy
\begin{align*}
|F_b^U|&=\max\left\lbrace0,\min\left\lbrace|F^U|,\left\lfloor\frac{|U_2|}{2}\right\rfloor,d_1,s_1,s_2\right\rbrace\right\rbrace \quad\textrm{and} \\
|F_c^V|&=\max\left\lbrace0,\min\left\lbrace|F^V|,d_2,t_1,t_2\right\rbrace\right\rbrace.
\end{align*}
In what follows, we use vertices in $U_1,X_1,\dots,X_{k-2}$ to extend each clique in $F_b^U$ to a copy of $K_k$ and vertices in $U_1,X_1,\dots,X_{k-2}$ to extend each clique in $F_c^V$ to a copy of $K_k$. These copies of $K_k$ will then be extended to copies of $K_{k+1}$ using vertices outside of $U_1,V_1,X_1,\dots,X_{k-2}$. Note that the resultant copies of $K_{k+1}$ will be $K_{k+1}$-connected by~\ref{item:constr-parallel-sym-connected}.

We build up our desired connected $K_{k+1}$-factor by running two parallel processes, one starting from $F_b^U$ in $U_1$ and the other starting from $F_c^V$ in $V_1$. Each process is a two-stage step-by-step process performing steps in tandem with the other process. Set $\overline{F}_{b-1}^U,\overline{F}_{c-1}^V:=\nth$. Stage one has steps $j=1,\dots,k-b+1$. In step $j\in[c-b]$ of stage one, we extend copies of $K_{b+j-1}$ in $F_{b+j-1}^U$ to vertex-disjoint copies of $K_{b+j}$ where possible. For each copy of $K_{b+j-1}$ in $F_{b+j-1}^U$ in turn we pick greedily, where possible, a common neighbour in $U_1$ which is not covered by $\overline{F}_{b-1}^U,\dotso,\overline{F}_{b+j-2}^U,F_{b+j-1}^U$ or previously chosen common neighbours. Since the vertices selected lie in $U_1$, $F_c^V$ is contained in $V_1$ and $U_1\cap V_1=\nth$ by~\ref{item:constr-parallel-sym-disjoint}, no vertex of $F_c^V$ is selected. Let $\overline{F}_{b+j-1}^U$ be the collection of copies of $K_{b+j-1}$ in $F_{b+j-1}^U$ which could not be extended and let $F_{b+j}^U$ be the collection of vertex-disjoint copies of $K_{b+j}$ which result from extending copies of $K_{b+j-1}$ in $F_{b+j-1}^U$. In step $j\in[k-b+1]\backslash[c-b]$ of stage one, we extend copies of $K_{b+j-1}$ in $F_{b+j-1}^U$ to vertex-disjoint copies of $K_{b+j}$ and copies of $K_{c+j-1}$ in $F_{c+j-1}^V$ to vertex-disjoint copies of $K_{c+j}$ where possible. For each copy of $K_{b+j-1}$ in $F_{b+j-1}^U$ in turn we pick greedily, where possible, a common neighbour in $U_1$ which is not covered by $\overline{F}_{b-1}^U,\dotso,\overline{F}_{b+j-2}^U,F_{b+j-1}^U$ or previously chosen common neighbours. Let $\overline{F}_{b+j-1}^U$ be the collection of copies of $K_{b+j-1}$ in $F_{b+j-1}^U$ which could not be extended and let $F_{b+j}^U$ be the collection of vertex-disjoint copies of $K_{b+j}$ which result from extending copies of $K_{b+j-1}$ in $F_{b+j-1}^U$. We do the same with $F_{c+j-1}^V$ within $V_1$. We end stage one after step $k-b+1$. Set $\overline{F}_{k+1}^U:=F_{k+1}^U$ and $\overline{F}_{k+1}^V:=F_{k+1}^V$. At this point, we have collections $\overline{F}_i^U$ and $\overline{F}_j^V$ of vertex-disjoint copies of $K_i$ and $K_j$ respectively, for each $i=b,\dots,k+1$ and $j=c,\dots,k+1$, some of which may be empty. Let $\overline{F}^U=\bigcup_{i=b}^{k+1}\overline{F}_i^U$ and $\overline{F}^V=\bigcup_{i=c}^{k+1}\overline{F}_i^V$. Note that $|\overline{F}^U|=|F_b^U|$ and $|\overline{F}^V|=|F_c^V|$. Order the elements of $\overline{F}^U\cup\overline{F}^V$ with those in $\overline{F}^U$ coming before those in $\overline{F}^V$, those in each of $\overline{F}^U$ and $\overline{F}^V$ in increasing size order, and those in each of $\overline{F}^U$ and $\overline{F}^V$ of the same size in some arbitrary order. We will use this ordering when attempting to extend cliques in stage two.

We begin stage two with $\widetilde{F}_0^U:=\overline{F}^U$ and $\widetilde{F}_0^V:=\overline{F}^V$. Stage two has steps $j=1,\dots,k-1$. In step $j\in[k-2]$ we attempt to extend each clique in $\widetilde{F}_{j-1}^U$ and $\widetilde{F}_{j-1}^V$ of order at most $k$ by one vertex using $X_j$. We will extend cliques in the order mentioned previously. For each clique of order at most $k$ in $\widetilde{F}_{j-1}^U$ and $\widetilde{F}_{j-1}^V$ in turn we pick greedily, where possible, a common neighbour in $X_j$ which is outside the previously chosen common neighbours. Let $\widetilde{F}_{j}^U$ and $\widetilde{F}_{j}^V$ be the collections of both cliques in $\widetilde{F}_{j-1}^U$ and $\widetilde{F}_{j-1}^V$ respectively of order $k+1$ and cliques resulting from the attempts to extend each clique of order at most $k$ in $\widetilde{F}_{j-1}^U$ and $\widetilde{F}_{j-1}^V$ respectively by one vertex, no matter whether they were successful or not. In step $k-1$ we attempt to extend each clique of order at most $k$ in $\widetilde{F}_{k-2}^U$ and $\widetilde{F}_{k-2}^V$ by one vertex using vertices of $G$ outside of $U_1\cup V_1\cup\left(\bigcup_{i=1}^{k-2}X_i\right)$ in a manner similar to that in earlier steps of stage two. We end stage two after step $k-1$ with collections $\widetilde{F}_{k-1}^U$ and $\widetilde{F}_{k-1}^V$ of $|F_b^U|$ and $|F_c^V|$ vertex-disjoint cliques in $G$ respectively.

We shall prove that $\widetilde{F}_{k-1}^U$ and $\widetilde{F}_{k-1}^V$ are collections of $|F_b^U|$ and $|F_c^V|$ vertex-disjoint copies of $K_{k+1}$ respectively. In fact, we shall prove that $\widetilde{F}_j^U$ and $\widetilde{F}_j^V$ are collections of $|F_b^U|$ and $|F_c^V|$ vertex-disjoint cliques of order at least $j+2$ respectively for each $j=b-2,\dots,k-1$. We shall first consider $\widetilde{F}_j^U$. We proceed by induction on $j$. The $j=b-2$ case is trivial. Consider $\widetilde{F}_j^U$ for $j\geq b-1$. By the induction hypothesis, $\widetilde{F}_{j-1}^U$ is a collection of $|F_b^U|$ vertex-disjoint cliques of order at least $j+1$. Hence, it suffices to show that the copies of $K_{j+1}$ in $\widetilde{F}_{j-1}^U$ are all extended to copies of $K_{j+2}$ in step $j$ to prove our claim. Observe that this holds trivially when $|F_b^U|=0$, so in what follows it is enough to consider when $|F_b^U|=\min\left\lbrace|F^U|,\left\lfloor\frac{|U_2|}{2}\right\rfloor,d_1,s_1,s_2\right\rbrace$.

Let $f$ be a copy of $K_{j+1}$ in $\widetilde{F}_{j-1}^U$ with $\ell\geq b$ vertices in $U_1$ and $\overline{f}$ be its corresponding clique in $\overline{F}^U$. In particular, $\overline{f}$ has $\ell$ vertices. Note that $f$ has vertices from only $X_1,\dots,X_{j-1},U_1$ and at most one vertex from each $X_i$. Define $I:=\{i:|f\cap X_i|=1\}$. Let $\overline{v}_i$ be the vertex of $f$ in $X_i$ for each $i\in I$.

First consider the case $j\leq k-2$. Every vertex $v$ of $U_1$ has at least $\delta-|A|-\deg(v;U_1)-\sum_{h\neq j}\deg(v;X_h)$ neighbours in $X_j$ and for each $i\in I$ the vertex $\overline{v}_i$ has at least $\delta-|A|-|U_2|-\deg(\overline{v}_i;U_1)-\sum_{h\neq j}\deg(\overline{v}_i;X_h)$ neighbours in $X_j$. By application of Lemma~\ref{lem:common-nbrhood-size} and noting that $|X_j|=n-|A|-|U_2|-|U_1|-\sum_{h\neq j}|X_h|$, the number of common neighbours of $f$ in $X_j$ is at least
\begin{equation} \label{eqn:common-nbrhood}
\begin{split}
a_j := &\sum_{v\in\overline{f}}\left(\delta-|A|-|U_2|-\deg(v;U_1)-\sum_{h\neq j}\deg(v;X_h)\right) \\
& + \sum_{i\in I} \left(\delta-|A|-|U_2|-\deg(\overline{v}_i;U_1)-\sum_{h\neq j}\deg(\overline{v}_i;X_h)\right) \\
& - j\left(n-|A|-|U_2|-|U_1|-\sum_{h\neq j}|X_h|\right) \\
= &(j+1)\delta-jn+(\ell-1)|U_2|-\left(\sum_{v\in f}\deg(v;U_1)-j|U_1|\right) \\
& - \sum_{h\neq j}\left(\sum_{v\in f}\deg(v;X_h)-j|X_h|\right)-|A|.
\end{split}
\end{equation}
Now we seek appropriate estimations of the terms in our expression. Since $\overline{f}$ could not be extended in step $\ell-b+1$ of stage one, $\Gamma(\overline{f};U_1)$ contains only vertices from elements of $\overline{F}_b^U,\dots,\overline{F}_{\ell}^U$ and $F_{\ell+1}^U$. For each $b \le h \le \ell$ the elements of $\overline{F}_h^U$ contain $h$ vertices each while the elements of $F_{\ell+1}^U$ contain $\ell+1$ vertices each. Furthermore, we have $|F_b^U|=|\overline{F}_b^U|+\dots+|\overline{F}_{\ell}^U|+|F_{\ell+1}^U|$ by the definitions of $\overline{F}_b^U,\dots,\overline{F}_{\ell}^U$, $F_{\ell+1}^U$. Hence, by applying Lemma~\ref{lem:common-nbrhood-size} to $U_1$ and $f$, we obtain
\begin{equation} \label{eqn:common-nbrhood-stage-one}
\sum_{v\in f}\deg(v;U_1)-j|U_1|\leq\deg(f;U_1)\leq\deg(\overline{f};U_1)\leq\ell|F_b^U|+|F_{\ell+1}^U|.
\end{equation}
For $h\in[k-1]$, by applying Lemma~\ref{lem:common-nbrhood-size} to $X_h$ and $f$, we get
\begin{equation} \label{eqn:common-nbrhood-size-parallel-X_h-f}
\sum_{v\in f}\deg(v;X_h)-j|X_h|\leq\deg(f;X_h).
\end{equation}
For $0 \le i < j$ let $f_i\in\widetilde{F}_i^U$ be the clique corresponding to $f$ right before step $i+1$ of stage two, so $f_i=\overline{f}\cup\{\overline{v}_h:h\in I\cap[i]\}$. Let $h\in I$. By the induction hypothesis $f_{h-1}$ is a clique of order at least $h+1$ with at least two vertices from $U_1$ and all other vertices from $\bigcup_{j=1}^{h-1}X_j$, so $\overline{v}_h$ has no neighbour in $\Gamma(f_{h-1};X_h)$ by~\ref{item:constr-parallel-sym-ind-set} applied with $(i,g) = (h,f_{h-1})$. Hence, we have $\deg(f_h;X_h)=0$ for all $h\in I$. Together with~\eqref{eqn:common-nbrhood-size-parallel-X_h-f} and the fact that $\deg(f;X_h)\leq\deg(f_h;X_h)$ for all $h\in I$, we obtain
\begin{equation} \label{eqn:common-nbrhood-stage-two-used}
\sum_{h\in I}\left(\sum_{v\in f} \deg(v;X_h)-j|X_h|\right)\leq\sum_{h\in I}\deg(f_h;X_h)=0.
\end{equation}
Given $i\notin I,i<j$, the clique $f_i$ was not extended in step $i$ of stage two. It follows that its common neighbourhood in $X_i$ contains only vertices used to extend cliques that came before it in the size ordering, of which there were fewer than $m:=|F_b^U|-|F_{\ell+1}^U|$. Noting further that $[j-1]\backslash I$ contains $\ell-2$ elements, by~\eqref{eqn:common-nbrhood-size-parallel-X_h-f} and that $\deg(f;X_h)\leq\deg(f_h;X_h)$ for all $h\in[j-1]\setminus I$, we get
\begin{equation} \label{eqn:common-nbrhood-stage-two-skipped}
\begin{split}
\sum_{h\in [j-1]\setminus I}\left(\sum_{v\in f}\deg(v;X_h)-j|X_h|\right)&\leq\sum_{h\in [j-1]\setminus I}\deg(f_h;X_h) \\
&\leq(\ell-2)(m-1).
\end{split}
\end{equation}
By~\ref{item:constr-parallel-sym-size-bound} $|X_h|\leq n-\delta$ for $h\in[k-1]$, so by~\eqref{eqn:common-nbrhood-size-parallel-X_h-f} we have
\begin{equation} \label{eqn:common-nbrhood-stage-two-future}
\sum_{h=j+1}^{k-1}\left(\sum_{v\in f}\deg(v;X_h)-j|X_h|\right)\leq\sum_{h=j+1}^{k-1}|X_h|\leq(k-j-1)(n-\delta).
\end{equation}
By~\eqref{eqn:common-nbrhood}, \eqref{eqn:common-nbrhood-stage-one}, \eqref{eqn:common-nbrhood-stage-two-used}, \eqref{eqn:common-nbrhood-stage-two-skipped}, \eqref{eqn:common-nbrhood-stage-two-future} and that $m=|F_b^U|-|F_{\ell+1}^U|$, we obtain
\begin{align*}
a_j&\ge (j+1)\delta-jn-|A|+(\ell-1)|U_2|-\ell|F_b^U|-|F_{\ell+1}^U|-(\ell-2)m \\
&\qquad -(k-j-1)(n-\delta) \\
&\ge k\delta-(k-1)n-|A|+(\ell-1)|U_2|-(\ell+1)|F_b^U|-(\ell-3)m.
\end{align*}
Since $\ell \ge b$ and by the definition of $F_b^U$ and $m$ we have $|U_2| \ge 2|F_b^U| \ge |F_b^U| + m$ and $|F_b^U| \ge m$, we obtain
\begin{align*}
a_j&\ge k\delta-(k-1)n-|A|+(\ell-1)|U_2|-(\ell+1)|F_b^U|-(\ell-3)m \\
&= k\delta-(k-1)n-|A|+(\ell-2)(|U_2|-|F_b^U|-m)+|U_2|-3|F_b^U|+m \\
&\ge k\delta-(k-1)n-|A|+(b-2)(|U_2|-|F_b^U|-m)+|U_2|-3|F_b^U|+m \\
&\ge k\delta-(k-1)n-|A|+(b-1)|U_2|-(2b-1)|F_b^U|+m.
\end{align*}
Now by the definition of $s_1$ we have
\begin{align*}
a_j&\ge k\delta-(k-1)n-|A|+(b-1)|U_2|-(2b-1)|F_b^U|+m \\
&\ge (2b-1)(s_1-|F_b^U|)+m\geq m
\end{align*}
so we are indeed able to pick a vertex in $X_j$ to extend $f$.

For the case $j=k-1$, an analysis analogous to~\eqref{eqn:common-nbrhood} implies that the number of common neighbours of $f$ outside of $U_1\cup U_2\cup V_1\cup\left(\bigcup_{i=1}^{k-2}X_i\right)$ is at least
\begin{align*}
a_{k-1}&:=k\delta-(k-1)n+(\ell-1)|U_2|-\left(\sum_{v\in f}\deg(v;U_1)-(k-1)|U_1|\right) \\
&\qquad-\sum^{k-2}_{h=1}\left(\sum_{v\in f}\deg(v;X_h)-(k-1)|X_h|\right)-|V_1|.
\end{align*}
By~\eqref{eqn:common-nbrhood-stage-one}, \eqref{eqn:common-nbrhood-stage-two-used}, \eqref{eqn:common-nbrhood-stage-two-skipped} and that $m=|F_b^U|-|F_{\ell+1}^U|$, we obtain
\begin{equation*}
a_{k-1}\ge k\delta-(k-1)n-|V_1|+(\ell-1)|U_2|-(\ell+1)|F_b^U|-(\ell-3)m.
\end{equation*}
Since $\ell \ge b$ and by the definition of $F_b^U$ and $m$ we have $|U_2| \ge 2|F_b^U| \ge |F_b^U| + m$ and $|F_b^U| \ge m$, we obtain
\begin{align*}
a_{k-1}&\geq k\delta-(k-1)n-|V_1|+(\ell-1)|U_2|-(\ell+1)|F_b^U|-(\ell-3)m \\
&= k\delta-(k-1)n-|V_1|+(\ell-2)(|U_2|-|F_b^U|-m)+|U_2|-3|F_b^U|+m \\
&\ge k\delta-(k-1)n-|V_1|+(b-2)(|U_2|-|F_b^U|-m)+|U_2|-3|F_b^U|+m \\
&\ge k\delta-(k-1)n-|V_1|+(b-1)|U_2|-(2b-1)|F_b^U|+m.
\end{align*}
Now by the definition of $s_2$ we have
\begin{align*}
a_{k-1}&\ge k\delta-(k-1)n-|V_1|+(b-1)|U_2|-(2b-1)|F_b^U|+m \\
&\ge (2b-1)(s_2-|F_b^U|)+m\geq m
\end{align*}
so we are indeed able to pick a vertex outside of $U_1\cup U_2\cup V_1\cup\left(\bigcup_{i=1}^{k-2}X_i\right)$ to extend $f$. This proves that copies of $K_{j+1}$ in $\widetilde{F}_{j-1}^U$ are all extended to copies of $K_{j+2}$ in step $j$ and so by induction $\widetilde{F}_j^U$ is a collection of $|F_b^U|$ vertex-disjoint cliques of order at least $j+2$ for each $j=b-2,\dots,k-1$. In particular, $\widetilde{F}_{k-1}^U$ is a collection of $|F_b^U|$ vertex-disjoint copies of $K_{k+1}$.

The proof for the $\widetilde{F}_j^V$ case is very similar to that for the $\widetilde{F}_j^U$ case. We also proceed by induction on $j$ and here the $j=c-2$ case is trivial. As in the $\widetilde{F}_j^U$ case, the desired outcome holds trivially when $|F_c^V|=0$, so in what follows it is enough to consider when $|F_c^V|=\min\left\lbrace|F^V|,d_2,t_1,t_2\right\rbrace$. Let $f$ be a copy of $K_{j+1}$ in $\widetilde{F}_{j-1}^V$ with $\ell\geq c$ vertices in $V_1$ and $\overline{f}$ be its corresponding clique in $\overline{F}^V$. In particular, $\overline{f}$ has $\ell$ vertices. Define $I:=\{i:|f\cap X_i|=1\}$ and let $\overline{v}_i$ be the vertex of $f$ in $X_i$ for each $i\in I$. Note that $f$ has vertices from only $X_1,\dots,X_{j-1},U_1$ and at most one vertex from each $X_i$.

Consider $c-1\leq j\leq k-1$. Let $m':=|F_c^V|-|F_{\ell+1}^V|$. An analysis analogous to~\eqref{eqn:common-nbrhood} implies that the number of common neighbours of $f$ in $X_j$ is at least 
\begin{equation} \label{eqn:common-nbrhood-v}
\begin{split}
b_j&:=(j+1)\delta-jn+(\ell-1)|V_2|-\left(\sum_{v\in f}\deg(v;V_1)-j|V_1|\right) \\
&\qquad-\sum_{h\neq j}\left(\sum_{v\in f}\deg(v;X_h)-j|X_h|\right)-|A'|. 
\end{split}
\end{equation}
By analyses similar to those for~\eqref{eqn:common-nbrhood-stage-one},~\eqref{eqn:common-nbrhood-stage-two-used},~\eqref{eqn:common-nbrhood-stage-two-skipped} and~\eqref{eqn:common-nbrhood-stage-two-future}, we also have
\begin{align}
&\sum_{v\in f}\deg(v;V_1)-j|V_1| \le \ell|F_c^V|+|F_{\ell+1}^V|, \label{eqn:common-nbrhood-stage-one-v} \\
&\sum_{h\in I}\left(\sum_{v\in f} \deg(v;X_h)-j|X_h|\right) \le 0, \label{eqn:common-nbrhood-stage-two-used-v} \\
&\sum_{h\in [j-1]\setminus I}\left(\sum_{v\in f}\deg(v;X_h)-j|X_h|\right) \le (\ell-2)(m'+|F_b^U|-1), \label{eqn:common-nbrhood-stage-two-skipped-v} \\
&\sum_{h=j+1}^{k-1}\left(\sum_{v\in f}\deg(v;X_h)-j|X_h|\right) \le (k-j-1)(n-\delta),  \label{eqn:common-nbrhood-stage-two-future-v}
\end{align}
respectively. The `additional' term of $|F_b^U|$ in~\eqref{eqn:common-nbrhood-stage-two-skipped-v} (cf.\ \eqref{eqn:common-nbrhood-stage-two-skipped}) arises because in each step of stage two we extend the cliques corresponding to $F_b^U$ before those corresponding to $F_c^V$. Now by~\eqref{eqn:common-nbrhood-v}, \eqref{eqn:common-nbrhood-stage-one-v}, \eqref{eqn:common-nbrhood-stage-two-used-v}, \eqref{eqn:common-nbrhood-stage-two-skipped-v} and~\eqref{eqn:common-nbrhood-stage-two-future-v} and that $m'=|F_c^V|-|F_{\ell+1}^V|$, we obtain
\begin{align*}
b_j&\ge (j+1)\delta-jn-|A'|+(\ell-1)|V_2|-\ell|F_c^V|-|F_{\ell+1}^V| \\
&\qquad -(\ell-2)(m'+|F_b^U|)-(k-j-1)(n-\delta) \\
&\ge k\delta-(k-1)n-|A'|+(\ell-1)|V_2|-(\ell+1)|F_c^V|-(\ell-2)|F_b^U| \\
&\qquad -(\ell-3)m'.
\end{align*}
Since $\ell \ge c$ and by the definition of $F_b^U$, $F_c^V$, $d_1$, $d_2$ and $m$ we have $|F_c^V| \ge m'$ and $|V_2| \ge 2d_2 + d_1 \ge 2|F_c^V | + |F_b^U| \ge |F_c^V | + |F_b^U| + m'$, we obtain
\begin{align*}
b_j&\ge k\delta-(k-1)n-|A'|+(\ell-1)|V_2|-(\ell+1)|F_c^V|-(\ell-2)|F_b^U| \\
&\qquad -(\ell-3)m' \\
&= k\delta-(k-1)n-|A'|+(\ell-2)(|V_2|-|F_c^V|-|F_b^U|-m') \\
&\qquad +|V_2|-3|F_c^V|+m' \\
&\ge k\delta-(k-1)n-|A'|+(c-2)(|V_2|-|F_c^V|-|F_b^U|-m') \\
&\qquad +|V_2|-3|F_c^V|+m' \\
&\ge k\delta-(k-1)n-|A'|+(c-1)|V_2|-(2c-1)|F_c^V|-(c-2)|F_b^U|+m'.
\end{align*}
Now by the definition of $t_1$ we have
\begin{align*}
b_j&\ge k\delta-(k-1)n-|A'|+(c-1)|V_2|-(2c-1)|F_c^V|-(c-2)|F_b^U|+m' \\
&\ge (2c-1)(t_1-|F_c^V|) +  m'+ |F_b^U| \ge m'+|F_b^U|
\end{align*}
so we are indeed able to pick a vertex in $X_j$ to extend $f$.

For the case $j=k-1$, an analogous analysis implies that the number of common neighbours of $f$ outside of $U_1\cup V_1\cup V_2\cup\left(\bigcup_{i=1}^{k-2}X_i\right)$ is at least
\begin{align*}
b_{k-1}&:=k\delta-(k-1)n+(\ell-1)|V_2|-\left(\sum_{v\in f}\deg(v;V_1)-(k-1)|V_1|\right) \\
&\qquad-\sum^{k-2}_{h=1}\left(\sum_{v\in f}\deg(v;X_h)-(k-1)|X_h|\right)-|U_1|.
\end{align*}
By~\eqref{eqn:common-nbrhood-stage-one-v}, \eqref{eqn:common-nbrhood-stage-two-used-v} and~\eqref{eqn:common-nbrhood-stage-two-skipped-v} and that $m'=|F_c^V|-|F_{\ell+1}^V|$, we obtain
\begin{align*}
b_{k-1}&\ge k\delta-(k-1)n-|U_1|+(\ell-1)|V_2|-\ell|F_c^V|-|F_{\ell+1}^V| \\
&\qquad -(\ell-2)(m'+|F_b^U|) \\
&\ge k\delta-(k-1)n-|U_1|+(\ell-1)|V_2|-(\ell+1)|F_c^V|-(\ell-2)|F_b^U| \\
&\qquad -(\ell-3)m'.
\end{align*}
Since $\ell \ge c$ and by the definition of $F_b^U$, $F_c^V$, $d_1$, $d_2$ and $m$ we have $|F_c^V| \ge m'$ and $|V_2| \ge 2d_2 + d_1 \ge 2|F_c^V | + |F_b^U| \ge |F_c^V | + |F_b^U| + m'$, we obtain
\begin{align*}
b_{k-1}&\ge k\delta-(k-1)n-|U_1|+(\ell-1)|V_2|-(\ell+1)|F_c^V|-(\ell-2)|F_b^U| \\
&\qquad -(\ell-3)m' \\
&= k\delta-(k-1)n-|U_1|+(\ell-2)(|V_2|-|F_c^V|-|F_b^U|-m') \\
&\qquad +|V_2|-3|F_c^V|+m' \\
&\ge k\delta-(k-1)n-|U_1|+(c-2)(|V_2|-|F_c^V|-|F_b^U|-m') \\
&\qquad +|V_2|-3|F_c^V|+m' \\
&\ge k\delta-(k-1)n-|U_1|+(c-1)|V_2|-(2c-1)|F_c^V|-(c-2)|F_b^U| \\
&\qquad+m'.
\end{align*}
Now by the definition of $t_2$ we have
\begin{align*}
b_{k-1}&\ge k\delta-(k-1)n-|U_1|+(c-1)|V_2|-(2c-1)|F_c^V|-(c-2)|F_b^U|+m' \\
&\ge (2c-1)(t_2-|F_c^V|) +  m'+ |F_b^U| \ge m'+|F_b^U|
\end{align*}
so we are indeed able to pick a vertex outside of $U_1\cup V_1\cup V_2\cup\left(\bigcup_{i=1}^{k-2}X_i\right)$ to extend $f$. This proves that copies of $K_{j+1}$ in $\widetilde{F}_{j-1}^V$ are all extended to copies of $K_{j+2}$ in step $j$ and so by induction $\widetilde{F}_j^V$ is a collection of $|F_c^V|$ vertex-disjoint cliques of order at least $j+2$ for each $j=c-2,\dots,k-1$. In particular, $\widetilde{F}_{k-1}^V$ is a collection of $|F_c^V|$ vertex-disjoint copies of $K_{k+1}$.

It remains to show that $\widetilde{F}_{k-1}^U\cup\widetilde{F}_{k-1}^V$ is a connected $K_{k+1}$-factor. Now $\widetilde{F}_{k-1}^U\cup\widetilde{F}_{k-1}^V$ consists of copies of $K_k$ in $G$ with either at least two vertices from $U_1$ and all other vertices from $\bigcup_{i=1}^{k-2}X_i$, or at least two vertices from $V_1$ and all other vertices from $\bigcup_{i=1}^{k-2}X_i$, so by~\ref{item:constr-parallel-sym-connected} the copies of $K_k$ in $\widetilde{F}_{k-1}^U\cup\widetilde{F}_{k-1}^V$ are pairwise $K_{k+1}$-connected. Hence, $\widetilde{F}_{k-1}^U\cup\widetilde{F}_{k-1}^V$ is a connected $K_{k+1}$-factor of size at least $(k+1)(|F_b^U|+|F_c^V|)$. If $F^V$ is empty then we have $|F_c^V|=0$, so $\widetilde{F}_{k-1}^U$ is a connected $K_{k+1}$-factor of size at least $(k+1)|F_b^U|$.
\end{proof}

Lemma~\ref{lem:high-min-deg-gen-intclique-clique-factor} is both the single partition analogue of and a straightforward consequence of Lemma~\ref{lem:high-min-deg-gen-//-intclique-clique-factor}. We will use it to find large connected $K_{k+1}$-factors when $\intr_k(G)$ contains a copy of $K_k$, specifically in Lemmas~\ref{lem:inductive-hangers-high-min-deg-base} and~\ref{lem:inductive-hangers-high-min-deg-layer-base}.

\begin{lemma} \label{lem:high-min-deg-gen-intclique-clique-factor}
Let $2\leq b\leq k$ be integers. Let $G$ be a graph on $n$ vertices with minimum degree $\delta=\delta(G)>\frac{(k-1)n}{k}$. Suppose there is a partition of $V(G)$ into vertex classes $U_1,U_2,X_1,\dots,X_{k-1},A$ such that
\begin{enumerate}[label=(\alph*)]
\item \label{item:constr-sym-no-edge} there are no edges between $U_1$ and $U_2$,
\item \label{item:constr-sym-connected} all copies of $K_k$ in $G$ with at least two vertices from $U_1$ and all other vertices from $\bigcup_{i=1}^{k-2}X_i$ are $K_{k+1}$-connected,
\item \label{item:constr-sym-size-bound} $|X_{i}|\leq n-\delta$ for $i\in[k-1]$, and
\item \label{item:constr-sym-ind-set}$X_{i}\cap\Gamma(g)$ is an independent set for each $(i,g)$ where $i\in[k-2]$ and $g$ is a clique of order at least $i$ with at least two vertices from $U_1$ and all other vertices from $\bigcup_{j=1}^{i-1}X_j$.
\end{enumerate}
Set $s_1:=\tfrac{k\delta-(k-1)n+(b-1)|U_2|-|A|}{2b-1}$. Let $F$ be a collection of vertex-disjoint copies of $K_b$ in $U_1$. Then $G$ contains a connected $K_{k+1}$-factor of size at least
\[(k+1)\min\left\{|F|,\left\lfloor\frac{|U_2|}{2}\right\rfloor,s_1\right\}.\]
\end{lemma}

\begin{proof}
Fix integers $2\leq b\leq k$ and set $c:=b$. Fix a graph $G$ and a partition of $V(G)$ with vertex classes $U_1,U_2,X_1,\dots,X_{k-1},A$ satisfying the lemma hypothesis. Define $V_1=X'=A'=F^V:=\nth$, $F^U:=F$ and $V_2:=V(G)\backslash\left(\bigcup_{i=1}^{k-2}X_i\right)$. Set $d_1:=|V_2|,d_2:=0$. Then the result follows by application of Lemma~\ref{lem:high-min-deg-gen-//-intclique-clique-factor}, noting that $|V_2|\geq|U_2|$ and $|V_1|=0$.
\end{proof}

We will find that Lemma~\ref{lem:high-min-deg-gen-intclique-clique-factor} is sometimes inadequate, especially when $\intr_k(G)$ does not contains a copy of $K_k$. This is partly due to the strength of conditions~\ref{item:constr-sym-connected} and~\ref{item:constr-sym-ind-set} forcing a large `bad' set $A$. The conditions are necessary when $b>2$, but we can weaken these conditions and sometimes do better when $b=2$. We present this as Lemma~\ref{lem:high-min-deg-gen-intclique-clique-factor-matching}. In this case, we require a smaller set of copies of $K_k$ in $G$ to be $K_{k+1}$-connected and $X_i\cap\Gamma(g)$ to be an independent set for a smaller set of copies $g$ of $K_{i+1}$ with $g\subseteq U_1\cup\left(\bigcup_{j=1}^{i-1}X_j\right)$.

\begin{lemma} \label{lem:high-min-deg-gen-intclique-clique-factor-matching}
Let $k\geq2$ be an integer. Let $G$ be a graph on $n$ vertices with minimum degree $\delta=\delta(G)>\frac{(k-1)n}{k}$. Suppose there is a partition of $V(G)$ into vertex classes $U_1,U_2$, $X_1,\dots,X_{k-1},A$ such that 
\begin{enumerate}[label=(\alph*)]
\item \label{item:constr-layers-no-edge} there are no edges between $U_1$ and $U_2$,
\item \label{item:constr-layers-connected} all copies of $K_k$ in $G$ comprising an edge of $G[U_1]$ and a vertex from each of $X_1,\dots,X_{k-2}$ are $K_{k+1}$-connected,
\item \label{item:constr-layers-size-bound} $|X_i|\leq n-\delta$ for $i\in[k-1]$, and
\item \label{item:constr-layers-ind-set} $X_i\cap\Gamma(g)$ is an independent set for each $(i,g)$ where $i\in[k-2]$ and $g$ is a copy of $K_{i+1}$ comprising an edge of $G[U_1]$ and a vertex from each of $X_1,\dots,X_{i-1}$. 
\end{enumerate}
Let $F$ be a matching in $U_1$. Set $q:=k\delta-(k-1)n+|U_2|-|U_1|-|A|$. Then $G$ contains a connected $K_{k+1}$-factor of size at least $(k+1)\min\left\lbrace|F|,q\right\rbrace$.
\end{lemma}

The proof approach is similar to that of Lemma~\ref{lem:high-min-deg-gen-//-intclique-clique-factor}; however, in this case we skip stage one and it turns out that we never fail to extend in stage two. Note that a copy of $K_k$ extending an edge from $F$ has at least $q$ common neighbours outside of both $U_1$ (which contains $F$) and a `bad' set $A$.

\begin{proof}
Let $\overline{F}\subseteq F$ satisfy $|\overline{F}|=\max\left\lbrace0,\min\left\lbrace|F|,q\right\rbrace\right\rbrace$. We will eventually extend each edge of $\overline{F}$ to a copy of $K_{k+1}$ using vertices in $X_1,\dots,X_{k-1}$. Note that the resultant copies of $K_{k+1}$ will be $K_{k+1}$-connected by~\ref{item:constr-layers-connected}.

We build up our desired connected $K_{k+1}$-factor step-by-step, starting with the aforementioned matching $\widetilde{F}_0:=\overline{F}$ in $U_1$. We have steps $j=1,\dots,k-1$. In step $j$ we extend each copy of $K_{j+1}$ in $\widetilde{F}_{j-1}$ to a copy of $K_{j+2}$ using $X_j$. For each copy of $K_{j+1}$ in $\widetilde{F}_{j-1}$ in turn we pick greedily a common neighbour in $X_j$ which is outside the previously chosen common neighbours to obtain a collection $\widetilde{F}_{j}$ of $|\overline{F}|$ vertex-disjoint copies of $K_{j+2}$. We claim that this is always possible for all $j\in[k-1]$. Observe that this holds trivially when $|\overline{F}|=0$, so in what follows it is enough to consider when $|\overline{F}|=\min\left\lbrace|F|,q\right\rbrace$.

Let $f$ be a copy of $K_{j+1}$ in $\widetilde{F}_{j-1}$. Note that $f$ has exactly one vertex in each $X_i$ for $i<j$, exactly two vertices in $U_1$ and none elsewhere. Let $v_1$ and $v_2$ be the vertices of $f$ in $U_1$, and let $\overline{v}_i$ be the vertex of $f$ in $X_i$ for each $i<j$. Every vertex $v$ of $U_1$ has at least $\delta-|A|-|U_1|-\sum_{h\neq j}\deg(v;X_h)$ neighbours in $X_j$, and for each $i<j$ the vertex $\overline{v}_i$ has at least $\delta-|A|-|U_2|-|U_1|-\sum_{h\neq j}\deg(\overline{v}_i;X_h)$ neighbours in $X_j$. By application of Lemma~\ref{lem:common-nbrhood-size} and noting that $|X_j|=n-|U_2|-|U_1|-|A|-\sum_{h\neq j}|X_h|$, the number of common neighhours of $f$ in $X_j$ is at least
\begin{align*}
a_j:= &\sum_{i=1}^2\left(\delta-|U_1|-|A|-\sum_{h\neq j}\deg(v_i;X_h)\right) \\
& + \sum_{i=1}^{j-1} \left(\delta-|U_2|-|U_1|-|A|-\sum_{h\neq j}\deg(\overline{v}_i;X_h)\right) \\
& - j\left(n-|U_2|-|U_1|-|A|-\sum_{h\neq j}|X_h|\right).
\end{align*}
Grouping terms together, we obtain
\begin{equation} \label{eqn:common-nbrhood-matching-grouped}
\begin{split}
a_j=&(j+1)\delta-jn+|U_2|-|U_1|-\sum_{h=1}^{j-1}\left(\sum_{v\in f}\deg(v;X_h)-j|X_h|\right) \\
& - \sum_{h=j+1}^{k-1}\left(\sum_{v\in f}\deg(v;X_h)-j|X_h|\right)-|A|.
\end{split}
\end{equation}
For $h\in[k-1]$, by applying Lemma~\ref{lem:common-nbrhood-size} to $X_h$ and $f$, we get
\begin{equation} \label{eqn:common-nbrhood-size-matching-X_h-f}
\sum_{v\in f}\deg(v;X_h)-j|X_h|\leq\deg(f;X_h).
\end{equation}
For $0 \le i < j$ let $f_i\in\widetilde{F}_i$ be the clique corresponding to $f$ right before step $i+1$, so $f_i=\{v_1,v_2\}\cup\{\overline{v}_h:h\in[i]\}$. Let $1 \le h < j$. Now $f_{h-1}$ is a clique of order $h+1$ comprising two vertices from $U_1$ and a vertex from each of $X_1,\dots,X_{h-1}$, so $\overline{v}_h$ has no neighbour in $\Gamma(f_{h-1};X_h)$ by~\ref{item:constr-layers-ind-set} applied with $(i,g) = (h,f_{h-1})$. Hence, we have $\deg(f_h;X_h)=0$ for all $h\in I$. Together with~\eqref{eqn:common-nbrhood-size-matching-X_h-f} and the fact that $\deg(f;X_h)\leq\deg(f_h;X_h)$ for all $h\in [j-1]$, we obtain
\begin{equation} \label{eqn:common-nbrhood-matching-used}
\sum_{h=1}^{j-1}\left(\sum_{v\in f} \deg(v;X_h)-j|X_h|\right)\leq\sum_{h=1}^{j-1}\deg(f_h;X_h)=0.
\end{equation}
By~\ref{item:constr-layers-size-bound} $|X_h|\leq n-\delta$ for $h\in[k-1]$, so by~\eqref{eqn:common-nbrhood-size-matching-X_h-f} we have
\begin{equation} \label{eqn:common-nbrhood-matching-future}
\sum_{h=j+1}^{k-1}\left(\sum_{v\in f}\deg(v;X_h)-j|X_h|\right)\leq\sum_{h=j+1}^{k-1}|X_h|\leq(k-j-1)(n-\delta).
\end{equation}
Putting together~\eqref{eqn:common-nbrhood-matching-grouped}, \eqref{eqn:common-nbrhood-matching-used} and~\eqref{eqn:common-nbrhood-matching-future}, we get $a_j \ge q \ge |\overline{F}|$, so we are indeed able to pick a vertex in $X_j$ to extend $f$. This proves that copies of $K_{j+1}$ in $\widetilde{F}_{j-1}$ are all extended to copies of $K_{j+2}$ in step $j$. Therefore, we terminate after step $k-1$ with a collection $\widetilde{F}_{k-1}$ of $|\overline{F}|$ vertex-disjoint copies of $K_{k+1}$ in $G$. All copies of $K_k$ in $G$ comprising an edge of $G[U_1]$ and a vertex from each of $X_1,\dots,X_{k-2}$ are $K_{k+1}$-connected by~\ref{item:constr-layers-connected}, so $\widetilde{F}_{k-1}$ is in fact a connected $K_{k+1}$-factor in $G$ of size at least $(k+1)\min\left\lbrace|F|,q\right\rbrace$.
\end{proof}

\section{The proof of Lemma~\ref{lem:general-cmpnt}} \label{section:stability-general-cmpnt}

In this section we provide a proof of Lemma~\ref{lem:general-cmpnt}, our stability result for graphs with at least two $K_{k+1}$-components where each $K_{k+1}$-component contains a copy of $K_{k+2}$. We start with a couple of preparatory lemmas which collect some observations about $K_{k+1}$-components.

The first lemma states that $K_{k+1}$-components cannot be too small, that there are no edges between the exteriors of different components and that certain spots in a $K_{k+1}$-component induce a graph with minimum degree $k\delta-(k-1)n$.

\begin{lemma} \label{lem:cmpnt-facts}
Let $k\in\NN$ and let $G$ be a graph on $n$ vertices with minimum degree $\delta(G)>\frac{(k-1)n}{k}$. Then
\begin{enumerate}[label=(\roman*)]
	\item \label{lem:cmpt-size} each $K_{k+1}$-component $C$ satisfies $|C|>\delta$,
    \item \label{lem:btwn-extr-no-edge} for distinct $K_{k+1}$-components $C$ and $C'$ there are no edges between $\extr(C)$ and $\extr(C')$,
    \item \label{lem:nbrhood-cmpt-size-min-deg} for each $K_{k+1}$-component $C$, each copy $u_1\dots u_{k-1}$ of $K_{k-1}$ of $C$, and $U=\{v:u_1$ $\dots u_{k-1}v\in C\}$, we have $\delta(G[U])\geq k\delta-(k-1)n$ and $|U|\geq k\delta-(k-1)n+1$.
\end{enumerate}
\end{lemma}

\begin{proof}
For~\ref{lem:cmpt-size} let $M$ be a maximal clique in $C$. Note that $|M|\geq k+1$. Count $\rho:=\sum_{m\in M,u\in V(G)}\mathbf{1}_{\{mu\in E(G)\}}$ in two ways. On the one hand, 
\[\rho=\sum_{m\in M}\sum_{u\in V(G)}\mathbf{1}_{\{mu\in E(G)\}} =\sum_{m\in M}\deg(m)\geq|M|\delta.\] 
On the other hand, noting that each vertex of $G$ which is not a vertex of $C$ is adjacent to at most $k-1$ vertices of $M$, while each vertex of $C$ is adjacent to at most $|M|-1$ vertices of $M$, we obtain
\begin{align*}
\rho&=\sum_{u\in V(G)}\sum_{m\in M}\mathbf{1}_{\{mu\in E(G)\}}=\sum_{u\in V(G)}\deg(u;M) \\
&\leq\sum_{u\in C}|M|-1 + \sum_{u\notin C}k-1=|C|(|M|-k)+(k-1)n
\end{align*}
and so $|M|\delta-(k-1)n\leq|C|(|M|-k)$. Since $(k-1)n<k\delta$ we conclude that $|C|>\delta$.

For~\ref{lem:btwn-extr-no-edge} suppose that $u$ is a vertex in $\extr(C)$, $v$ is a vertex in $\extr(C')$ and $uv$ is an edge in $G$. Apply Lemma~\ref{lem:clique-extn} to complete $uv$ to a copy of $K_k$ in $G$. Since this copy of $K_k$ contains a vertex from each of $\extr(C)$ and $\extr(C')$, it is in both $C$ and $C'$, which is a contradiction.

For \ref{lem:nbrhood-cmpt-size-min-deg} note that $U$ is non-empty as $u_1\dots u_{k-1}$ is a copy of $K_{k-1}$ of $C$. Let $u_k\in U$, so by definition $u_1\dots u_k\in C$. Since $\Gamma(u_1,\dots,u_k)\subseteq U$, by Lemma~\ref{lem:common-nbrhood-size} we have $\deg(u_k;U)=\deg(u_1,\dots,u_k)\geq k\delta-(k-1)n$. Now $\{u_k\}\cup\Gamma(u_k;U)\subseteq U$ so $|U|\geq k\delta-(k-1)n+1$.
\end{proof}

The next lemma says that graphs with more than one $K_{k+1}$-component have a non-empty $K_{k+1}$-interior and gives a lower bound on the size of said $K_{k+1}$-interior. This is an easy consequence of Lemma~\ref{lem:cmpnt-facts}\ref{lem:cmpt-size}.

\begin{lemma} \label{lem:non-empty-intr-extr}
Let $k\geq2$ be an integer. Let $G$ be a graph on $n$ vertices with minimum degree $\delta(G)=\delta>\frac{(k-1)n}{k}$ and more than one $K_{k+1}$-component. Then
\begin{enumerate}[label=(\roman*)]
\item \label{item:non-empty-intr} $|\intr_k(G)|\geq2\delta-n+2>0$, and
\item \label{item:non-empty-extr} for each $K_{k+1}$-component $C$ of $G$ we have $|\extr(C)|\leq n-\delta-1$.
\end{enumerate}
\end{lemma}
\begin{proof}
For~\ref{item:non-empty-intr}, let $C$ and $C'$ be distinct $K_{k+1}$-components of $G$. Lemma~\ref{lem:cmpnt-facts}\ref{lem:cmpt-size} tells us that $|C|,|C'|>\delta$. $\intr_k(G)$ contains all vertices which are vertices of both $C$ and $C'$ so $|\intr_k(G)|\geq|C_1|+|C_2|-n\geq2\delta-n+2>0$.

For~\ref{item:non-empty-extr}, let $C'$ be a $K_{k+1}$-component of $G$ distinct from $C$. Now $\extr(C)$ contains no vertex of $C'$ and by Lemma~\ref{lem:cmpnt-facts}\ref{lem:cmpt-size} we have $|C'|>\delta$, so it follows that $|\extr(C)|\leq n-\delta-1$.
\end{proof}

Central to our proof of Lemma~\ref{lem:general-cmpnt} is the construction of sufficiently large connected $K_{k+1}$-factors. Lemma~\ref{lem:cmpnt-facts}\ref{lem:nbrhood-cmpt-size-min-deg} enables us to find spots in a $K_{k+1}$-component which induce a graph with minimum degree $k\delta-(k-1)n$. In our proof of Lemma~\ref{lem:general-cmpnt}, we will often use this to find a large matching in such spots (this is possible due to Lemma~\ref{lem:min-deg-matching}\ref{lem:min-deg-matching-gen}). The family of configurations introduced in Section~\ref{subsection:configs}, the structural analysis in Section~\ref{subsection:configs-structure} and our construction procedures in Section~\ref{subsection:clique-factor-constr} will then enable us to extend such a matching to a connected $K_{k+1}$-factor.

As mentioned in Section~\ref{subsection:configs}, our proof of Lemma~\ref{lem:general-cmpnt} considers two cases -- when $\intr_k(G)$ contains a copy of $K_k$ and when $\intr_k(G)$ does not contains a copy of $K_k$. In the first case, we prove that if $\intr_k(G)$ contains a copy of $K_k$ then $\ckf_{k+1}(G)\geq\pp_k(n,\delta+\eta n)$. In fact, we prove the contrapositive statement in Lemma~\ref{lem:inductive-hangers-high-min-deg}, which involves proving that if $\ckf_{k+1}(G)<\pp_k(n,\delta+\eta n)$, then $G$ does not contain the configurations \config{k}{\ell}{j} for all $1\leq j<\ell\leq k$: it follows immediately from the definition of $\intr_k(G)$ that any copy of $K_k$ in $\intr_k(G)$ acts as the `central' copy of $K_k$ in an instance of the configuration \config{k}{k}{1}. We will use structural properties of these configurations proved in Section~\ref{subsection:configs-structure} and clique factor construction procedures from Section~\ref{subsection:clique-factor-constr} to do so.

\begin{lemma} \label{lem:inductive-hangers-high-min-deg}
Let $k\geq3$ be an integer and $\mu>0$. Let $\eta>0$ and $n\in\NN$ satisfy $\frac{1}{n}\ll\eta\ll\mu,\frac{1}{k}$. Let $G$ be a graph on $n$ vertices with minimum degree $\delta(G)\geq\delta\geq\left(\frac{k-1}{k}+\mu\right)n$ and at least two $K_{k+1}$-components. Suppose $\ckf_{k+1}(G)<\pp_k(n,\delta+\eta n)$. Then $G$ does not contain the configuration \config{k}{\ell}{j} for all $j,\ell$ such that $1\leq j<\ell\leq k$. In particular, $\intr_k(G)$ is $K_k$-free.
\end{lemma}

We prove Lemma~\ref{lem:inductive-hangers-high-min-deg} by induction on some function $f(j,\ell)$. While the specific choices of proof method and induction function are motivated by technical considerations, let us discuss the underlying ideas of our proof. We work in the context where $\intr_k(G)$ contains a copy of $K_k$, which is equivalent to $G$ containing \config{k}{k}{1} by definition, and we want to construct a sufficiently large connected $K_{k+1}$-factor. To this end, we seek increasingly structured graph configurations; this is achieved as a consequence of Lemma~\ref{lem:inductive-hangers-high-min-deg-layer-int}. Roughly speaking, the larger the interfaces between copies of $K_k$ in different $K_{k+1}$-components, the more highly structured the configuration. Eventually, we arrive at a configuration of the form \config{k}{\ell}{\ell-1}, which represents the `pinnacle of evolution' with copies of $K_k$ in different $K_{k+1}$-components that share a copy of $K_{k-1}$. These possess sufficient structure for the construction of a sufficiently large connected $K_{k+1}$-factor; we handle them in Lemmas~\ref{lem:inductive-hangers-high-min-deg-base} and~\ref{lem:inductive-hangers-high-min-deg-layer-base}. For technical reasons, we need treat \config{k}{2}{1} separately. We first consider the $j+1=\ell=2$ case.

\begin{lemma} \label{lem:inductive-hangers-high-min-deg-base}
Let $k\geq3$ be an integer and $\mu>0$. Let $\eta>0$ and $n\in\NN$ satisfy $\frac{1}{n}\ll\eta\ll\mu,\frac{1}{k}$. Let $G$ be a graph on $n$ vertices with minimum degree $\delta(G)\geq\delta\geq\left(\frac{k-1}{k}+\mu\right)n$ and at least two $K_{k+1}$-components. Suppose $\ckf_{k+1}(G)<\pp_k(n,\delta+\eta n)$. Then $G$ does not contain the configuration \config{k}{2}{1}.
\end{lemma}

\begin{proof}
Let $0<\eta<\min\{\frac{1}{1000k^2},\eta_0(k,\mu)\}$ and $n_1:=\max\{n_2(k,\mu,\eta),\frac{2}{\eta}\}$ with $\eta_0(k,\mu)$ and $n_2(k,\mu,\eta)$ given by Lemma~\ref{lem:extr-fn-bounds}. Suppose that $G$ contains the configuration \config{k}{2}{1}, so by Definition~\ref{defn:config} there are vertices $u_1,\dots,u_k,v_2,w_{2,1}$ in $V(G)$ such that~\ref{item:config-central}--\ref{item:config-dangle} hold. Say $f:=u_2\dots u_k$ lies in distinct $K_{k+1}$-components $C_1,\dots,C_p$ and $f':=u_1u_3\dots u_k$ lies in distinct $K_{k+1}$-components $C'_1,\dots,C'_q$ with $p,q\geq2$ and $f_0:=u_1\dots u_k\in C_1=C_1'$. Define
\[U_i=\{y:fy\in C_i\}\textrm{ for }i\in[p]\textrm{ and }V_j=\{y:f'y\in C_i'\}\textrm{ for }j\in[q],\]
so $\{U_i\}_{i\in[p]}$ and $\{V_j\}_{j\in[q]}$ partition $\Gamma(f)$ and $\Gamma(f')$ respectively. Note that
\begin{equation} \label{eqn:inductive-hangers-high-min-deg-base-disjoint-periphery}
U_i\cap V_j=\nth\quad\textrm{for all }(i,j)\in([p]\times[q])\setminus\{(1,1)\}.
\end{equation}
By Lemma~\ref{lem:cmpnt-facts}\ref{lem:nbrhood-cmpt-size-min-deg} we have 
\begin{equation} \label{eqn:inductive-hangers-high-min-deg-base-cmpnt-nbrhood-size}
|U_i|,|V_j| \ge k\delta-(k-1)n+1
\end{equation}
for all $i\in[p],j\in[q]$. Since we have $\deg(u_1,\dots,u_k) \ge k\delta-(k-1)n > 0$ by Lemma~\ref{lem:common-nbrhood-size}, we can pick a vertex $w\in\Gamma(f_0)\subseteq U_1\cap V_1$. Now $w$ has no neighbours in $\left(\bigcup_{1<i\leq p}U_i\right)\cup\left(\bigcup_{1<j\leq q}V_j\right)$, so by~\eqref{eqn:inductive-hangers-high-min-deg-base-cmpnt-nbrhood-size} we have
\begin{equation} \label{eqn:inductive-hangers-high-min-deg-base-small-periphery-total}
\delta \le \deg(w) < n - \sum_{1<i\leq p}|U_i| - \sum_{1<j\leq q}|V_j| \le n - 2(k\delta-(k-1)n+1)
\end{equation}
and we obtain $\delta\leq\frac{(2k-1)n-3}{2k+1}<\left(\frac{k}{k+1}-2\eta\right)n$. By Lemma~\ref{lem:common-nbrhood-size} we have
\begin{equation} \label{eqn:inductive-hangers-high-min-deg-base-nbrhood-size}
\begin{split}
|\Gamma(f)|&=\sum_{i\in[p]}|U_i|\geq(k-1)\delta-(k-2)n\quad\textrm{and} \\
|\Gamma(f')|&=\sum_{j\in[q]}|V_j|\geq(k-1)\delta-(k-2)n,
\end{split}
\end{equation}
so we obtain
\begin{equation} \label{eqn:inductive-hangers-high-min-deg-base-large-centre}
\begin{split}
|U_1|&=|\Gamma(f)|-\sum_{1<i\leq p}|U_i| \geByRef{eqn:inductive-hangers-high-min-deg-base-small-periphery-total} |\Gamma(f)|-(n-\delta-1) + \sum_{1<j\leq q}|V_j| \\
&\geByRef{eqn:inductive-hangers-high-min-deg-base-nbrhood-size} k\delta-(k-1)n+1 + \sum_{1<j\leq q}|V_j| \geByRef{eqn:inductive-hangers-high-min-deg-base-cmpnt-nbrhood-size} 2(k\delta-(k-1)n+1).
\end{split}
\end{equation}
By symmetry we have $|V_1|\geq 2(k\delta-(k-1)n+1)$. We have $p,q\geq2$, so $U_2,V_2\neq\nth$ and we can pick $u\in U_2$ and $v\in V_2$. Now $u$ and $v$ have no neighbours in $U_1$ and $V_1$ respectively, so we conclude that
\begin{equation} \label{eqn:inductive-hangers-high-min-deg-base-centre-non-nbr-bound}
|U_1|,|V_1|<n-\delta.
\end{equation}

We now define 
\begin{align*}
&X_{k-1},Y_{k-1}:=V(G)\setminus\Gamma(u_k), \\
&X_i,Y_i:=\Gamma(u_{i+2},\dots,u_k)\setminus\Gamma(u_{i+1})\textrm{ for }i\in[k-2]\setminus\{1\}, \\
&X_1,X'_1:=\Gamma(u_3,\dots,u_k)\setminus\Gamma(u_2),\quad Y_1,Y'_1:=\Gamma(u_3,\dots,u_k)\setminus\Gamma(u_1), \\
&X'_i:=\Gamma(u_2,\dots,u_i,u_{i+2},\dots,u_k)\setminus\Gamma(u_{i+1})\textrm{ for }i\in[k-1]\setminus\{1\}, \\
&Y'_i:=\Gamma(u_1,u_3,\dots,u_i,u_{i+2},\dots,u_k)\setminus\Gamma(u_{i+1})\textrm{ for }i\in[k-1]\setminus\{1\},\\
&Z'_i:=X'_{i+1}\cap Y'_{i+1}; Z''_i:=Z'_i\cap\Gamma(w)\textrm{ for }i\in[k-2]; \\
&A := \bigcup_{i=1}^{k-1}(X_i\setminus X'_i) = \bigcup_{i=2}^{k-1}(X_i\setminus X'_i), A' := \bigcup_{i=1}^{k-1}(Y_i\setminus Y'_i) = \bigcup_{i=2}^{k-1}(Y_i\setminus Y'_i), \\
&A'':=(\bigcup_{i=1}^{k-2}Z'_i)\setminus\Gamma(w); \quad B:=A\cup A'\cup A''.
\end{align*}
Note that $A$ is the set of vertices in $G$ with at least two non-neighbours in $f$. Count $\rho:=\sum_{v\in V(G),u\in f}\mathbf{1}_{\{vu\notin E(G)\}}$ in two ways. On the one hand, 
\[\rho=\sum_{u\in f}\left(\sum_{v\in V(G)}\mathbf{1}_{\{vu\notin E(G)\}}\right)=\sum_{u\in f}|V(G)\setminus\Gamma(u)|\leq(k-1)(n-\delta).\]
On the other hand, we have
\[\rho=\sum_{v\in V(G)}\left(\sum_{u\in f}\mathbf{1}_{\{vu\notin E(G)\}}\right)=\sum_{v\in V(G)}|f\setminus\Gamma(v)|\geq n-|\Gamma(f)|+|A|.\]
Hence, by~\eqref{eqn:inductive-hangers-high-min-deg-base-nbrhood-size} we obtain $|A|\leq\sum_{i\in[p]}|U_i|-n+(k-1)(n-\delta)$. Similarly, $A'$ is the set of vertices in $G$ with at least two non-neighbours in $f'$. Hence, $|A'|\leq\sum_{j\in[q]}|V_j|-n+(k-1)(n-\delta)$. No vertex in $\left(\bigcup_{1<i\leq p}U_i\right)\cup\left(\bigcup_{1<j\leq q}V_j\right)\cup A''$ is adjacent to $w$, so $|A''|\leq n-\delta-1-\sum_{1<i\leq p}|U_i|-\sum_{1<j\leq q}|V_j|$. Therefore, we conclude that
\begin{equation} \label{eqn:inductive-hangers-high-min-deg-base-bad-set-bound}
|B|\leq|U_1|+|V_1|-2[k\delta-(k-1)n+1]-(n-\delta-1).
\end{equation}

Let $1 < h \le p$. Lemma~\ref{lem:cmpnt-facts}\ref{lem:nbrhood-cmpt-size-min-deg} tells us that $\delta(G[U_h])\geq k\delta-(k-1)n$, so we have a matching $M$ in $U_h$ with $|M|=\min\{k\delta-(k-1)n,\left\lfloor\frac{|U_h|}{2}\right\rfloor\}$ by Lemma~\ref{lem:min-deg-matching}\ref{lem:min-deg-matching-gen}. We check the conditions to apply Lemma~\ref{lem:high-min-deg-gen-intclique-clique-factor} for $b=2$ with $U_h,\bigcup_{i\neq h}U_i,Z''_1,\dots,Z''_{k-2},X_1,B$ and $M$ as $U_1,U_2,X_1,\dots,X_{k-1},A$ and $F$ respectively. By definition $U_h$ and $\bigcup_{i\neq h}U_i$ partition $\Gamma(f)$. For each $i\in[k-2]$ the set $Z''_i$ consists of the neighbours of $w$ whose only non-neighbour in $f_0$ is $u_{i+2}$. The set $X_1$ consists of the vertices whose only non-neighbour in $f$ is $u_2$. The set $B$ consists of the non-neighbours of $w$ whose only non-neighbour in $f_0$ is $u_{i+2}$ for some $i\in[k-2]$ and the vertices with at least two non-neighbours in $f$ or at least two non-neighbours in $f'$. Hence, $U_h,\bigcup_{i\neq h}U_i,Z''_1,\dots,Z''_{k-2},X_1,B$ form a partition of $V(G)$ such that there are no edges between $U_h$ and $\bigcup_{i\neq h}U_i$. Note that $Z''_i\subseteq V(G)\setminus\Gamma(u_{i+2})$ for $i\in[k-2]$ and $X_1\subseteq V(G)\setminus\Gamma(u_2)$, so $|Z''_i|\leq n-\delta$ for each $i\in[k-2]$ and $|X_1|\leq n-\delta$. For each $(e,i)\in E(G[U_h])\times[k-2]$, by applying Lemma~\ref{lem:high-min-deg-k-intclique-edge-int-nbrhood-indset} for $i+1$ with $u_1,\dots,u_k,w$ as themselves, $C_h$ as $C_1$, $C_1$ as $C_2$ and $e$ as $uv$, we have that $Z''_i\cap\Gamma(e)$ is an independent set. Furthermore, all copies of $K_k$ in $G$ with at least two vertices from $U_h$ and all other vertices from $\left(\bigcup_{i=1}^{k-2}Z''_i\right)$ are $K_{k+1}$-connected: we can construct a $K_{k+1}$-walk from such a copy $g$ of $K_k$ to $f_0$ by a step-by-step vertex replacement of the vertices of $g$ with the vertices of $f_0$.

Since the requisite conditions are satisfied, we apply Lemma~\ref{lem:high-min-deg-gen-intclique-clique-factor} for $b=2$ with $U_h,\bigcup_{i\neq h}U_i,Z''_1$, $\dots,Z''_{k-2},X_1,B$ and $M$ as $U_1,U_2,X_1,\dots,X_{k-1},A$ and $F$ respectively; since $\sum_{i\neq h}|U_i| \ge |U_1|$ and by noting~\eqref{eqn:inductive-hangers-high-min-deg-base-large-centre}, \eqref{eqn:inductive-hangers-high-min-deg-base-centre-non-nbr-bound} and~\eqref{eqn:inductive-hangers-high-min-deg-base-bad-set-bound}, we obtain that $\ckf_{k+1}(G)$ is at least
\begin{align*}
&(k+1)\min\left\lbrace k\delta-(k-1)n,\left\lfloor\tfrac{|U_h|}{2}\right\rfloor,\left\lfloor\tfrac{\sum_{i\neq h}|U_i|}{2}\right\rfloor,\tfrac{k\delta-(k-1)n+\sum_{i\neq h}|U_i|-|B|}{3}\right\rbrace \\
&\ge (k+1)\min\left\lbrace k\delta-(k-1)n,\left\lfloor\tfrac{|U_h|}{2}\right\rfloor\right\rbrace.
\end{align*}
Since~\eqref{eqn:extr-approx-bound} and~\eqref{eqn:spacious-bound} hold and $\ckf_{k+1}(G)<\pp_k(n,\delta+\eta n)$, we deduce that
\begin{equation} \label{eqn:inductive-hangers-high-min-deg-base-small-periphery-bound}
|U_h|<2(k\delta-(k-1)n)\textrm{ and }|U_h|<\tfrac{(k-1)(\delta+3\eta n)-(k-2)n}{r'},
\end{equation}
where $r':=r_p(n,\delta+\eta n)$. If furthermore $\delta\in\left[\left(\frac{k-1}{k}+\mu\right)n,\left(\frac{2k-1}{2k+1}-2\eta\right)n\right]$, by~\eqref{eqn:three-quarters-bound} we also deduce that
\begin{equation} \label{eqn:inductive-hangers-high-min-deg-base-small-periphery-bound-tighter}
|U_h|<\tfrac{3}{2}(k\delta-(k-1)n)+1.
\end{equation}
By symmetry we also have that for all $1<j\leq q$,
\begin{equation} \label{eqn:inductive-hangers-high-min-deg-base-small-periphery-bound-V}
|V_j|<2(k\delta-(k-1)n)\textrm{ and }|V_j|<\tfrac{(k-1)(\delta+3\eta n)-(k-2)n}{r'},
\end{equation}
and if furthermore $\delta\in\left[\left(\frac{k-1}{k}+\mu\right)n,\left(\frac{2k-1}{2k+1}-2\eta\right)n\right]$ then
\begin{equation} \label{eqn:inductive-hangers-high-min-deg-base-small-periphery-bound-V-tighter}
|V_j|<\tfrac{3}{2}(k\delta-(k-1)n)+1.
\end{equation}

Let $1 < i \le p$ and $1 < j \le q$. By Lemma~\ref{lem:min-deg-matching}\ref{lem:min-deg-matching-gen} and Lemma~\ref{lem:cmpnt-facts}\ref{lem:nbrhood-cmpt-size-min-deg} and noting~\eqref{eqn:inductive-hangers-high-min-deg-base-small-periphery-bound} and~\eqref{eqn:inductive-hangers-high-min-deg-base-small-periphery-bound-V}, there are matchings $M_u$ and $M_v$ in $U_i$ and $V_j$ respectively with $|M_u|=\min\{\left\lfloor\frac{|U_i|}{2}\right\rfloor,k\delta-(k-1)n\}=\left\lfloor\frac{|U_i|}{2}\right\rfloor$ and $|M_v|=\min\{\left\lfloor\frac{|V_j|}{2}\right\rfloor,k\delta-(k-1)n\}=\left\lfloor\frac{|V_j|}{2}\right\rfloor$. Without loss of generality, suppose $|V_j| \ge|U_i|$. Let $uv$ be an edge in $U_i$. Note that $u$ and $v$ each has at most $n-\delta-1-\sum_{h\neq i}|U_h| \leByRef{eqn:inductive-hangers-high-min-deg-base-nbrhood-size} |U_i|-(k\delta-(k-1)n+1)$ non-neighbours outside of $\left(\bigcup_{h\neq i}U_h\right)\cup\{u,v\}$. Hence, $\Gamma(u,v)$ has at most $2[|U_i|-(k\delta-(k-1)n+1)]$ non-neighbours outside of $\left(\bigcup_{h\neq i}U_h\right)\cup\{u,v\}$. Since we have $|V_j| \ge|U_i|$ and by~\eqref{eqn:inductive-hangers-high-min-deg-base-disjoint-periphery} $V_j$ is disjoint from $\left(\bigcup_{h\neq i}U_h\right)\cup\{u,v\}$, by~\eqref{eqn:inductive-hangers-high-min-deg-base-small-periphery-bound} we obtain
\begin{equation} \label{eqn:inductive-hangers-high-min-deg-base-V_j-common-nbrhood}
|\Gamma(u,v;V_j)| \ge |V_j| - 2[|U_i|-(k\delta-(k-1)n+1)] > 0.
\end{equation}
Hence, we may pick $x\in\Gamma(u,v;V_j)$. Suppose $\Gamma(u,v;V_j)$ is an independent set. Now $x$ has no neighbour in $\left(\bigcup_{h\neq j}V_h\right)\cup\Gamma(u,v;V_j)$, so we have $n-\delta \ge \sum_{h\neq j}|V_h|+|\Gamma(u,v;V_j)| \ge |\Gamma(f')|- 2\left[|U_i|-(k\delta-(k-1)n+1)\right]$ by~\eqref{eqn:inductive-hangers-high-min-deg-base-nbrhood-size} and~\eqref{eqn:inductive-hangers-high-min-deg-base-V_j-common-nbrhood}. Hence, we obtain
\begin{equation} \label{eqn:inductive-hangers-high-min-deg-base-large-periphery}
|V_j| \ge |U_i| \ge k\delta-(k-1)n+1+\tfrac{|\Gamma(f')|-(n-\delta)}{2} \geByRef{eqn:inductive-hangers-high-min-deg-base-nbrhood-size} \tfrac{3}{2}(k\delta-(k-1)n)+1.
\end{equation}
Note that $U_i\cap V_j=\nth$ and $w$ has no neighbours in $U_i\cup V_j$, so $n-\delta>|U_i\cup V_j|\geq3[k\delta-(k-1)n+1]$, which implies $\delta\leq\frac{(3k-2)n-4}{3k+1}<\left(\frac{2k-1}{2k+1}-2\eta\right)n$. However, this means that~\eqref{eqn:inductive-hangers-high-min-deg-base-large-periphery} contradicts~\eqref{eqn:inductive-hangers-high-min-deg-base-small-periphery-bound-tighter}. Therefore, there is an edge $u'v'$ in $\Gamma(u,v;V_j)$. Then $uvu'v'u_3\dots u_k$ is a copy of $K_{k+2}$ with $uvu_3\dots u_k\in C_i$ and $u'v'u_3\dots u_k\in C'_j$ so $C_i=C'_j$. Noting that $i\in[p]\setminus\{1\}$ and $j\in[q]\setminus\{1\}$ are arbitrary, we deduce that in fact $p=q=2$. 

In what follows, we check the conditions to apply Lemma~\ref{lem:high-min-deg-gen-//-intclique-clique-factor} for $b=c=2$ with $U_2,U_1,V_2,V_1$, $Z''_1,\dots,Z''_{k-2},X_1,Y_1,B,B,M_u$ and $M_v$ as the required inputs $U_1,U_2,V_1,V_2,X_1,\dots,X_{k-1},X',A$, $A',F^U$ and $F^V$ respectively. We know from earlier that $U_2,U_1,Z''_1,\dots,Z''_{k-2},X_1,B$ form a partition of $V(G)$ such that there are no edges between $U_2$ and $U_1$, that $|Z''_i|\leq n-\delta$ for each $i\in[k-2]$ and $|X_1|\leq n-\delta$, that $Z''_i\cap\Gamma(e)$ is an independent set for each $(e,i)\in E(G[U_2])\times[k-2]$ and that all copies of $K_k$ in $G$ with at least two vertices from $U_2$ and all other vertices from $\left(\bigcup_{i=1}^{k-2}Z''_i\right)$ are $K_{k+1}$-connected. By swapping the roles of $u_1$ and $u_2$, we also have that $V_2,V_1,Z''_1,\dots,Z''_{k-2},Y_1,B$ form a second partition of $V(G)$ such that there are no edges between $V_1$ and $V_2$, that $|Y_1|\leq n-\delta$, that $Z''_i\cap\Gamma(e)$ is an independent set for each $(e,i)\in E(G[V_2])\times[k-2]$ and that all copies of $K_k$ in $G$ with at least two vertices from $V_2$ and all other vertices from $\left(\bigcup_{i=1}^{k-2}Z''_i\right)$ are $K_{k+1}$-connected.

Since the conditions are satisfied, we apply Lemma~\ref{lem:high-min-deg-gen-//-intclique-clique-factor} with the given inputs and $d_1=d_2=\left\lfloor\tfrac{|V_1|}{3}\right\rfloor$ to obtain that $\ckf_{k+1}(G)$ is at least
\begin{align*}
&(k+1)\min\bigg\lbrace\left\lfloor\tfrac{|U_2|}{2}\right\rfloor,\left\lfloor\tfrac{|U_1|}{2}\right\rfloor,\left\lfloor\tfrac{|V_1|}{3}\right\rfloor,\tfrac{k\delta-(k-1)n-|B|+|U_1|}{3},\tfrac{k\delta-(k-1)n+|U_1|-|V_2|}{3}\bigg\rbrace \\
&+(k+1)\min\bigg\lbrace\left\lfloor\tfrac{|V_2|}{2}\right\rfloor,\left\lfloor\tfrac{|V_1|}{3}\right\rfloor,\tfrac{k\delta-(k-1)n-|B|+|V_1|-|U_2|/2}{3},\tfrac{k\delta-(k-1)n+|V_1|-3|U_2|/2}{3}\bigg\rbrace.
\end{align*}
By~\eqref{eqn:inductive-hangers-high-min-deg-base-small-periphery-total}, \eqref{eqn:inductive-hangers-high-min-deg-base-nbrhood-size}, \eqref{eqn:inductive-hangers-high-min-deg-base-large-centre}, \eqref{eqn:inductive-hangers-high-min-deg-base-centre-non-nbr-bound}, \eqref{eqn:inductive-hangers-high-min-deg-base-bad-set-bound} and~\eqref{eqn:inductive-hangers-high-min-deg-base-small-periphery-bound}, this is at least
\begin{equation*}
(k+1)\left(\min\left\lbrace\left\lfloor\tfrac{|U_2|}{2}\right\rfloor,\tfrac{2(k\delta-(k-1)n)}{3}\right\rbrace+\min\left\lbrace\left\lfloor\tfrac{|V_2|}{2}\right\rfloor,\tfrac{2(k\delta-(k-1)n)}{3}-\tfrac{|U_2|}{6}\right\rbrace\right).
\end{equation*}
Now by Lemma~\ref{lem:cmpnt-facts}\ref{lem:nbrhood-cmpt-size-min-deg} we have $\left\lfloor\frac{|U_2|}{2}\right\rfloor,\left\lfloor\frac{|V_2|}{2}\right\rfloor\geq\frac{(k\delta-(k-1)n}{2}$ and we have~\eqref{eqn:inductive-hangers-high-min-deg-base-small-periphery-bound}, so in fact it is at least $(k+1)(k\delta-(k-1)n)\geq\pp_k(n,\delta+\eta n)$ by~\eqref{eqn:spacious-bound}. However, this is a contradiction so $G$ does not contain \config{k}{2}{1}.
\end{proof}

Next, we consider the $3\leq j+1=\ell\leq k$ case.

\begin{lemma} \label{lem:inductive-hangers-high-min-deg-layer-base}
Let $k,\ell$ be integers satisfying $3\leq\ell\leq k$ and let $\mu>0$. Let $\eta>0$ and $n\in\NN$ satisfy $\frac{1}{n}\ll\eta\ll\mu,\frac{1}{k}$. Let $G$ be a graph on $n$ vertices with minimum degree $\delta(G)\geq\delta\geq\left(\frac{k-1}{k}+\mu\right)n$ and at least two $K_{k+1}$-components. Suppose $\ckf_{k+1}(G)<\pp_k(n,\delta+\eta n)$ and $G$ does not contain \config{k}{q}{p} for all $p<q<\ell$. Then $G$ does not contain the configuration \config{k}{\ell}{\ell-1}.
\end{lemma}

\begin{proof}
Let $0<\eta<\min\{\frac{1}{1000k^2},\eta_0(k,\mu)\}$ and $n_1:=\max\{n_2(k,\mu,\eta),\frac{2}{\eta}\}$ with $\eta_0(k,\mu)$ and $n_2(k,\mu,\eta)$ given by Lemma~\ref{lem:extr-fn-bounds}. Suppose that $G$ contains the configuration \config{k}{\ell}{\ell-1}, so by Definition~\ref{defn:config} there are vertices $u_1,\dots,u_k,v_\ell$, $w_{\ell,1},\dots,w_{\ell,\ell-1}$ in $V(G)$ such that~\ref{item:config-central}--\ref{item:config-dangle} hold. Set $f_i:=u_1\dots u_{i-1}u_{i+1}\dots u_{\ell-1}u_{\ell+1}\dots u_k$ for each $i\in[\ell-1]$. Let $f:=u_1\dots u_{\ell-1}u_{\ell+1}\dots u_k$. Observe that $u_{\ell}$ and $v_{\ell}$ are distinct vertices: if not, $fu_{\ell}=fv_{\ell}$ would be a copy of $K_k$ in two different $K_{k+1}$-components, giving a contradiction. Furthermore, $v_\ell u_\ell$ is not an edge: if not, $fu_{\ell}v_{\ell}$ would be a copy of $K_{k+1}$ in $G$ where $fu_{\ell}$ and $fv_{\ell}$ would belong to different $K_{k+1}$-components, giving a contradiction. Hence, $u_1,\dots,u_k,v_\ell, w_{\ell,1},\dots,w_{\ell,\ell-1}$ are all distinct vertices. Set $X_i:=\Gamma(f_i)\setminus\{u_i\}$ for $i\in[\ell-1]$, $X:=\bigcup_{i=1}^{\ell-1}X_i$ and $Y_j:=\Gamma(u_\ell,v_\ell,w_{\ell,j};X)$ for $j\in[\ell-1]$.

We claim that $Y_j = \nth$ for some $j\in[\ell-1]$. Indeed, suppose that $Y_j \neq \nth$ for all $j\in[\ell-1]$. Pick $y_j\in Y_j$ for each $j\in[\ell-1]$. Fix a function $\phi\colon[\ell-1]\to[\ell-1]$ such that $y_j\in X_{\phi(j)}$. Observe that $y_jf_{\phi(j)}u_\ell\in C$ for each $j\in[\ell-1]$: if not, then $fu_\ell\in C$, $fv_\ell\notin C$ and $y_jf_{\phi(j)}u_\ell\notin C$ would yield \config{k}{2}{1} with $fu_\ell$ as the `central' copy of $K_k$ and $f_{\phi(j)}$ as the common vertices. Similarly, we have $y_jf_{\phi(j)}v_\ell\notin C$ for each $j\in[\ell-1]$. Now for each $j\in[\ell-1]$ apply Lemma~\ref{lem:clique-extn} to complete $u_\ell\dots u_kw_{\ell,j}y_j$ to a copy $D_j:=u_\ell\dots u_kw_{\ell,j}y_jy_{j,1}\dots y_{j,\ell-3}$ of $K_k$. Observe that $D_j\in C$ for each $j\in[\ell-1]$: if not, then $y_jf_{\phi(j)}v_\ell\notin C$, $y_jf_{\phi(j)}u_\ell\in C$ and $D_j\notin C$ would yield \config{k}{\ell-1}{\ell-2} with $y_jf_{\phi(j)}u_\ell$ as the `central' copy of $K_k$, $y_ju_{\ell+1}\dots u_k$ as the common vertices and $D_j$ `dangling off' $u_{\ell}$. But now $D_j\in C$ for $j\in[\ell-1]$ with $u_\ell\dots u_k w_{\ell,1}\dots w_{\ell,\ell-1}\notin C$ as the `central' copy of $K_k$ yields \config{k}{\ell-1}{1} with $u_{\ell}\dots u_k$ as the common vertices, giving a contradiction. Hence, $Y_j$ is empty for some $j\in[\ell-1]$.

Pick $j\in[\ell-1]$ such that $Y_j = \nth$, which exists by the claim above. Apply Lemma~\ref{lem:common-nbrhood-size} with $U=V(G)\setminus\{u_1,\dots,u_{\ell-1},u_{\ell+1},\dots,u_k\}$ to obtain
\begin{equation}
\label{eqn:inductive-hangers-high-min-deg-layer-base-petal-size}
|X_i|\geq(k-2)(\delta-k+2)-(k-3)(n-k+1)=(k-2)\delta-(k-3)n-1
\end{equation}
for each $i\in[\ell-1]$. Since $X_h\cap X_i=\Gamma(f)$ for all $\{h,i\}\in\binom{[\ell-1]}{2}$, we have
\begin{equation}
\label{eqn:inductive-hangers-high-min-deg-layer-base-sunflower}
|X|=\sum_{i=1}^{\ell-1}|X_i|-(\ell-2)|\Gamma(f)|.
\end{equation}
We claim that $w_{\ell,j}\notin X$. Indeed, suppose that $w_{\ell,j}\in X$. Without loss of generality, $w_{\ell,j}\in X_1$. Observe that $w_{\ell,j}f_1u_\ell\in C$: if not, then $fu_\ell\in C$, $fv_\ell\notin C$ and $w_{\ell,j}f_1u_\ell\notin C$ would yield \config{k}{2}{1} with $fu_\ell$ as the `central' copy of $K_k$ and $f_1$ as the common vertices. Similarly, we have $w_{\ell,j}f_1v_\ell\notin C$. But now $w_{\ell,j}f_1v_\ell\notin C$, $w_{\ell,j}f_1u_\ell\in C$ and $u_\ell\dots u_k w_{\ell,1}\dots w_{\ell,\ell-1}\notin C$ yields \config{k}{\ell-1}{\ell-2} with $w_{\ell,j}f_1u_\ell$ as the `central' copy of $K_k$, $u_{\ell+1}\dots u_kw_{\ell,j}$ as the common vertices and $u_\ell\dots u_k w_{\ell,1}\dots w_{\ell,\ell-1}$ `dangling off' $u_{\ell}$, giving a contradiction. Now apply Lemma~\ref{lem:common-nbrhood-size} with $U=X\backslash\{u_\ell,v_\ell\}$ to obtain
\begin{equation}
\label{eqn:inductive-hangers-high-min-deg-layer-base-sunflower-nbrhood}
\begin{split}
|Y_j|&\geq2(\delta-n+|X|)+(\delta-n+|X|-1)-2(|X|-2) \\
&=|X|-3(n-\delta-1).
\end{split}
\end{equation}

Denote by $C'$ the $K_{k+1}$-component of $G$ which contains $fv_\ell$. Define
\begin{align*}
W_1&:=\{u\in\Gamma(f):uf\in C\},\quad W_2:=\{u\in\Gamma(f):uf\in C'\}, \\
W_3&:=\{u\in\Gamma(f):uf\notin C,C'\}.
\end{align*}
Since $u_\ell\in W_1$ and $v_\ell\in W_2$, we have $W_1,W_2\neq\nth$. Moreover, we have $\Gamma(u_1,\dots, u_k)\subseteq W_1$ and $\Gamma(u_1,\dots,u_{\ell-1}, u_{\ell+1},\dots,u_k,v_\ell)\subseteq W_2$. Let $w_1\in W_1$ and $w_2\in W_2$. Note that $w_1$ has no neighbour in $W_2\cup W_3$ and $w_2$ has no neighbour in $W_1\cup W_3$, so
\begin{equation}
\label{eqn:inductive-hangers-high-min-deg-layer-base-non-nbr-bound}
|W_1\cup W_3|,|W_2\cup W_3|\leq n-\delta-1.
\end{equation}
Since $Y_j$ is empty and~\eqref{eqn:inductive-hangers-high-min-deg-layer-base-petal-size}, \eqref{eqn:inductive-hangers-high-min-deg-layer-base-sunflower} and~\eqref{eqn:inductive-hangers-high-min-deg-layer-base-sunflower-nbrhood} hold, we obtain
\[0=|Y_j|\geq(\ell-1)((k-2)\delta-(k-3)n-1)-(\ell-2)|\Gamma(f)|-3(n-\delta-1).\]
By rearrangement, we obtain
\begin{equation}
\label{eqn:inductive-hangers-high-min-deg-layer-base-large-nbrhood}
|\Gamma(f)|=\sum_{i\in[3]}|W_i|\geq(k-2)\delta-(k-3)n-1+\tfrac{(k+1)\delta-kn+2}{\ell-2}.
\end{equation}
By~\eqref{eqn:inductive-hangers-high-min-deg-layer-base-non-nbr-bound} and~\eqref{eqn:inductive-hangers-high-min-deg-layer-base-large-nbrhood}, we have
\begin{align*}
(k-1)\delta-(k-2)n+\tfrac{(k+1)\delta-kn+2}{\ell-2} \le |W_1|,|W_2| \le n-\delta-1.
\end{align*}
Hence, $\delta\leq\frac{[(\ell-1)(k-1)+1]n-\ell}{(\ell-1)k+1}\leq\frac{(2k-1)n-3}{2k+1}<\left(\frac{k}{k+1}-2\eta\right)n$. By multiplying both sides of the first upper bound on $\delta$ by $\ell-3$ and rearranging, we obtain
\[n-\delta-1 + \tfrac{(k+1)\delta-kn+2}{\ell-2} \ge (\ell-2)(k\delta-(k-1)n+1).\]
Recalling~\eqref{eqn:inductive-hangers-high-min-deg-layer-base-large-nbrhood} and $\ell \ge 3$, we obtain
\begin{equation} \label{eqn:inductive-hangers-high-min-deg-layer-base-large-nbrhood-two}
\begin{split}
|\Gamma(f)|&\geq(k-1)\delta-(k-2)n+(\ell-2)(k\delta-(k-1)n+1) \\
&\geq(k-1)\delta-(k-2)n+k\delta-(k-1)n+1
\end{split}
\end{equation}
and
\begin{equation} \label{eqn:inductive-hangers-high-min-deg-layer-base-large-sets}
|W_1|,|W_2|\geq(\ell-1)[k\delta-(k-1)n+1]\geq2[k\delta-(k-1)n+1].
\end{equation}

Now pick a vertex $w\in W_2$ and define 
\begin{align*}
&Z_i:=\Gamma(u_{i+2},\dots,u_k)\setminus\Gamma(u_{i+1})\textrm{ for }\ell\leq i\leq k-1,\\
&Z_i:=\Gamma(u_{i+1},\dots,u_{\ell-1},u_{\ell+1}\dots,u_k)\setminus\Gamma(u_i)\textrm{ for }i\in[\ell-1];\\
&Z'_i:=\Gamma(u_1,\dots,u_{i-1},u_{i+1},\dots,u_{\ell-1},u_{\ell+1},\dots,u_k)\setminus\Gamma(u_i)\textrm{ for }i\in[\ell-1],\\
&Z'_i:=\Gamma(u_1,\dots,u_{\ell-1},u_{\ell+1}\dots,u_{i},u_{i+2},\dots,u_k)\setminus\Gamma(u_{i+1})\textrm{ for }\ell\leq i\leq k-1;\\
&Z''_i:=Z'_i\cap\Gamma(w)\textrm{ for }i\in[k-1]; \\
&A_1:=\bigcup_{i=1}^{k-1}(Z_i\setminus Z'_i),A_2:=\left(\bigcup_{i=1}^{k-1}Z'_i\right)\setminus\Gamma(w),A:=A_1\cup A_2.
\end{align*}
Note that $|A_1|$ is the number of vertices in $G$ with at least two non-neighbours in $f$. Count $\rho:=\sum_{v\in V(G),u\in f}\mathbf{1}_{\{vu\notin E(G)\}}$ in two ways. On the one hand,
\[\rho=\sum_{u\in f}\left(\sum_{v\in V(G)}\mathbf{1}_{\{vu\notin E(G)\}}\right)=\sum_{u\in f}|V(G)\setminus\Gamma(u)|\leq(k-1)(n-\delta).\]
On the other hand, 
\[\rho=\sum_{v\in V(G)}\left(\sum_{u\in f}\mathbf{1}_{\{vu\notin E(G)\}}\right)=\sum_{v\in V(G)}|f\setminus\Gamma(v)|\geq n-|\Gamma(f)|+|A_1|.\]
Hence, we have $|A_1|\leq|\Gamma(f)|-n+(k-1)(n-\delta)$. No vertex in $W_1\cup W_3\cup A_2$ is adjacent to $w$, so $|A_2|\leq n-\delta-1-|W_1|-|W_3|$. Therefore, noting that $|\Gamma(f)|=\sum_{i\in[3]}|W_i|$, we obtain
\begin{equation} \label{eqn:inductive-hangers-high-min-deg-layer-base-bad-set-bound}
|A|\leq|W_2|-[k\delta-(k-1)n+1].
\end{equation}

Lemma~\ref{lem:cmpnt-facts}\ref{lem:nbrhood-cmpt-size-min-deg} tells us that $\delta(G[W_1])\geq k\delta-(k-1)n$, so by Lemma~\ref{lem:min-deg-matching}\ref{lem:min-deg-matching-gen} and~\eqref{eqn:inductive-hangers-high-min-deg-layer-base-large-sets} we have a matching $M$ of size $|M|=k\delta-(k-1)n$ in $W_1$. We shall check the conditions to apply Lemma~\ref{lem:high-min-deg-gen-intclique-clique-factor} for $b=2$ with $W_1,W_2\cup W_3,Z''_1,\dots,Z''_{k-1},A$ and $M$ as $U_1,U_2,X_1,\dots,X_{k-1},A$ and $F$ respectively. By definition $W_1$ and $W_2\cup W_3$ partition $\Gamma(f)$. For each $i\in[k-1]$ the set $Z''_i$ consists of the neighbours of $w$ whose only non-neighbour in $f$ is $u_i$ if $i < \ell$ and $u_{i+1}$ if $i \ge \ell$. The set $A$ consists of the non-neighbours of $w$ with exactly one non-neighbour in $f$ and the vertices with at least two non-neighbours in $f$. Hence, $W_1,W_2\cup W_3,Z''_1,\dots,Z''_{k-1},A$ form a partition of $V(G)$ such that there are no edges between $W_1$, $W_2$ and $W_3$. Given $i\in[k-1]$ there exists $j\in[k]$ such that $Z''_i\subseteq V(G)\setminus\Gamma(u_j)$ so $|Z''_i|\leq n-\delta$. For each $(e,i)\in E(G[W_1])\times[k-1]$, by applying Lemma~\ref{lem:high-min-deg-intclique-edge-int-nbrhood-indset-umbr}\ref{lem:high-min-deg-intclique-edge-int-nbrhood-indset} for $i$ with $u_1,\dots,u_{\ell-1},w$ as themselves, $u_{a+1}$ as $u_a$ for $\ell \le a < k$, $C$ as $C_1$, $C'$ as $C_2$ and $e$ as $uv$, we have that $Z''_i\cap\Gamma(e)$ is an independent set. Furthermore, all copies of $K_k$ in $G$ with at least two vertices from $W_1$ and all other vertices from $\left(\bigcup_{i=1}^{k-2}Z''_i\right)$ are $K_{k+1}$-connected: we can construct a $K_{k+1}$-walk from such a copy $g$ of $K_k$ to $fu_{\ell}$ by a step-by-step vertex replacement of the vertices of $g$ with the vertices of $fu_{\ell}$.

Since the requisite conditions are satisfied, we apply Lemma~\ref{lem:high-min-deg-gen-intclique-clique-factor} for $b=2$ with $W_1,W_2\cup W_3,Z''_1,\dots,Z''_{k-1},A$ and $M$ as $U_1,U_2,X_1,\dots,X_{k-1},A$ and $F$ respectively; noting that~\eqref{eqn:inductive-hangers-high-min-deg-layer-base-large-sets} and~\eqref{eqn:inductive-hangers-high-min-deg-layer-base-bad-set-bound} hold, we obtain that $\ckf_{k+1}(G)$ is at least
\begin{align*}
&(k+1)\min\left\{k\delta-(k-1)n,\left\lfloor\tfrac{|W_2\cup W_3|}{2}\right\rfloor,\tfrac{2[k\delta-(k-1)n]+|W_3|+1}{3}\right\} \\
&\geq(k+1)\min\left\{k\delta-(k-1)n,\tfrac{2[k\delta-(k-1)n]+|W_3|+1}{3}\right\}.
\end{align*}
First suppose there is a vertex $u\in W_3$. Since $\Gamma(u,f)\subseteq W_3$, by Lemma~\ref{lem:common-nbrhood-size} we have $|W_3|\geq|\Gamma(u,f)|\geq k\delta-(k-1)n$. This implies $\pp_k(n,\delta+\eta n)>\ckf_{k+1}(G)\geq(k+1)(k\delta-(k-1)n)$, which contradicts~\eqref{eqn:spacious-bound}. Hence, we have $W_3=\nth$. We distinguish three cases.

Case 1: $\delta\in\left[\left(\frac{k-1}{k}+\mu\right)n,\left(\frac{3k-2}{3k+1}-2\eta\right)n\right]\cup\left[\left(\frac{3k-2}{3k+1}+\eta\right)n,\left(\frac{2k-1}{2k+1}-2\eta\right)n\right]$. In this case, we have $\pp_k(n,\delta+\eta n)>\ckf_{k+1}(G)\geq\frac{2(k+1)(k\delta-(k-1)n)}{3}$, which contradicts~\eqref{eqn:two-thirds-bound}.

Case 2: $\delta\in\left[\left(\frac{3k-2}{3k+1}-2\eta\right)n,\left(\frac{3k-2}{3k+1}+\eta\right)n\right]$. Without loss of generality, we have $|W_1|\geq|W_2|$. By the upper bound on $\delta$, we have $n-\delta-1 \ge 3(k\delta-(k-1)n+1) - (3k+1)\eta n - 2$. Now together with~\eqref{eqn:inductive-hangers-high-min-deg-layer-base-large-nbrhood-two} we obtain $|W_1|\geq\frac{|\Gamma(f)|}{2}\geq\frac{9}{4}(k\delta-(k-1)n)+3$ so $\left\lfloor\frac{|W_1|}{3}\right\rfloor\geq\frac{3}{4}(k\delta-(k-1)n)$. Note that $\delta(G[W_1])\geq|W_1|-(n-\delta-|W_2|)$. By Corollary~\ref{cor:hajnal-szemeredi-cor} applied to $G[W_1]$ with $k=2$, \eqref{eqn:inductive-hangers-high-min-deg-layer-base-large-nbrhood-two} and~\eqref{eqn:inductive-hangers-high-min-deg-layer-base-large-sets}, the number of vertex-disjoint triangles in $G[W_1]$ is at least
\begin{align*}
&\min\left\lbrace|\Gamma(f)|+|W_2|-2(n-\delta),\left\lfloor\tfrac{|W_1|}{3}\right\rfloor\right\rbrace \\
&\geq\min\left\lbrace4(k\delta-(k-1)n)-(n-\delta),\left\lfloor\tfrac{|W_1|}{3}\right\rfloor\right\rbrace\geq\tfrac{3}{4}(k\delta-(k-1)n).
\end{align*}
Let $T$ be a collection of $\frac{3}{4}(k\delta-(k-1)n)$ vertex-disjoint triangles in $G[W_1]$. We apply Lemma~\ref{lem:high-min-deg-gen-intclique-clique-factor} for $b=3$ with $W_1,W_2,Z''_1,\dots,Z''_{k-1},A$ and $T$ as $U_1,U_2,X_1,\dots,X_{k-1},A$ and $F$ respectively; the requisite conditions have already been shown to be satisfied at the preceding application of Lemma~\ref{lem:high-min-deg-gen-intclique-clique-factor}. Noting~\eqref{eqn:inductive-hangers-high-min-deg-layer-base-large-sets}, we obtain that $\ckf_{k+1}(G)$ is at least
\begin{align*}
&(k+1)\min\left\lbrace\tfrac{3}{4}(k\delta-(k-1)n),\left\lfloor\tfrac{|W_2|}{2}\right\rfloor,\tfrac{2[k\delta-(k-1)n]+|W_2|+1}{5}\right\rbrace \\
&\geq\tfrac{3}{4}(k+1)(k\delta-(k-1)n),
\end{align*}
so $\pp_k(n,\delta+\eta n)>\frac{3}{4}(k+1)(k\delta-(k-1)n)$,
which contradicts~\eqref{eqn:three-quarters-bound}.

Case 3: $\delta\in\left[\left(\frac{2k-1}{2k+1}-2\eta\right)n,\frac{(2k-1)n-3}{2k+1}\right]$. Define
\[\widetilde{Z}_i=Z_i\cap\Gamma(w)\textrm{ for }i\in[k-1]\textrm{ and }\widetilde{A}:=\left(\bigcup_{i=1}^{k-1}Z_i\right)\setminus\Gamma(w).\]
No vertex in $W_1\cup\widetilde{A}$ is adjacent to $w$, so 
\begin{equation}
\label{eqn:inductive-hangers-high-min-deg-layer-base-bad-set-bound-matching}
|\widetilde{A}|\leq n-\delta-1-|W_1|.
\end{equation}
By Lemma~\ref{lem:cmpnt-facts}\ref{lem:nbrhood-cmpt-size-min-deg} we have $\delta(G[W_1])\geq k\delta-(k-1)n$, so there is a matching $M$ of size $|M|=k\delta-(k-1)n$ in $W_1$ by Lemma~\ref{lem:min-deg-matching}\ref{lem:min-deg-matching-gen} and~\eqref{eqn:inductive-hangers-high-min-deg-layer-base-large-sets}. We shall check the conditions to apply Lemma~\ref{lem:high-min-deg-gen-intclique-clique-factor-matching} with $W_1,W_2,\widetilde{Z}_1,\dots,\widetilde{Z}_{k-1},\widetilde{A}$ and $M$ as $U_1,U_2,X_1,\dots,X_{k-1},A$ and $F$ respectively. $W_1$ and $W_2$ partition $\Gamma(f)$ by definition. For each $i\in[k-1]$ the set $\widetilde{Z}_i$ consists of the neighbours $v$ of $w$ such that $\max\{j\in[k]\setminus\{\ell\}:vu_j\notin E(G)\}$ is well-defined and equal to $i$ if $i < \ell$ and to $i+1$ if $i \ge \ell$. The set $\widetilde{A}$ consists of the non-neighbours of $w$ with at least one non-neighbour in $f$. Hence, $W_1,W_2,\widetilde{Z}_1,\dots,\widetilde{Z}_{k-1},\widetilde{A}$ form a partition of $V(G)$ such that there are no edges between $W_1$ and $W_2$. Given $i\in[k-1]$ there is $j\in[k]$ such that $\widetilde{Z}_i\subseteq V(G)\setminus\Gamma(u_j)$ so $|\widetilde{Z}_i|\leq n-\delta$. Let $i\in[k-1]$ and let $g$ be a copy of $K_{i+1}$ comprising an edge $e$ of $G[W_1]$ and a copy $g'$ of $K_{i-1}$ with a vertex from each of $\widetilde{Z}_1,\dots,\widetilde{Z}_{i-1}$. By applying Lemma~\ref{lem:high-min-deg-intclique-edge-int-nbrhood-indset-umbr}~\ref{lem:high-min-deg-intclique-edge-int-nbrhood-indset-matching} for $i$ with $u_1,\dots,u_{\ell-1},w$ as themselves, $u_{a+1}$ as $u_a$ for $\ell \le a < k$, $C$ as $C_1$, $C'$ as $C_2$, $e$ as $uv$ and $g'$ as $g$, we have that $\widetilde{Z}_i\cap\Gamma(g)$ is an independent set. Furthermore, all copies of $K_k$ in $G$ comprising an edge of $G[W_1]$ and a vertex from each of $\widetilde{Z}_1,\dots,\widetilde{Z}_{k-2}$ are $K_{k+1}$-connected: we can construct a $K_{k+1}$-walk from such a copy $g$ of $K_k$ to $fu_{\ell}$ by a step-by-step vertex replacement of the vertices of $g$ with the vertices of $fu_{\ell}$.

Since the requisite conditions are satisfied, we apply Lemma~\ref{lem:high-min-deg-gen-intclique-clique-factor-matching} with the objects $W_1,W_2,\widetilde{Z}_1$, $\dots,\widetilde{Z}_{k-1},\widetilde{A}$ and $M$ as $U_1,U_2,X_1,\dots,X_{k-1},A$ and $F$ respectively; noting that in this case we have $\delta \ge \left(\tfrac{2k-1}{2k+1}-2\eta\right)n$, \eqref{eqn:inductive-hangers-high-min-deg-layer-base-large-sets} and~\eqref{eqn:inductive-hangers-high-min-deg-layer-base-bad-set-bound-matching}, we obtain that $\ckf_{k+1}(G)$ is at least
\begin{align*}
&(k+1)\min\left\lbrace k\delta-(k-1)n,k\delta-(k-1)n+|W_2|-|W_1|-|\widetilde{A}|\right\rbrace \\
&\geq(k+1)\min\left\lbrace k\delta-(k-1)n,k\delta-(k-1)n+|W_2|-(n-\delta-1)\right\rbrace \\
&\geq(k+1)\min\left\lbrace k\delta-(k-1)n,k\delta-(k-1)n+(2k+1)\delta-(2k-1)n+3\right\rbrace \\
&\geq(k+1)(k\delta-(k-1)n-2(2k+1)\eta n),
\end{align*}
so $\pp_k(n,\delta+\eta n)>(k+1)(k\delta-(k-1)n-2(2k+1)\eta n)$, contradicting~\eqref{eqn:spacious-bound}.
\end{proof}

Now we prove Lemma~\ref{lem:inductive-hangers-high-min-deg}.

\begin{proof}[Proof of Lemma~\ref{lem:inductive-hangers-high-min-deg}]
Let $S=\{(j,\ell)\in\mathbb{Z}^2\mid 1\leq j<\ell\leq k\}$. Note that $f:S\to\left[\frac{k(k-1)}{2}\right]$ given by $f(j,l)=\frac{\ell(\ell-1)}{2}-j+1$ is bijective and $f(j,\ell)<f(j',\ell')\iff \ell<\ell' \textrm{ or } (\ell=\ell',j'<j)$. We proceed by induction on $f(j,\ell)$. The base case $f(j,\ell)=1$ corresponds to $(j,\ell)=(1,2)$. By Lemma~\ref{lem:inductive-hangers-high-min-deg-base}, $G$ does not contain \config{k}{2}{1}. For $f(j,\ell)>1$, there are two cases to consider: $j+1=\ell\leq k$ and $j+1<\ell\leq k$.

Consider the first case $j+1=\ell\leq k$. By the inductive hypothesis, $G$ does not contain \config{k}{q}{p} for all $p,q$ such that $p<q<\ell$. Hence, by Lemma~\ref{lem:inductive-hangers-high-min-deg-layer-base} $G$ does not contain \config{k}{\ell}{\ell-1}. Consider the second case $j+1<\ell\leq k$. By the inductive hypothesis, $G$ does not contain \config{k}{q}{p} for all pairs $p,q$ such that $p<q<\ell$ or $j<p<q=\ell$. Hence, by Lemma~\ref{lem:inductive-hangers-high-min-deg-layer-int} $G$ does not contain \config{k}{\ell}{j}. This completes the proof by induction.

Finally, $G$ does not contain \config{k}{k}{1} so $\intr_k(G)$ is $K_k$-free.
\end{proof}

It remains to handle the case where $\intr_k(G)$ contains no copy of $K_k$. The following lemma represents an application of Lemma~\ref{lem:high-min-deg-gen-intclique-clique-factor-matching} for this case.

\begin{lemma} \label{lem:high-min-deg-ext-clique-factor}
Let $k\geq3$ be an integer. Let $G$ be a graph on $n$ vertices with minimum degree $\delta(G)\geq\delta>\frac{(k-1)n}{k}$, at least two $K_{k+1}$-components and $\intr_k(G)$ $K_k$-free. Let $C_1,\dots,C_p$ be the $K_{k+1}$-components of $G$. Set $q':=k\delta-(k-1)n+\sum_{j\neq1}|\extr(C_j)|-|\extr(C_1)|$. Then
\[\ckf_{k+1}(G)\geq(k+1)\min\left\lbrace k\delta-(k-1)n,\left\lfloor\frac{|\extr(C_1)|}{2}\right\rfloor,q'\right\rbrace.\]
\end{lemma}

\begin{proof}
By Lemma~\ref{lem:non-empty-intr-extr}\ref{item:non-empty-intr} we have $|\intr_k(G)|\geq2\delta-n+2>0$. Pick $u_{k-1}\in\intr_k(G)$ and recursively pick $u_i\in\Gamma(u_{k-1},\dots,u_{i+1};\intr_k(G))$ for $i\in[k-2]$. By Lemma~\ref{lem:common-nbrhood-size} we have
\begin{equation} \label{eqn:high-min-deg-ext-clique-factor-common-nbrhood-size}
\begin{split}
|\Gamma(u_{k-1},\dots,u_{i+1};\intr_k(G))|&\geq|\intr_k(G)|-(k-i-1)(n-\delta) \\
&\geq(k-i+1)\delta-(k-i)n+2>0
\end{split}
\end{equation}
for each $i\in[k-1]$ so this is well-defined. For $i\in[k-1]$ define
\[L_i=\Gamma(u_{k-1},\dots,u_{i+1};\intr_k(G))\backslash\Gamma(u_i).\]
We want to apply Lemma~\ref{lem:high-min-deg-gen-intclique-clique-factor-matching} with $\extr(C_1),\bigcup_{j\neq1}\extr(C_j),L_1,\dots,L_{k-1}$ and $\nth$ as $U_1,U_2,X_1$, $\dots,X_{k-1}$ and $A$ respectively. We claim that $L_1,\dots,L_{k-1}$ give a partition of $\intr_k(G)$. Indeed, for each $v\in\intr_k(G)$ we have $v\in L_h$ if and only if $h=\max\{a\in[k-1]:v\notin\Gamma(u_a)\}$; this quantity is well-defined because $\intr_k(G)$ is $K_k$-free. Furthermore, each set $L_i$ is nonempty by~\eqref{eqn:high-min-deg-ext-clique-factor-common-nbrhood-size} and the fact that each $u_i$ has at most $n-\delta$ non-neighbours. Hence, $L_1,\dots,L_{k-1},\extr(C_1),\bigcup_{j\neq1}\extr(C_j)$ gives a partition of $V(G)$. No vertex of $L_i$ is adjacent to $u_i$ so $|L_i|\leq n-\delta$ for each $i\in[k-1]$ and $|\intr_k(G)|\leq(k-1)(n-\delta)$. By Lemma~\ref{lem:cmpnt-facts}\ref{lem:btwn-extr-no-edge} there are no edges between $\bigcup_{j\neq1}\extr(C_j)$ and $\extr(C_1)$. This means that vertices in $\extr(C_1)$ have neighbours in only $\extr(C_1)$ and $\intr_k(G)$, so $\delta(\extr(C_1))\geq\delta-|\intr_k(G)|\geq k\delta-(k-1)n$. Hence, we have a matching $M$ in $\extr(C_1)$ with $|M|=\min\left\{k\delta-(k-1)n,\left\lfloor\frac{|\extr(C_1)|}{2}\right\rfloor\right\}$ by Lemma~\ref{lem:min-deg-matching}\ref{lem:min-deg-matching-gen}. All copies of $K_k$ in $G$ containing an edge of $G[\extr(C_1)]$ belong to $C_1$, so they are all $K_{k+1}$-connected. Let $i\in[k-2]$ and let $f$ be a copy of $K_{i+1}$ comprising an edge of $G[\extr(C_1)]$ and a vertex from each of $L_1,\dots,L_{i-1}$. Since we have $L_1\cup\dots\cup L_i\subseteq\Gamma(u_{k-1},\dots,u_{i+1};\intr_k(G))$, an edge in $L_i\cap\Gamma(f)$ would form a copy of $K_k$ in $\intr_k(G)$ together with $u_{k-1},\dots,u_{i+1}$ and the vertices of $f$ in $L_1\cup\dots\cup L_{i-1}$. This contradicts the assumption that $\intr_k(G)$ is $K_k$-free, so $L_i\cap\Gamma(f)$ is an independent set.

Since the requisite conditions are satisfied, we apply Lemma~\ref{lem:high-min-deg-gen-intclique-clique-factor-matching} with the objects $\extr(C_1)$, $\bigcup_{j\neq1}\extr(C_j),L_1,\dots,L_{k-1},\nth$ and $M$ as $U_1,U_2,X_1,\dots,X_{k-1},A$ and $F$ respectively to obtain that $\ckf_{k+1}(G)$ is at least
\[(k+1)\min\left\lbrace k\delta-(k-1)n,\left\lfloor\tfrac{|\extr(C_1)|}{2}\right\rfloor,q'\right\rbrace\]
as required.
\end{proof}

Now we aim to prove that if $\ckf_{k+1}(G)<\pp_k(n,\delta+\eta n)$ and $\intr_k(G)$ contains no copy of $K_k$, then $\intr_k(G)$ is in fact $(k-1)$-partite and its copies of $K_{k-1}$ lie in at least $r_p(n,\delta+\eta n)$ $K_{k+1}$-components. 

\begin{lemma} \label{lem:high-min-deg-interior-partite-many-exteriors}
Let $k\geq3$ be an integer and let $\mu>0$. Let $\eta>0$ and $n\in\NN$ satisfy $\frac{1}{n}\ll\eta\ll\mu,\frac{1}{k}$. Let $G$ be a graph on $n$ vertices with at least two $K_{k+1}$-components and minimum degree $\delta(G)\geq\delta\geq(\frac{k-1}{k}+\mu)n$. Suppose $\ckf_{k+1}(G)<\pp_k(n,\delta+\eta n)$ and $\intr_k(G)$ is $K_k$-free. Then $\intr_k(G)$ is $(k-1)$-partite and all copies of $K_{k-1}$ in $\intr_k(G)$ are contained in at least $r_p(n,\delta+\eta n)$ $K_{k+1}$-components of $G$.
\end{lemma}

\begin{proof}
Let $0<\eta<\min\{\frac{1}{1000k^2},\eta_0(k,\mu)\}$ and $n_1:=\max\{n_2(k,\mu,\eta),\frac{2}{\eta}\}$ with $\eta_0(k,\mu)$ and $n_2(k,\mu,\eta)$ given by Lemma~\ref{lem:extr-fn-bounds}. Set $r':=r_p(n,\delta+\eta n)$. Let $f:=u_1\dots u_{k-1}$ be a copy of $K_{k-1}$ in $\intr_k(G)$ and let $C_1,\dots,C_p$ be the $K_{k+1}$-components of $G$.

We claim that $f$ is a copy of $K_{k-1}$ of every $K_{k+1}$-component of $G$. Indeed, suppose $f$ is not a copy of $K_{k-1}$ of $C_i$ for some $i\in[p]$. Since $|\Gamma(f)|\geq(k-1)\delta-(k-2)n$ by Lemma~\ref{lem:common-nbrhood-size} and $|C_i|>\delta$ by Lemma~\ref{lem:cmpnt-facts}\ref{lem:cmpt-size}, there is a vertex $w\in\Gamma(f)$ which is also a vertex of $C_i$. Now since $fw\notin C_i$, we have $fw\in C_j$ for some $j\neq i$ and hence $w$ is a vertex of $C_j$. Since $w$ is a vertex of both $C_i$ and $C_j$, we have $w\in\intr_k(G)$, which in turn implies that $fw$ is a copy of $K_k$ in $\intr_k(G)$, contradicting our lemma hypothesis.

For $\delta\geq\left(\frac{k}{k+1}-2\eta\right)n$, note that by Lemma~\ref{lem:non-empty-intr-extr}\ref{item:non-empty-intr} we have $|\intr_k(G)|\geq2\delta-n+2>\frac{3k-4}{3}(n-\delta)$, so $\delta(G[\intr_k(G)])\geq\delta-n+|\intr_k(G)|>\tfrac{3k-7}{3k-4}|\intr_k(G)|$. Then, Theorem~\ref{thm:andrasfai-erdos-sos} implies that $\intr_k(G)$ is $(k-1)$-partite. Furthermore, by~\eqref{eqn:r'-bound-nearHAM} and since $G$ has at least two $K_{k+1}$-components, we have that all copies of $K_{k-1}$ in $\intr_k(G)$ are contained in at least $r'\leq2$ $K_{k+1}$-components. Therefore, it remains to consider the case $\delta<\left(\frac{k}{k+1}-2\eta\right)n$; by~\eqref{eqn:r'-bound} we have $r'\geq2$. For each $i\in[p]$, let $U_i$ be the set of common neighbours $v$ of $f$ such that $fv\in C_i$. Since $\intr_k(G)$ is $K_k$-free, we have $U_i\subseteq\extr(C_i)$ for each $i\in[p]$. Without loss of generality, let $\extr(C_1)$ be a largest $K_{k+1}$-component exterior of $G$.

Let $i\neq1$. Applying Lemma~\ref{lem:high-min-deg-ext-clique-factor} and noting that $|\extr(C_1)|\geq|\extr(C_i)|$, we have that $\ckf_{k+1}(G)$ is at least
\begin{align*}
(k+1)\min\left\{\left\lfloor\frac{|\extr(C_i)|}{2}\right\rfloor,k\delta-(k-1)n\right\}.
\end{align*}
Since $(k+1)(k\delta-(k-1)n)\geq\pp_k(n,\delta+\eta n)$ by~\eqref{eqn:spacious-bound}, we find that $(k+1)\left\lfloor\frac{|\extr(C_i)|}{2}\right\rfloor<\pp_k(n,\delta+\eta n)\leq\frac{k+1}{2}\left(\frac{(k-1)(\delta+3\eta n)-(k-2)n}{r'}-2\right)$ by~\eqref{eqn:extr-approx-bound}. Hence, we have
\begin{equation} \label{high-min-deg-interior-partite-many-exteriors-small-extr}
|U_i|\leq|\extr(C_i)|<\frac{(k-1)(\delta+3\eta n)-(k-2)n}{r'}.
\end{equation}
By Lemma~\ref{lem:cmpnt-facts}\ref{lem:cmpt-size} we have $|\extr(C_i)|+|\intr_k(G)|\geq|C_i|>\delta$, so by~\eqref{eqn:partite-intr-bound} we have
\begin{equation}
\label{high-min-deg-interior-partite-many-exteriors-large-intr}
|\intr_k(G)|>\delta-\frac{(k-1)(\delta+3\eta n)-(k-2)n}{r'}>\frac{3k-4}{3}(n-\delta).
\end{equation}
It follows that $\delta(G[\intr_k(G)])\geq\delta-n+|\intr_k(G)|>\tfrac{3k-7}{3k-4}|\intr_k(G)|$, so Theorem~\ref{thm:andrasfai-erdos-sos} implies that $\intr_k(G)$ is $(k-1)$-partite. Let $I_1,\dots,I_{k-1}$ be the parts of $\intr_k(G)$. For each $j\in[k-1]$ we have that $I_j$ is an independent set, so $|I_j|\leq n-\delta$. Hence, we have $|I_j|=|\intr_k(G)|-\sum_{h\neq j}|I_h|>(k-1)\delta-(k-2)n-\frac{(k-1)(\delta+3\eta n)-(k-2)n}{r'}$. Furthermore, each vertex in $I_j$ is adjacent to all but at most $n-\delta-|I_j|$ vertices outside $I_j$.

It remains to show $p\geq r'$, so suppose $p<r'$. In particular, this implies $r'\geq3$. Since~\eqref{high-min-deg-interior-partite-many-exteriors-small-extr} and
$\sum_{i\in[p]}|\extr(C_i)|\geq\sum_{i\in[p]}|U_i|=|\Gamma(f)|\geq(k-1)\delta-(k-2)n$ hold, we obtain $|\extr(C_1)|\geq|\Gamma(f)|-\sum_{i\neq 1}|\extr(C_i)|>\frac{2[(k-1)\delta-(k-2)n]-3(k-1)(r'-2)\eta n}{r'}$. By Lemma~\ref{lem:cmpnt-facts}\ref{lem:btwn-extr-no-edge}, there are no edges between $\extr(C_1)$ and $\bigcup_{i\neq1}\extr(C_i)$, so every vertex in $\extr(C_1)$ has neighbours in $\extr(C_1)$ and $\intr_k(G)$ only. Hence, we have $\delta(\extr(C_1))\geq\delta-|\intr_k(G)|$. By Lemma~\ref{lem:min-deg-matching}\ref{lem:min-deg-matching-gen}, there is a matching $F_0$ in $\extr(C_1)$ with
\[|F_0|=\min\left\{\delta-|\intr_k(G)|,\frac{[(k-1)\delta-(k-2)n]-3(k-1)(r'-1)\eta n}{r'}\right\}.\]

We now build up our desired connected $K_{k+1}$-factor step-by-step, starting from the aforementioned matching $F_0$ in $\extr(C_1)$. We have steps $j=1,\dots,k-1$. In step $j$, we extend the $K_{j+1}$-factor $F_{j-1}$ to a $K_{j+2}$-factor $F_j$ using $I_j$. We greedily match vertices of $I_j$ with distinct copies of $K_{j+1}$ of $F_{j-1}$ to form copies of $K_{j+2}$. We find that $|I_j|>|F_0|\geq|F_{j-1}|$, so we stop only when we encounter a vertex $x\in I_j$ which is not a common neighbour of any remaining copy of $K_{j+1}$ of $F_{j-1}$. Since at most $n-\delta-|I_j|$ copies of $K_{j+1}$ in $F_{j-1}$ do not have $x$ as a common neighbour, we obtain a $K_{j+2}$-factor $F_j$ with at least $|F_{j-1}|-(n-\delta)+|I_j|$ copies of $K_{j+2}$.

We terminate after step $k-1$ with a collection $F_{k-1}$ of at least $|F_0|-(k-1)(n-\delta)+|\intr_k(G)|$ vertex-disjoint copies of $K_{k+1}$ in $G$. Since each copy of $K_{k+1}$ in $F_{k-1}$ uses an edge of $F_0\subseteq G[\extr(C_1)]$ and~\eqref{high-min-deg-interior-partite-many-exteriors-large-intr} holds, we deduce that $F_{k-1}$ is in fact a connected $K_{k+1}$-factor of size at least $(k+1)(k\delta-(k-1)n-3(k-1)\eta n)$.
By~\eqref{eqn:spacious-bound}, this means that $\ckf_{k+1}(G)\geq\pp_k(n,\delta+\eta n)$, which is a contradiction. This completes the proof.
\end{proof}

We prove in the following lemma that a graph which has very high minimum degree and is not near-extremal in fact contains a large connected $K_{k+1}$-factor. We handle this case separately as it turns out that our greedy-type methods in Section~\ref{subsection:clique-factor-constr} are insufficient. To overcome this, we employ a Hall-type argument (see Lemma~\ref{lem:min-deg-matching}\ref{lem:min-deg-matching-bipartite}) for the purposes of extending our large matchings to sufficiently large connected $K_{k+1}$-factors.

\begin{lemma} \label{lem:very-high-min-deg-non-extr-partite-int-clique-factor}
Let $k\geq2$ be an integer and let $\mu>0$. Let $\eta>0$ and $n\in\NN$ satisfy $\frac{1}{n}\ll\eta\ll\mu,\frac{1}{k}$. Let $G$ be a graph on $n$ vertices with minimum degree $\delta=\delta(G)\geq\left(\frac{2k-1}{2k+1}-2\eta\right)n$, exactly two $K_{k+1}$-components, $\intr_k(G)$ $(k-1)$-partite and either $|\intr_k(G)|<(k-1)(n-\delta)-5k\eta n$ or the larger $K_{k+1}$-component exterior $X$ satisfies $|X|>\frac{19}{10}(k\delta-(k-1)n)$. Then $\ckf_{k+1}(G)\geq\pp_k(n,\delta+\eta n)$.
\end{lemma}

\begin{proof}
Let $0<\eta<\min\{\frac{1}{1000k^2},\eta_0(k,\mu),\frac{k\mu^2}{k+1}\}$ and $n_1:=\max\{n_2(k,\mu,\eta),\frac{2}{\eta}\}$ with $\eta_0(k,\mu)$ and $n_2(k,\mu,\eta)$ given by Lemma~\ref{lem:extr-fn-bounds}. Let $C_1$ and $C_2$ be the two $K_{k+1}$-components of $G$. There is a partition of $V(G)$ into the vertex classes $\intr_k(G), \extr(C_1)$ and $\extr(C_2)$; $\intr_k(G)$ is further partitioned into $k-1$ independent sets $I_1, \dots, I_{k-2}$ and $I_{k-1}$. Without loss of generality, suppose $|\extr(C_1)|\geq|\extr(C_2)|$. Since $I_i$ is an independent set, we have
\begin{equation}
\label{eqn:very-high-min-deg-non-extr-partite-int-clique-factor-ind-set-upper-bound}
|I_i|\leq n-\delta\quad\textrm{for each }i\in[k-1].
\end{equation}

If $\delta\geq\left(\frac{k}{k+1}-2\eta\right)n$, then by Lemma~\ref{lem:non-empty-intr-extr}\ref{item:non-empty-intr} we have $|\intr_k(G)|\geq2\delta-n+2\geq(k-1)(n-\delta)-2(k+1)\eta n$ and by Lemma~\ref{lem:non-empty-intr-extr}\ref{item:non-empty-extr} we have $|\extr(C_1)|\leq n-\delta-1\leq k\delta-(k-1)n+2(k+1)\eta n\leq\frac{19}{10}(k\delta-(k-1)n)$, which contradicts the lemma hypothesis. Therefore, we have $\delta<\left(\frac{k}{k+1}-2\eta\right)n$. In particular, this means that $r':=r_p^{(k)}(n,\delta+\eta n)\geq2$.

By~\eqref{eqn:very-high-min-deg-non-extr-partite-int-clique-factor-ind-set-upper-bound} we have $|\intr_k(G)|\leq(k-1)(n-\delta)$. By Lemma~\ref{lem:cmpnt-facts}\ref{lem:cmpt-size} we have $|C_1|>\delta$, so $|\extr(C_1)|>\delta-(k-1)(n-\delta)=k\delta-(k-1)n\geq0$. By Lemma~\ref{lem:cmpnt-facts}\ref{lem:btwn-extr-no-edge}, there are no edges between $\extr(C_1)$ and $\extr(C_2)$, so every vertex in $\extr(C_1)$ has neighbours in $\extr(C_1)$ and $\intr_k(G)$ only. Hence, we have $\delta(\extr(C_1))\geq\delta-|\intr_k(G)|\geq\delta-(k-1)(n-\delta)=k\delta-(k-1)n$. Therefore, we can conclude by Lemma~\ref{lem:min-deg-matching}\ref{lem:min-deg-matching-gen} that there is matching $F_0$ in $\extr(C_1)$ of size $|F_0|=\min\{k\delta-(k-1)n, \left\lfloor\frac{|\extr(C_1)|}{2}\right\rfloor\}$.

We build up the desired connected $K_{k+1}$-factor step-by-step, starting from the aforementioned matching $F_0$. We have steps $j=1,\dots,k-1$. In step $j$ we extend the $K_{j+1}$-factor $F_{j-1}$ to a $K_{j+2}$-factor $F_j$ using $I_j$. By Lemma~\ref{lem:cmpnt-facts}\ref{lem:btwn-extr-no-edge}, there are no edges between $\extr(C_1)$ and $\extr(C_2)$, so every vertex in $\extr(C_1)$ has at least $\delta-|\extr(C_1)|-\sum_{h\neq j}|I_h|$ neighbours in $I_j$. For each $i\in[j-1]$, since $I_i$ is an independent set, every vertex of $I_i$ has at least $\delta-(n-|I_j|-|I_i|)$ neighbours in $I_j$. Therefore, by Lemma~\ref{lem:common-nbrhood-size} every copy of $K_{j+1}$ in $F_{j-1}$ has at least 
\begin{align*}
a_j&:=2\left(\delta-|\extr(C_1)|-\sum_{h\neq j}|I_h|\right) + \sum_{i=1}^{j-1}(\delta-n+|I_j|+|I_i|)-j|I_j| \\
&=(j+1)\delta-(j-1)n-2|\extr(C_1)|-\sum_{i=1}^{k-1}|I_i|-\sum_{i=j+1}^{k-1}|I_i|
\end{align*}
common neighbours in $I_j$. At the same time, since $I_j$ is an independent set, every vertex of $I_j$ has at least $\delta-(n-|\extr(C_1)|-|I_1|-\dots-|I_j|)$ neighbours in $\extr(C_1)\cup I_1\cup\dots\cup I_j$, of which all but at most $|\extr(C_1)|+|I_1|+\dots+|I_j|-(j+1)|F_{j-1}|$ are in $F_{j-1}$. Hence, every vertex in $I_j$ has at least 
\begin{align*}
b_j & := \delta - ( n - |\extr(C_1)| + |I_1| + \dots + |I_j| ) \\
& \qquad - ( |\extr(C_1)| + |I_1| + \dots + |I_{j-1}| - (j+1)|F_{j-1}| )-j|F_{j-1}| \\
& = \delta - n + |I_j| + |F_{j-1}|
\end{align*}
copies of $K_{j+1}$ of $F_{j-1}$ in its neighbourhood. Form an auxiliary bipartite graph with vertex set $F_{j-1}\cup I_j$, where $f\in F_{j-1}$ is adjacent to $u\in I_j$ if and only if $fu$ is a copy of $K_{j+2}$ in $G$. By Lemma~\ref{lem:min-deg-matching}\ref{lem:min-deg-matching-bipartite}, there is a matching in the auxiliary bipartite graph with at least $\min\{a_j+b_j,|F_{j-1}|,|I_j|\}$ edges, which corresponds to a collection $F_j$ of
\begin{equation} \label{eqn:very-high-min-deg-non-extr-partite-int-clique-factor-hall}
|F_j|=\min\{a_j+b_j,|F_{j-1}|,|I_j|\}
\end{equation}
vertex-disjoint copies of $K_{j+2}$ in $G$. Lemma~\ref{lem:non-empty-intr-extr}\ref{item:non-empty-intr} tells us $|\intr_k(G)|\geq2\delta-n+2$, so by~\eqref{eqn:very-high-min-deg-non-extr-partite-int-clique-factor-ind-set-upper-bound} we have
\begin{equation}
\label{eqn:very-high-min-deg-non-extr-partite-int-clique-factor-ind-set-lower-bound}
|I_j|=|\intr_k(G)|-\sum_{h\neq j}|I_h|> k\delta-(k-1)n\geq|F_0|\geq|F_{j-1}|.
\end{equation}
Observe that by~\eqref{eqn:very-high-min-deg-non-extr-partite-int-clique-factor-ind-set-upper-bound} we have
\begin{equation} \label{eqn:very-high-min-deg-non-extr-partite-int-clique-factor-hall-sum}
\begin{split}
a_j+b_j&=(j+2)\delta-jn-2|\extr(C_1)|-\sum_{i=j+1}^{k-1}|I_i|-\sum_{i\neq j}|I_i|+|F_{j-1}| \\
&\geq(2k-1)\delta-(2k-3)n-2|\extr(C_1)|+|F_{j-1}|.
\end{split}
\end{equation}
Since by Lemma~\ref{lem:non-empty-intr-extr}\ref{item:non-empty-extr} we have $|\extr(C_1)|\leq n-\delta-1$ and recalling our assumption that $\delta\geq\left(\frac{2k-1}{2k+1}-2\eta\right)n$, by~\eqref{eqn:very-high-min-deg-non-extr-partite-int-clique-factor-hall-sum} we have
\begin{equation}
\label{eqn:very-high-min-deg-non-extr-partite-int-clique-factor-hall-sum-lower-bound}
a_j+b_j\geq|F_{j-1}|-2(2k+1)\eta n.
\end{equation}
Furthermore, by~\eqref{eqn:very-high-min-deg-non-extr-partite-int-clique-factor-hall-sum} and $\delta\geq\left(\frac{2k-1}{2k+1}-2\eta\right)n$ we obtain
\begin{equation}
\label{eqn:very-high-min-deg-non-extr-partite-int-clique-factor-hall-sum-small-extr-large}
\textrm{if }|\extr(C_1)|\leq\left(\tfrac{2}{2k+1}-(2k-1)\eta\right)n,\textrm{ then }a_j+b_j\geq|F_{j-1}|\textrm{ for all }j.
\end{equation}
All copies of $K_k$ in $G$ containing an edge of $G[\extr(C_1)]$ belong to $C_1$, so they are all $K_{k+1}$-connected. Therefore, $F_{k-1}$ is a connected $K_{k+1}$-factor.

It remains to check that $(k+1)|F_{k-1}|\geq\pp_k(n,\delta+\eta n)$. We first consider when $|F_0|=k\delta-(k-1)n$. In this case, noting~\eqref{eqn:very-high-min-deg-non-extr-partite-int-clique-factor-hall}, \eqref{eqn:very-high-min-deg-non-extr-partite-int-clique-factor-ind-set-lower-bound} and~\eqref{eqn:very-high-min-deg-non-extr-partite-int-clique-factor-hall-sum-lower-bound} we have that $|F_j|\geq|F_{j-1}|-2(2k+1)\eta n$ for each $j\in[k-1]$, so $F_{k-1}$ is a connected $K_{k+1}$-factor in $G$ of size at least $(k+1)(k\delta-(k-1)n-2(k-1)(2k+1)\eta n)\geq\pp_k(n,\delta+\eta n)$ by~\eqref{eqn:spacious-bound}. Now consider when $|F_0|=\left\lfloor\frac{|\extr(C_1)|}{2}\right\rfloor$. We distinguish two cases.

Case 1: $a_j+b_j\geq|F_{j-1}|$ for each $j\in[k-1]$. In this case, $F_{k-1}$ is a connected $K_{k+1}$-factor in $G$ of size $(k+1)|F_0|=(k+1)\left\lfloor\frac{|\extr(C_1)|}{2}\right\rfloor$. Suppose that this is less than $\pp_k(n,\delta+\eta n)$. By~\eqref{eqn:extr-approx-bound} and $\delta\geq\left(\frac{2k-1}{2k+1}-2\eta\right)n$, we have $|\extr(C_1)|<\frac{(k-1)(\delta+3\eta n)-(k-2)n}{2}\leq \frac{19}{10}(k\delta-(k-1)n)$. Furthermore, $|\intr_k(G)|\geq n-2|\extr(C_1)|>(k-1)(n-\delta)-3(k-1)\eta n$. This contradicts the lemma hypothesis.

Case 2: $a_j+b_j<|F_{j-1}|$ for some $j\in[k-1]$. By~\eqref{eqn:very-high-min-deg-non-extr-partite-int-clique-factor-hall-sum-small-extr-large}, this means that $|\extr(C_1)|>\left(\frac{2}{2k+1}-(2k-1)\eta\right)n\geq 2(k\delta-(k-1)n)-(2k-1)\eta n$. By~\eqref{eqn:very-high-min-deg-non-extr-partite-int-clique-factor-hall}, \eqref{eqn:very-high-min-deg-non-extr-partite-int-clique-factor-ind-set-lower-bound} and~\eqref{eqn:very-high-min-deg-non-extr-partite-int-clique-factor-hall-sum-lower-bound} we have $|F_j|\geq|F_{j-1}|-2(2k+1)\eta n$ for each $j\in[k-1]$, so $F_{k-1}$ is a connected $K_{k+1}$-factor in $G$ of size at least
\[(k+1)(|F_0|-2(k-1)(2k+1)\eta n)\geq(k+1)(k\delta-(k-1)n-6k^2\eta n).\]
By~\eqref{eqn:spacious-bound} this is at least $\pp_k(n,\delta+\eta n)$.
\end{proof}

Finally, we prove Lemma~\ref{lem:general-cmpnt}.

\begin{proof}[Proof of Lemma~\ref{lem:general-cmpnt}]
Given $k\geq3$, $\mu>0$ and $0<\eta<\min\{\frac{1}{1000k^2},\eta_0(k,\mu),\frac{k\mu^2}{k+1}\}$, let $m_1:=\max\{n_2(k,\mu,\eta),2/\eta,k(k+1)\}$ with the quantities $\eta_0(k,\mu)$ and $n_2(k,\mu,\eta)$ given by Lemma~\ref{lem:extr-fn-bounds}. Let $\delta\geq(\frac{k-1}{k}+\mu)n$. Let $G$ be a graph on $n\geq m_1$ vertices with minimum degree $\delta(G)\geq\delta$ and at least two $K_{k+1}$-components, each of which contains a copy of $K_{k+2}$. Let $C_1,\dots,C_\ell$ be the $K_{k+1}$-components of $G$. Set $\alpha:=|\intr_k(G)|$.

Lemma~\ref{lem:non-empty-intr-extr}\ref{item:non-empty-intr} tells us $\intr_k(G)\neq\nth$ and $|\intr_k(G)|>2\delta-n>(k-2)(n-\delta)$, so $\delta(G[\intr_k(G)])\geq\delta-n+|\intr_k(G)|>\frac{k-3}{k-2}|\intr_k(G)|$. Hence, any vertex in $\intr_k(G)$ can be extended to a copy of $K_{k-1}$ in $\intr_k(G)$ by Lemma~\ref{lem:clique-extn}. In particular, $\intr_k(G)$ contains a copy of $K_{k-1}$.

Suppose that~\ref{item:sub-stability-target-factor} does not hold. Lemma~\ref{lem:inductive-hangers-high-min-deg} tells us that $G[\intr_k(G)]$ is $K_k$-free, so Lemma~\ref{lem:high-min-deg-interior-partite-many-exteriors} implies that $\intr_k(G)$ is $(k-1)$-partite and all copies of $K_{k-1}$ in $G[\intr_k(G)]$ (of which there is at least one) are contained in at least $r':=r_p(n,\delta+\eta n)$ $K_{k+1}$-components. Hence, $G$ has at least $r'$ $K_{k+1}$-components. Since $\intr_k(G)$ is $(k-1)$-partite, we have $\alpha\leq(k-1)(n-\delta)$. Lemma~\ref{lem:cmpnt-facts}\ref{lem:cmpt-size} tells us that $|C_i|>\delta$, so
\begin{equation}
\label{eqn:general-cmpnt-extr-lower-bound}
|\extr(C_i)|\geq\delta-\alpha+1\geq k\delta-(k-1)n+1
\end{equation}
for each $i\in[\ell]$. In particular, every $K_{k+1}$-component has a non-empty exterior. Pick $x\in\extr(C_2)$. It has at least $\delta$ neighbours, none of which are in $\extr(C_1)\cup\extr(C_3)\cup\dots\cup\extr(C_\ell)$ by Lemma~\ref{lem:cmpnt-facts}\ref{lem:btwn-extr-no-edge}. Observe that
\begin{align}
n&=|\intr_k(G)|+|\extr(C_1)|+\dots+|\extr(C_\ell)|\quad\textrm {and} \label{eqn:general-cmpnt-partition-sum} \\
n&\geq1+\delta+|\extr(C_1)|+|\extr(C_3)|+\dots+|\extr(C_\ell)|. \label{eqn:general-cmpnt-extr2-ineq}
\end{align}
Without loss of generality, suppose $\extr(C_1)$ is a largest $K_{k+1}$-component exterior. By Lemma \ref{lem:cmpnt-facts}\ref{lem:btwn-extr-no-edge}, there are no edges between any pair of $K_{k+1}$-component exteriors. Note that for any $K_{k+1}$-component $C$, all copies of $K_k$ in $G$ containing at least one vertex of $\extr(C)$ are in $C$ and are therefore $K_{k+1}$-connected in $G$. Hence, it is enough to prove that
\[\alpha\geq(k-1)(n-\delta)-5k\eta n \textrm{ and }|\extr(C_1)|\leq\frac{19}{10}(k\delta-(k-1)n),\]
as this would imply that~\ref{item:sub-stability-extremal} holds. Suppose this is not the case.

\begin{claim}
$G$ has exactly $r'$ $K_{k+1}$-components.
\end{claim}

\begin{claimproof}
Suppose that $\ell\geq r'+1$. By~\eqref{eqn:general-cmpnt-partition-sum} we have $(r'+1)(\delta-\alpha)+\alpha<n$. We consider two cases.

Case 1: $\alpha<(k-1)(n-\delta)-5k\eta n$. Then we have
\begin{align*}
(r'+1)\delta&<n+r'\alpha<n+r'((k-1)(n-\delta)-5k\eta n) \\
&=[(k-1)r'+1]n-(k-1)r'\delta-(kr'+1)\eta n-(4kr'-1)\eta n,
\end{align*}
which we rearrange to obtain
\[\delta+\eta n<\frac{[(k-1)r'+1]n-(4kr'-1)\eta n}{kr'+1}.\]
Comparing this with~\eqref{eqn:r-delta-ineq2} applied to $r':=r_p(n,\delta+\eta n)$, we deduce $r'>(4kr'-1)\eta n\geq4kr'-1\geq r'$, which is a contradiction.

Case 2: $|\extr(C_1)|>\frac{19}{10}(k\delta-(k-1)n)$. By~\eqref{eqn:general-cmpnt-extr-lower-bound} and~\eqref{eqn:general-cmpnt-extr2-ineq}, we have
\[1+\delta+\frac{19}{10}(k\delta-(k-1)n)+(r'-1)[k\delta-(k-1)n+1]\leq n,\]
which we simplify to
\[\frac{9}{10}(k\delta-(k-1)n)+r'[k\delta-(k-1)n]<n-\delta.\]
Since by~\eqref{eqn:r-delta-ineq1} we have $r'\geq\frac{n-\delta-\eta n}{k\delta-(k-1)n+k\eta n+1}$, we deduce that
\[\frac{9}{10}(k\delta-(k-1)n)+\frac{n-\delta-\eta n}{k\delta-(k-1)n+k\eta n+1}[k\delta-(k-1)n]<n-\delta.\]
Since $\eta<\frac{k\mu^2}{k+1}$ and $k\delta-(k-1)n\geq k\mu n$, we have
\begin{align*}
&(k\delta-(k-1)n+k\eta n+1)(1-\mu) \\
&<k\delta-(k-1)n+(k+1)\eta n-\mu(k\delta-(k-1)n) \\
&\leq k\delta-(k-1)n+(k+1)\eta n-k\mu^2n \\
&<k\delta-(k-1)n,
\end{align*}
so applying this to the previous inequality, we obtain
\[\frac{9}{10}k\mu n+(n-\delta-\eta n)(1-\mu)<n-\delta.\]
However, since $\eta<\mu$ and $n-\delta<\frac{n}{k}$, this is a contradiction. Therefore, $G$ has exactly $r'$ $K_{k+1}$-components.
\end{claimproof}

In particular, this means that $r'\geq2$. For $r'=2$, Lemma~\ref{lem:very-high-min-deg-non-extr-partite-int-clique-factor} gives a contradiction, so it remains to consider the case $r'\geq3$. First suppose $|\extr(C_1)|\leq\sum_{h\neq1}|\extr(C_h)|$.  By Lemma~\ref{lem:high-min-deg-ext-clique-factor}, we have
\[\ckf_{k+1}(G) \ge (k+1)\min\left\lbrace\left\lfloor\tfrac{|\extr(C_1)|}{2}\right\rfloor,k\delta-(k-1)n\right\rbrace.\]
We have $\ckf_{k+1}(G) < \pp_k(n,\delta+\eta n)$ by assumption, so by~\eqref{eqn:spacious-bound} we have
\[(k+1)\left\lfloor\tfrac{|\extr(C_1)|}{2}\right\rfloor < \pp_k(n,\delta+\eta n).\]
Hence, by~\eqref{eqn:extr-approx-bound} and~\eqref{eqn:spacious-bound} we obtain $|\extr(C_1)|<\frac{(k-1)(\delta+3\eta n)-(k-2)n}{r'}$ and $|\extr(C_1)|\leq\frac{19}{10}(k\delta-(k-1)n)$. Then, by~\eqref{eqn:general-cmpnt-partition-sum} we have $|\intr_k(G)|=n-\sum_{i\in[r']}|\extr(C_i)|>(k-1)(n-\delta)-3(k-1)\eta n$. This is contradicts our earlier supposition, so we have that
\begin{equation}
\label{eqn:general-cmpnt-extr-dominant}
|\extr(C_1)|>\sum_{h\neq1}|\extr(C_h)|.
\end{equation}

Set $r:=r_p(n,\delta)$. By~\eqref{eqn:general-cmpnt-extr-lower-bound} and~\eqref{eqn:general-cmpnt-extr2-ineq}, we have that
\[1+\delta+(r'-1)(k\delta-(k-1)n+1)+(r'-2)(k\delta-(k-1)n+1)\leq n.\]
Rearranging and applying~\eqref{eqn:r-delta-ineq1}, we obtain
\[2r'-3\leq\frac{n-\delta-1}{k\delta-(k-1)n+1}<r\leq r'+1,\]
which gives $r=r'+1=4$. In particular, by~\eqref{eqn:r-delta-ineq2} we have $\delta\geq\left(\frac{3k-2}{3k+1}-2\eta\right)n$. By~\eqref{eqn:general-cmpnt-extr-lower-bound} and~\eqref{eqn:general-cmpnt-extr-dominant}, we have $|\extr(C_1)|>|\extr(C_2)|+|\extr(C_3)|\geq2(k\delta-(k-1)n+1)$. By~\eqref{eqn:general-cmpnt-extr2-ineq} and the fact that $\delta\geq\left(\frac{3k-2}{3k+1}-2\eta\right)n$, we have $|\extr(C_1)|<n-\delta-|\extr(C_3)|\leq2[k\delta-(k-1)n]+2(3k+1)\eta n$. Finally, by Lemma~\ref{lem:high-min-deg-ext-clique-factor} and~\eqref{eqn:spacious-bound}, we have
\[\ckf_{k+1}(G)\geq(k+1)\left(k\delta-(k-1)n-2(3k+1)\eta n\right)\geq\pp_k(n,\delta+\eta n),\]
which is a contradiction. This completes the proof of Lemma~\ref{lem:general-cmpnt}.
\end{proof}

\section{Near-extremal graphs} \label{section:extremal}

In this section we provide our proof of Lemma~\ref{lem:extremal}. To this end, we start with two useful lemmas. The first lemma will be used to construct $k$th powers of paths and cycles from simple paths and cycles through repeated application.

\begin{lemma} \label{lem:power-step-up}
Given $h\in\NN, h'\in[h+1]$ and a graph $G$, let $T=t_1\dots t_{(h+1)\ell+h'-1}$ be the $h$th power of a path in $G$ and let $W$ be a set of vertices disjoint from $T$. Let $Q_1:=t_1\dots t_{h+1}$, $Q_i:=t_{(h+1)(i-2)+1}\dots t_{(h+1)i}$ for each $1<i\leq\ell$, and $Q_{\ell+1}:=t_{(h+1)\ell-h}\dots t_{(h+1)\ell+h'-1}$. If there exists a permutation $\sigma$ of $[\ell+1]$ such that for each $i\in[\ell+1]$ the vertices of $Q_{\sigma(i)}$ have at least $i$ common neighbours in $W$, then there is the $(h+1)$st power of a path 
\[(q_1t_1\dots t_{h+1})\dots (q_\ell t_{(h+1)\ell-h}\dots t_{(h+1)\ell})(q_{\ell+1}t_{(h+1)\ell+1}\dots t_{(h+1)\ell+h'-1})\]
in $G$, with $q_i\in W$ for each $i\in[\ell+1]$, using every vertex of $T$. If $T$ is a cycle on $(h+1)\ell$ vertices we let instead $Q_1:=t_{(h+1)\ell-h}\dots t_{(h+1)\ell}t_1\dots t_{h+1}$, $Q_i:=t_{(h+1)(i-2)+1}\dots t_{(h+1)i}$ for each $1<i\leq\ell$ and $\sigma$ be a permutation on $[\ell]$. Then, under the same conditions, we have the $(h+1)$st power of a cycle $C^{h+1}_{(h+2)\ell}$.
\end{lemma}

\begin{proof}
Choose for each $i$ in succession $q_{\sigma(i)}$ to be any so far unused common neighbour of $Q_{\sigma(i)}$. The lemma hypothesis ensures that this is always possible.
\end{proof}

The second lemma allows us to construct paths and cycles of desired lengths which keep certain `bad' vertices far apart. We apply Theorem~\ref{thm:williamson} in its proof.

\begin{lemma} \label{lem:extr-ext-path-cycle}
Let $H$ be a graph on $h\geq10$ vertices and $B\subseteq V(H)$ be of size at most $\frac{h}{12}$. Suppose that every vertex in $B$ has at least $3|B|+1$ neighbours in $H$, and every vertex outside $B$ has at least $\frac{h}{2}+2|B|+2$ neighbours in $H$. Then for any $3\leq\ell\leq h$ we can find a cycle $C_\ell$ of length $\ell$ in $H$ on which no four consecutive vertices contain more than one vertex of $B$. Furthermore, if $x$ and $y$ are any two vertices not in $B$ and $5\leq\ell\leq h$, we can find an $\ell$-vertex path $P_\ell$ whose end-vertices are $x$ and $y$ and on which no four consecutive vertices contain more than one vertex of $B\cup\{x,y\}$. 
\end{lemma}

\begin{proof}
If we seek a path in $H$ from $x$ to $y$ and $xy\notin E(H)$, add $xy$ as a `dummy' edge. If we seek a cycle, let $xy$ be any edge of $H$ such that $x,y\notin B$. Hence, it suffices to show for each $3\leq\ell\leq h$ and each edge $xy\in V(H)$ with $x,y\notin B$ that we can find a cycle $C_\ell$ of length $\ell$ with $xy$ as an edge, on which no four consecutive vertices contain more than one vertex of $B$ and on which any four consecutive vertices including a vertex of $B$ contain neither $x$ nor $y$.

Let $H_1:=H[V(H)\setminus B]$. Since $H_1$ is a graph on $h-|B|\geq4$ vertices with minimum degree $\delta(H_1)\ge\frac{h}{2}+|B|+2\ge\frac{h-|B|}{2}+1$, by Theorem~\ref{thm:williamson} $H_1$ is panconnected. Hence, $H_1$ has paths between $x$ and $y$ of every number of vertices from $3$ to $h-|B|$. By adding the edge $xy$ to these paths, we obtain cycles of every length from $3$ to $h-|B|$ with the desired properties.

To find the required cycles of length greater than $h-|B|$, we first construct a path $P$ in $H$ covering $B$ with $x$ as an end-vertex and $xy$ as an edge. Let $B=\{b_1,\dots,b_{|B|}\}$ and set $B':=B\cup\{x,y\}$. For each $i\in[|B|]$ choose distinct vertices $u_{i+1},v_i\in V(H)\setminus B'$ adjacent to $b_i$. Every vertex in $B$ has at least $3|B|+1$ neighbours in $H$, so we may pick these vertices to be distinct for all $i\in[|B|]$. Choose a different vertex $u_1\in V(H)\setminus B'$ adjacent to $y$. We can do so as $y$ has at least $\frac{h}{2}+2|B|+2$ neighbours in $H$ and $h\ge12|B|$. Let $i\in[|B|]$. Both $u_i$ and $v_i$ have $\frac{h}{2}+2|B|+2$ neighbours in $H$, so they have at least $4|B|+4$ common neighbours. At most $3|B|+3$ of these are in $B\cup\{x,y,u_1,\dots,u_{|B|+1},v_1,\dots,v_{|B|}\}$, so we can find a thus far unused vertex $w_i$ adjacent to $u_i$ and $v_i$. We may pick the vertices $w_1,\dots,w_{|B|}$ greedily as we require only $|B|$ vertices. Hence, we obtain a path
\[P=xyu_1w_1v_1b_1u_2w_2v_2b_2\dots v_{|B|}b_{|B|}u_{|B|+1}\]
on $4|B|+3$ vertices. Notably, any cycle containing $P$ of length at least $4|B|+5$ has the desired properties.

Let $H_2:=H[V(H)\setminus(V(P)\setminus\{x,u_{|B|+1}\})]$. Since $H_2$ is a graph on $h-4|B|-1\geq4$ vertices with minimum degree $\delta(H_2)\ge\frac{h}{2}-2|B|+1\ge\frac{h-4|B|-1}{2}+1$, by Theorem~\ref{thm:williamson} $H_2$ is panconnected. Hence, $H_2$ has paths between $x$ and $u_{|B|+1}$ of every number of vertices from $4$ to $h-4|B|-1$. By adding the path $P$ to these paths, we obtain cycles of every length from $4|B|+5$ to $h$ with the desired properties.
\end{proof}

Before providing the proof of Lemma~\ref{lem:extremal} we first give an outline of our method. Recall that the Lemma is given a Szemer\'edi partition with a `near-extremal' structure. We shall show that the underlying graph either also has a `near-extremal' structure, or possesses features which lead to longer $k$th powers of paths and cycles than required for the conclusion of the Lemma. The complication we encounter is the insensitivity of the Szemer\'edi partition to misassignment of sublinearly many vertices and to editing of subquadratically many edges.

Recall that the sets $I_i$ are subsets of $V(R)$ and the elements of each set $I_i$ correspond to clusters in $V(G)$. We shall denote by $\bigcup I_i$ the union of the elements of the set $I_i$ as clusters in $V(G)$. We begin by separating those vertices with `few' neighbours in $\bigcup I_i$, which we collect in a set $W_i$, from those with `many' for each $i\in[k-1]$. We then show that if there are two vertex-disjoint edges in $W_i$, then the sets $\bigcup B_1$ and $\bigcup B_2$ `belong' to the same $K_{k+1}$-component of $G$. We shall show that this enables us to construct very long $k$th powers of paths and cycles by applying Lemma~\ref{lem:embedding}.

Hence we may assume that $W_i$ does not contain two vertex-disjoint edges, so $W_i$ is almost independent with `near-extremal' size. $W=\bigcup^{k-1}_{i=1}W_i$ now resembles a `near-extremal' interior and the minimum degree condition on $G$ will guarantee that almost every edge from $W$ to $V(G)\setminus W$ is present. At this point, we would like to say that we can find a long path outside $W$ with sufficiently nice properties (which we need because the bipartite graph $G[W,V(G)\setminus W]$ is unfortunately not actually complete) so that we can repeatedly apply Lemma~\ref{lem:power-step-up} to extend it to the $k$th power of a path (and similarly for powers of cycles) using vertices from $W$. The purpose of Lemma~\ref{lem:extr-ext-path-cycle} is precisely to provide paths and cycles with those nice properties. The rest of the proof then focuses on establishing the right conditions for the application of Lemma~\ref{lem:extr-ext-path-cycle} and working out the details of the various applications of Lemma~\ref{lem:power-step-up}.

\begin{proof}[Proof of Lemma~\ref{lem:extremal}]
Given an integer $k\geq3$ and $0<\nu<1$ let $\eta>0$ and $d>0$ satisfy
\begin{equation} \label{eqn:lem-extr-eta-d-bound}
\eta\leq\frac{\nu^4}{(k+1)^{13}10^8}\;\textrm{ and }\;d\leq\frac{\nu^4}{(k+1)^{13}10^8}.
\end{equation}
Given $k\geq3$ and $d>0$, Lemma~\ref{lem:embedding} returns a constant $\eps_{EL}>0$. Set
\begin{equation} \label{eqn:lem-extr-epsilon-bound}
\eps_0:=\min\left\{\eps_{EL},\frac{\nu^4}{(k+1)^{13}10^8}\right\}.
\end{equation}
Given $m_{EL}$ and $0<\eps<\eps_0$, Lemma~\ref{lem:embedding} returns a constant $n_{EL}$. Given $t=k$ and $\rho=\eps^{1/2}$, Theorem~\ref{thm:erdos-stone} returns a constant $n_{ES}\in\NN$. Set
\begin{equation} \label{eqn:lem-extr-n-bound}
N:=\max\left\{n_{EL},\nu^{-1}n_{ES},100m_{EL}^{k+2},100(k+1)\eta^{-1}\nu^{-1}\right\}.
\end{equation}
Let $\delta\in\left[\left(\frac{k-1}{k}+\nu\right)n,\frac{kn}{k+1}\right)$. Let $G$, $R$, and the partition $V(R)=\left(\bigcup_{i=1}^{k-1}I_i\right)\cup\left(\bigcup_{j=1}^{\ell}B_j\right)$ satisfy conditions~\ref{item:extr-lem-component-clique}--\ref{item:extr-lem-ext} of the lemma.

Note that the specific case of finding $P^k_n$ in $G$ when $\delta\geq\frac{kn-1}{k+1}$ is settled by Corollary~\ref{cor:komlos-sarkozy-szemeredi-seymour-path}. Therefore, by~\eqref{eqn:spacious-bound-threshold-hug} and~\eqref{eqn:cycle-half-bound} it would be sufficient to find
\begin{equation} \label{eqn:lem-extr-non-hamiltonian}
k\textrm{th powers of cycles and paths of all lengths up to }\frac{11n}{20}.
\end{equation}

$R$ is an $(\eps,d)$-reduced graph of $G$, so
\begin{equation} \label{eqn:lem-extr-reduced-min-deg-bound}
\delta(R)\geq\delta':=\left(\frac{\delta}{n}-d-\eps\right)m.
\end{equation}
Moreover, by~\ref{item:extr-lem-int} for each $i\in[k-1]$ clusters in $I_i$ have $\delta'$ neighbours outside $I_i$ in $R$, so
\begin{equation} \label{eqn:lem-extr-reduced-ind-set-upper-bound}
\left|I_i\right|\leq m-\delta'=\left(1-\frac{\delta}{n}+d+\eps\right)m.
\end{equation}
Set $I_J:=\bigcup_{i\in J}I_i$ for each $J\subseteq[k-1]$. By~\ref{item:extr-lem-ext} each cluster $C\in B_j$ has neighbours only in $B_j\cup I_{[k-1]}$ in $R$, so by~\eqref{eqn:lem-extr-reduced-ind-set-upper-bound} we have $\delta'\leq\mathrm{deg}(C)=\mathrm{deg}\left(C,B_j\cup I_{[k-1]}\right)\leq\mathrm{deg}(C,B_j)+\left|I_{[k-1]}\right|\leq\mathrm{deg}(C,B_j)+(k-1)(m-\delta')$. Then, by~\eqref{eqn:lem-extr-reduced-min-deg-bound} we have
\begin{align*}
|B_j|&>\mathrm{deg}(C,B_j)\geq k\delta'-(k-1)m \\
&\geq\frac{m}{n}(k\delta-(k-1)n-k(d+\eps)n).
\end{align*}
Since $\delta\ge\left(\frac{k-1}{k}+\nu\right)n$, we have $k\delta-(k-1)n\geq k\nu n$; hence, by~\eqref{eqn:lem-extr-eta-d-bound} and \eqref{eqn:lem-extr-epsilon-bound} we obtain
\begin{equation} \label{eqn:lem-extr-reduced-ext-lower-bound}
|B_j|\geq\frac{38(k\delta-(k-1)n)m}{39n}\geq\frac{38k\nu m}{39}.
\end{equation}

Set $\xi:=\sqrt[4]{d+\eps+6k\eta}$. By~\eqref{eqn:lem-extr-eta-d-bound} and \eqref{eqn:lem-extr-epsilon-bound}, we have
\begin{equation} \label{eqn:lem-extr-xi-bound}
\xi\leq\frac{\nu}{50(k+1)^3}.
\end{equation}
For each $i\in[k-1]$ define $W_i$ to be the set of vertices of $G$ with no more than $\xi n$ neighbours in $\bigcup I_i$. Since $\xi>d+\eps$, the independence of $I_i$ and the definition of an $(\eps,d)$-regular partition implies that $\bigcup I_i\subseteq W_i$. Set $W_J:=\bigcup_{i\in J}W_i$ and $I^*_J:=\bigcup_{i\in J}\left(\bigcup I_i\right)$ for each $J\subseteq[k-1]$. Note that by~\eqref{eqn:lem-extr-reduced-ind-set-upper-bound} and~\ref{item:extr-lem-int} we have
\begin{equation} \label{eqn:lem-extr-reduced-ind-set-J-union-lower-bound}
\begin{split}
\left|I_J\right|&\geq\left|I_{[k-1]}\right|-(k-1-|J|)(m-\delta') \\
&\geq\frac{m}{n}|J|(n-\delta)-5k\eta m-(k-1-|J|)(d+\eps)m
\end{split}
\end{equation}
for each $J\subseteq[k-1]$. Hence, we have
\begin{equation} \label{eqn:lem-extr-intr-ind-set-cluster-J-union-lower-bound}
\begin{split}
\left|I^*_J\right|&\geq\frac{(1-\eps)n}{m}\left|I_J\right| \\
&\geq|J|(n-\delta)-5k\eta n-(k-|J|-1)(d+\eps)n-\eps n
\end{split}
\end{equation}
for each $J\subseteq[k-1]$.

The claim below states that if there are two vertex-disjoint edges in some $W_i$, then we have two vertex-disjoint copies of $K_k$ on $W_{[k-1]}$.

\begin{claim} \label{claim:extr-edge-clique-interior}
Suppose that for some $i\in[k-1]$ there are two vertex-disjoint edges in $W_i$. Then, there are two vertex-disjoint copies of $K_k$ in $W_{[k-1]}$ each comprising two vertices of $W_i$ and a vertex of $W_h$ for each $h\in[k-1]\backslash\{i\}$.
\end{claim}

\begin{claimproof}
We consider the $i=1$ case and note that an analogous argument applies for each $i\neq1$. We prove the following statement for all $2\leq j\leq k$ by backwards induction on $j$. If there are two vertex-disjoint copies of $K_j$ on $W_{[j-1]}$ each comprising two vertices of $W_1$ and a vertex of $W_h$ for each $1<h<j$, then there are two vertex-disjoint copies of $K_k$ on $W_{[k-1]}$ each comprising two vertices of $W_1$ and a vertex of $W_h$ for each $1<h<k$. Setting $j=2$ then gives our desired statement for the $i=1$ case.

The statement is trivially true for $j=k$. Consider $2\leq j<k$. Let $u_1\dots u_j$ and $u'_1\dots u'_j$ be two vertex-disjoint copies of $K_j$ on $W_{[j-1]}$ with $u_1,u'_1\in W_1$ and $u_{i+1},u'_{i+1}\in W_i$ for each $i\in[j-1]$. By definition, $u_1$ and $u'_1$ each have at most $\xi n$ neighbours in $I^*_{\{1\}}$ and $u_{i+1}$ and $u'_{i+1}$ each have at most $\xi n$ neighbours in $I^*_{\{i\}}$ for each $i\in[j-1]$. Then, by~\eqref{eqn:lem-extr-eta-d-bound},~\eqref{eqn:lem-extr-epsilon-bound},~\eqref{eqn:lem-extr-xi-bound}~and~\eqref{eqn:lem-extr-intr-ind-set-cluster-J-union-lower-bound} we have
\begin{align*}
&\deg(u_1,\dots,u_j;W_j) \\
&\geq\sum_{i\in[j-1]}\left(\delta-n+|W_j|+\left|I^*_{\{i\}}\right|-\xi n\right)+\left(\delta-n+|W_j|+\left|I^*_{\{1\}}\right|-\xi n\right) \\
&\qquad{}-(j-1)|W_j| \\
&\geq-j(n-\delta)+|W_j|+\left|I^*_{[j-1]}\right|+\left|I^*_{\{1\}}\right|-j\xi n \\
&\geq-j(n-\delta)+\left|I^*_{[j]}\right|+\left|I^*_{\{1\}}\right|-j\xi n \\
&\geq n-\delta-10k\eta n-(j+1)(k-1)(d+\eps)n-j\xi n>1.
\end{align*}
An analogous argument gives
\begin{align*}
&\deg(u'_1,\dots,u'_j;W_j) \\
&\geq n-\delta-10k\eta n-(j+1)(k-1)(d+\eps)n-j\xi n>1.
\end{align*}
Hence, we have two vertices $u_{j+1}\in\Gamma(u_1,\dots,u_j;W_j)$ and $u'_{j+1}\in\Gamma(u'_1,\dots,u'_j;W_j)$ which are distinct. Notice that $u_1\dots u_{j+1}$ and $u'_1\dots u'_{j+1}$ are two vertex-disjoint copies of $K_{j+1}$ on $W_{[j]}$ each comprising two vertices of $W_1$ and a vertex of $W_h$ for each $1<h\leq j$, so by the inductive hypothesis there are two vertex-disjoint copies of $K_k$ on $W_{[k-1]}$ each comprising two vertices of $W_1$ and a vertex of $W_h$ for each $1<h<k$, completing the proof.
\end{claimproof}

Now suppose that for some $i\in[k-1]$ we have a copy $u_1\dots u_k$ of $K_k$ on $W_{[k-1]}$ with two vertices of $W_i$ and a vertex of $W_h$ for each $h\in[k-1]\backslash\{i\}$. We shall consider the $i=1$ case and note that for each $i\neq1$ an analogous version of the following argument applies. Without loss of generality, let $u_1\in W_1$ and $u_{i+1}\in W_i$ for $i\in[k-1]$. We shall count the common neighbours of $u_1\dots u_k$ outside $I^*_{[k-1]}$. By definition $u_1$ has at most $\xi n$ neighbours in $I^*_{\{1\}}$ and $u_{i+1}$ has at most $\xi n$ neighbours in $I^*_{\{i\}}$ for each $i\in[k-1]$. Then, \eqref{eqn:lem-extr-eta-d-bound}, \eqref{eqn:lem-extr-epsilon-bound}, \eqref{eqn:lem-extr-xi-bound}, \eqref{eqn:lem-extr-intr-ind-set-cluster-J-union-lower-bound} and the fact that $k\delta-(k-1)n\geq k\nu n$ imply that $u_1\dots u_k$ has at least
\begin{equation} \label{eqn:lem-extr-clique-nbrhood-size}
\begin{split}
&\sum_{i\in[k-1]}\left(\delta-\left|I^*_{[k-1]}\right|+\left|I^*_{\{i\}}\right|-\xi n\right)+\left(\delta-\left|I^*_{[k-1]}\right|+\left|I^*_{\{1\}}\right|-\xi n\right) \\
&\qquad{}-(k-1)\left(n-\left|I^*_{[k-1]}\right|\right) \\
&\geq(k-1)\delta-(k-2)n-\frac{k\delta-(k-1)n}{48}.
\end{split}
\end{equation}
common neighbours outside $I^*_{[k-1]}$. Now the following claim tells us that we are done if we can find two vertex-disjoint copies of $K_k$ which satisfy~\eqref{eqn:lem-extr-clique-nbrhood-size}.

\begin{claim} \label{claim:extr-disjoint-cliques-cycle-path-power}
Suppose that $u_1\dots u_k$ and $u'_1\dots u'_k$ are vertex-disjoint copies of $K_k$ in $G$ such that each of them has at least $(k-1)\delta-(k-2)n-\frac{k\delta-(k-1)n}{48}$ common neighbours outside $I^*_{[k-1]}$. Then $G$ contains $P^k_{\pp_k(n,\delta)}$ and $C^k_\ell$ for each $\ell\in[k+1,\pc_k(n,\delta)]$ such that $\chi(C^k_\ell)\leq k+2$.
\end{claim}

\begin{claimproof}
Let $D'$ be the set of clusters $C\in V(R)\setminus I_{[k-1]}$ such that $u_1\dots u_k$ has at most $\frac{2dn}{m}$ common neighbours in $C$. By the hypothesis, $u_1\dots u_k$ has at least $(k-1)\delta-(k-2)n-\frac{k\delta-(k-1)n}{48}$ common neighbours outside $I^*_{[k-1]}$. Of these, at most $\eps n$ are in the exceptional set $V_0$ of the regular partition, and at most $\frac{2dn|D'|}{m}$ are in $\bigcup D'$. The remaining common neighbours all lie in $\bigcup\left(V(R)\setminus(I_{[k-1]}\cup D')\right)$, so by~\ref{item:extr-lem-int} we have the inequality
\begin{align*}
&(k-1)\delta-(k-2)n-\frac{k\delta-(k-1)n}{48}-\eps n-\frac{2dn|D'|}{m} \\
&\leq (m-\left|I_{[k-1]}\right|-|D'|)\frac{n}{m}\leq n-(k-1)(n-\delta)+5k\eta n-|D'|\frac{n}{m}.
\end{align*}
Simplifying this, we obtain
\[(1-2d)\frac{n}{m}|D'|\leq\eps n+5k\eta n+\frac{k\delta-(k-1)n}{48},\]
and so by~\eqref{eqn:lem-extr-eta-d-bound} and~\eqref{eqn:lem-extr-epsilon-bound} we have $|D'|\leq\frac{(k\delta-(k-1)n)m}{40n}$.

Now let $D$ be the set of clusters $C\in V(R)\setminus I_{[k-1]}$ such that either $u_1\dots u_k$ or $u'_1\dots u'_k$ has at most $\frac{2dn}{m}$ common neighbours in $C$. Since the same analysis holds for $u'_1\dots u'_k$, we obtain
\begin{equation} \label{eqn:lem-extr-cluster-few-nbrs-upper-bound}
|D|\leq\frac{(k\delta-(k-1)n)m}{20n}.
\end{equation}

We now show that there is a copy $X_1\dots X_{k-2}$ of $K_{k-2}$ in $R$ such that $X_j\in I_j\setminus D$ for each $j\in[k-2]$. In fact, we prove the following statement for all $i\in[k-2]$ by backwards induction on $i$: there is a copy $X_1\dots X_i$ of $K_i$ in $R$ such that $X_j\in I_j\setminus D$ for each $j\in[i]$. Setting $i=k-2$ then gives the desired statement.

Consider $i=1$. From~\eqref{eqn:lem-extr-reduced-ind-set-J-union-lower-bound} and~\eqref{eqn:lem-extr-cluster-few-nbrs-upper-bound} we conclude that
\begin{align*}
|I_1\setminus D|&\geq\frac{m}{n}\left(n-\delta-5k\eta n-(k-2)(d+\eps)n-\frac{k\delta-(k-1)n}{20}\right) \\
&\geq\frac{m}{n}\left(n-\delta-\frac{k\delta-(k-1)n}{10}\right)>0,
\end{align*}
so we may choose $X_1\in I_1\setminus D$. Now consider $1<i\leq k-2$. By the induction hypothesis, there is a copy $X_1\dots X_{i-1}$ of $K_{i-1}$ such that $X_j\in I_j\setminus D$ for each $j\in[i-1]$. By~\ref{item:extr-lem-int} $I_j$ is an independent set for each $j\in[i-1]$, so $\Gamma(X_1,\dots,X_{i-1})\cap I_{[i-1]}=\nth$. Then applying Lemma~\ref{lem:common-nbrhood-size}, \eqref{eqn:lem-extr-eta-d-bound}, \eqref{eqn:lem-extr-epsilon-bound}, \eqref{eqn:lem-extr-reduced-min-deg-bound}, \eqref{eqn:lem-extr-reduced-ind-set-J-union-lower-bound} and~\eqref{eqn:lem-extr-cluster-few-nbrs-upper-bound}, we obtain
\begin{align*}
\deg(X_1,\dots,X_{i-1};I_i)&\geq\deg(X_1,\dots,X_{i-1})-m+\left|I_{[i]}\right| \\
&\geq\left|I_{[i]}\right|-(i-1)(m-\delta') \\
&\geq\frac{m}{n}((k-1)(n-\delta)-5k\eta n)-(k-2)(m-\delta') \\
&=\frac{m}{n}(n-\delta-(k-2)(d+\eps)n-5k\eta n) \\
&\geq\frac{(k\delta-(k-1)n)m}{2n}>|D|,
\end{align*}
so we may pick $X_i\in\Gamma(X_1,\dots,X_{i-1})\cap \left(I_i\setminus D\right)$. Then, $X_1\dots X_i$ is a copy of $K_i$ such that $X_j\in I_j\setminus D$ for each $j\in[i]$, concluding our inductive proof.

Hence, there is a copy $X_1\dots X_{k-2}$ of $K_{k-2}$ such that $X_j\in I_j\setminus D$ for each $j\in[k-2]$. By~\ref{item:extr-lem-int} $I_j$ is an independent set for each $j\in[k-1]$, so $\Gamma(X_1,\dots,X_{k-2})\cap I_{[k-2]}=\nth$. Now by Lemma~\ref{lem:common-nbrhood-size}, \eqref{eqn:lem-extr-eta-d-bound}, \eqref{eqn:lem-extr-epsilon-bound},
\eqref{eqn:lem-extr-reduced-min-deg-bound}, \eqref{eqn:lem-extr-reduced-ext-lower-bound},
\eqref{eqn:lem-extr-reduced-ind-set-J-union-lower-bound} and~\eqref{eqn:lem-extr-cluster-few-nbrs-upper-bound}, we have
\begin{align*}
&\deg(X_1,\dots,X_{k-2};B_1) \\
&\geq\deg(X_1,\dots,X_{k-2})-m+|B_1|+\left|I_{[k-2]}\right| \\
&\geq|B_1|+\left|I_{[k-2]}\right|-(k-2)(m-\delta') \\
&\geq|B_1|+\frac{m}{n}((k-1)(n-\delta)-5k\eta n)-(k-1)(m-\delta') \\
&=|B_1|-(5k\eta+(k-1)(d+\eps))m \\
&\geq\frac{(k\delta-(k-1)n)m}{2n}>|D|,
\end{align*}
so we may pick $X\in\Gamma(X_1,\dots,X_{k-2})\cap \left(B_1\setminus D\right)$. By Lemma~\ref{lem:common-nbrhood-size},
\ref{item:extr-lem-ext}, \eqref{eqn:lem-extr-eta-d-bound}, \eqref{eqn:lem-extr-epsilon-bound},
\eqref{eqn:lem-extr-reduced-min-deg-bound},
\eqref{eqn:lem-extr-reduced-ind-set-upper-bound} and~\eqref{eqn:lem-extr-cluster-few-nbrs-upper-bound} we have
\begin{align*}
\deg(X_1,\dots,X_{k-2},X;B_1)&\geq\deg(X_1,\dots,X_{k-2})-|I_{k-1}| \\
&\geq k\delta'-(k-1)m \\
&\geq\frac{(k\delta-(k-1)n)m}{2n}>|D|,
\end{align*}
so we may pick $Y\in\Gamma(X_1,\dots,X_{k-2},X)\cap \left(B_1\setminus D\right)$. An analogous argument tells us that we may pick $X'\in\Gamma(X_1,\dots,X_{k-2})\cap \left(B_2\setminus D\right)$ and $Y'\in\Gamma(X_1,\dots,X_{k-2},X')\cap \left(B_2\setminus D\right)$. Therefore, we have copies $X_1\dots X_{k-2}XY$ and $X_1\dots X_{k-2}X'Y'$ of $K_k$ such that $X_j\in I_j\setminus D$ for each $j\in[k-2]$, $X,Y\in B_1\setminus D$ and $X',Y'\in B_2\setminus D$.

Since $\delta_R(B_1),\delta_R(B_2) \ge \delta'-|I|$ and $|B_i| > \delta_R(B_i)$ for all $i\in[2]$, by Lemma~\ref{lem:min-deg-matching}\ref{lem:min-deg-matching-gen} we can find a matching $F_2:=M$ in $R[B_1\cup B_2]$ with $\delta'-|I|$ edges. Using a step-by-step process with steps $1,\dots,k-1$, we will extend the edges of $F_2$ to copies of $K_{k+1}$ each consisting of an edge of $F_2$ and exactly one vertex from each $I_i$. The final collection of copies of $K_{k+1}$ will have size at least $k\delta'-(k-1)m-5k\eta m$. Let $i\in[k-1]$ and let $F_{i+1}$ be the set of at least $i\delta'-(i-1)m-\sum_{h=i}^{k-1}|I_h|-5k\eta m$ vertex-disjoint copies of $K_{i+1}$ which we have immediately before step $i$. Every cluster in $I_i$ has at most $m-|I_i|-\delta'$ non-neighbours outside $I_i$. Hence, every cluster in $|I_i|$ forms a copy of $K_{i+2}$ with at least $|F_{i+1}|-(m-|I_i|-\delta')\geq(i+1)\delta'-im-\sum_{h=i+1}^{k-1}|I_h|-5k\eta m$ copies of $K_{i+1}$ of $F_{i+1}$. Therefore, we may choose greedily clusters in $I_i$ to obtain a set $F_{i+2}$ of at least
\begin{align*}
&\min\left\{(i+1)\delta'-im-\sum_{h=i+1}^{k-1}|I_h|-5k\eta m,|I_i|\right\} \\
&\geq(i+1)\delta'-im-\sum_{h=i+1}^{k-1}|I_h|-5k\eta m
\end{align*}
vertex-disjoint copies of $K_{i+2}$ formed from copies of $K_{i+1}$ of $F_{i+1}$ and clusters of $I_i$. After step $k-1$, we have a set $T:=F_{k+1}$ of at least $k\delta'-(k-1)m-5k\eta m$ vertex-disjoint copies of $K_{k+1}$ each comprising an edge of $M$ and a vertex from $I_i$ for each $i\in[k-1]$. Let $T_1$ be the collection of the copies of $K_{k+1}$ of $T$ contained in $B_1\cup I_{[k-1]}$ and $T_2$ the collection of those contained in $B_2\cup I_{[k-1]}$. By~\ref{item:extr-lem-ext}, all the copies of $K_{k+1}$ in $T_1$ are in the same $K_{k+1}$-component as $X_1\dots X_{k-2}XY$ and all the copies of $K_{k+1}$ in $T_2$ are in the same $K_{k+1}$-component as $X_1\dots X_{k-2}X'Y'$.

Apply Lemma~\ref{lem:embedding} with $X_{i,j}=X_j$ for $(i,j)\in[2]\times[k-2]$ and $X_{i,k-1}=X,X_{i,k}=Y$ for $i\in[2]$ to find the $k$th power of a path starting with $u_1\dots, u_k$ and ending with $u'_1\dots u'_k$ using the copies of $K_{k+1}$ in $T_1$. Similarly, apply Lemma~\ref{lem:embedding} with $X_{i,j}=X_j$ for $(i,j)\in[2]\times[k-2]$, $X_{i,k-1}=X',X_{i,k}=Y'$ for $i\in[2]$ and $A$ as the set of vertices of the $k$th power of a path we have above which are not in $\bigcup T_1$, to find the $k$th power of a path starting with $u'_1\dots, u'_k$ and ending with $u_1\dots u_k$ using the copies of $K_{k+1}$ in $T_2$, intersecting the first only at $u_1,\dots,u_k,u'_1,\dots,u'_k$. Choosing appropriate lengths for these $k$th power of paths and concatenating them yields the $k$th power of a cycle $C_\ell^k$ for any $6(k+1)m^{k+1}\leq\ell\leq(k+1)(1-d)\left(k\delta'-(k-1)m-5k\eta m\right)\frac{n}{m}$. Applying Lemma~\ref{lem:embedding} to a copy of $K_{k+2}$ in a $K_{k+1}$-component directly yields $C^k_\ell$ for each $k+1\leq\ell\leq(k+1)(1-d)\frac{n}{m}$ such that $\chi(C_\ell^k)\leq k+2$. By~\eqref{eqn:lem-extr-eta-d-bound} and~\eqref{eqn:lem-extr-n-bound} we have $(k+1)(1-d)\frac{n}{m}\geq6(k+1)m^{k+1}$, and by~\eqref{eqn:lem-extr-eta-d-bound}, \eqref{eqn:lem-extr-epsilon-bound} we have $(k+1)(1-d)\left(k\delta'-(k-1)m-5k\eta m\right)\frac{n}{m}\geq\pc_k(n,\delta)$. It follows that $G$ contains $C^k_\ell$ for each $\ell\in[k+1,\pc_k(n,\delta)]$ such that $\chi(C^k_\ell)\leq k+2$. For the case $\delta\in\left[\left(\frac{k-1}{k}+\nu\right)n,\frac{kn-1}{k+1}\right)$, note that by~\eqref{eqn:lem-extr-eta-d-bound}, \eqref{eqn:lem-extr-epsilon-bound} we have $(k+1)(1-d)\left(k\delta'-(k-1)m-5k\eta m\right)\frac{n}{m}\geq\pp_k(n,\delta)$, so $G$ contains $P^k_{\pp_k(n,\delta)}$.
\end{claimproof}

By Claim~\ref{claim:extr-edge-clique-interior} and~\eqref{eqn:lem-extr-clique-nbrhood-size}, if we can find two vertex-disjoint edges in some $W_i$, then we are done by Claim~\ref{claim:extr-disjoint-cliques-cycle-path-power}. Hence, we assume in the following that $W_i$ does not contain two vertex-disjoint edges for each $i\in[k-1]$. This means that for each $i\in[k-1]$ there are two vertices in $W_i$ which meet every edge in $W_i$. For each $i\in[k-1]$ let $W'_i$ be $W_i$ without these two vertices. Since neither of these two vertices has more than $\xi n$ neighbours in $I^*_{\{i\}}\subseteq W_i$, while $|I_i|\geq\frac{m}{k+1}-5k\eta m-(k-2)(d+\eps)m$ by~\eqref{eqn:lem-extr-reduced-ind-set-J-union-lower-bound} and because $\delta<\frac{kn}{k+1}$, there is a vertex in $W_i$ adjacent to no vertex of $W_i$. By~\eqref{eqn:lem-extr-intr-ind-set-cluster-J-union-lower-bound} we conclude that
\begin{equation} \label{eqn:lem-extr-int-ind-set-J-union-bounds}
\begin{split}
&|J|(n-\delta)-5k\eta n-(k-1-|J|)(d+\eps)n-\eps n \\ &\leq\left|I^*_J\right|\leq|W_J|\leq|J|(n-\delta)
\end{split}
\end{equation}
for each $J\subseteq[k-1]$. Set $W:=W_{[k-1]}$. For each $i\in[k-1]$ the total number of non-edges between $W_i$ and $V(G)\setminus W_i$ is at most
\begin{align*}
|W_i||V(G)\setminus W_i|-|W_i|(\delta-2)&=|W_i|(n-\delta+2-|W_i|) \\
&\leq|W_i|((k-1)(n-\delta)+2-|W|).
\end{align*}
Hence, by~\eqref{eqn:lem-extr-int-ind-set-J-union-bounds} the total number of non-edges between $W$ and $V(G)\setminus W$ is at most
\begin{align*}
|W|\big((k-1)(n-\delta)+2-|W|\big)&\leq|W|(5k\eta n+\eps n+2) \\
&\leq5k\eta n^2+\eps n^2+2n.
\end{align*}
In particular, by the definition of $\xi$ and~\eqref{eqn:lem-extr-n-bound}, we have
\begin{equation} \label{eqn:lem-extr-int-ext-almost-complete}
\left|\left\{v\in V(G)\backslash W:\deg(v;W)<|W|-\xi^2n\right\}\right|\leq\xi^2n.
\end{equation}
Recall that the sets $B_i$ are subsets of $V(R)$ and the elements of each set $B_i$ correspond to clusters in $V(G)$. We shall denote by $\bigcup B_i$ the union of the elements of the set $B_i$ as clusters in $V(G)$. By~\ref{item:extr-lem-ext} we have $|B_i|\leq\frac{19m(k\delta-(k-1)n)}{10n}$, which together with $\delta<\frac{kn}{k+1}$, \eqref{eqn:lem-extr-eta-d-bound}, \eqref{eqn:lem-extr-epsilon-bound} and \eqref{eqn:lem-extr-xi-bound} implies
\begin{equation} \label{eqn:lem-extr-ext-cluster-union-upper-bound}
\begin{split}
\left|\bigcup B_i\right|&\leq\frac{19}{10}(k\delta-(k-1)n)\leq\frac{19}{20}((k-1)\delta-(k-2)n) \\
&<(k-1)\delta-(k-2)n-\xi n-(d+\eps)n.
\end{split}
\end{equation}
By~\ref{item:extr-lem-ext} and the definition of an $(\eps,d)$-regular partition, vertices in $\bigcup B_i$ have at most $(d+\eps)n$ neighbours outside of $\left(\bigcup B_i\right)\cup I^*_{[k-1]}$; hence, by $\delta(G)\ge\delta$, \eqref{eqn:lem-extr-int-ind-set-J-union-bounds} and~\eqref{eqn:lem-extr-ext-cluster-union-upper-bound} they have more than $\xi n$ neighbours in $\bigcup I_h$ for all $h\in[k-1]$. Now the definition of $W_h$ implies $\bigcup B_i\cap W_h=\nth$ for all $(i,h)\in[\ell]\times[k-1]$, so in fact
\begin{equation} \label{eqn:lem-extr-B_i-W}
\bigcup B_i\cap W=\nth\textrm{ for all }i\in[\ell].
\end{equation}
Furthermore, \eqref{eqn:lem-extr-eta-d-bound}, \eqref{eqn:lem-extr-epsilon-bound}, \eqref{eqn:lem-extr-xi-bound}, \eqref{eqn:lem-extr-int-ind-set-J-union-bounds} and~\eqref{eqn:lem-extr-ext-cluster-union-upper-bound} imply that $v\in\bigcup B_i$ has at least
\begin{equation} \label{eqn:lem-extr-ext-min-deg}
\delta-|W|-(d+\eps)n\geq k\delta-(k-1)n-(d+\eps)n>\left|\bigcup B_i\right|/2+50\xi^2n
\end{equation}
neighbours in $\bigcup B_i$.

Now for each $i\in[\ell]$ let $A_i$ be the set of vertices in $\bigcup B_i$ which are adjacent to at least $|W|-\xi^2n$ vertices of $W$. By~\eqref{eqn:lem-extr-int-ext-almost-complete} we have
\begin{equation} \label{eqn:lem-extr-ext-not-almost-complete-vtxs-few}
\left|\bigcup_{i\in[\ell]}\left(\bigcup B_i\right)\backslash A_i\right|\leq\xi^2n.
\end{equation}
Vertices which are neither in $W$ nor in any of the sets $A_i$ must either be in the exceptional set $V_0$ or in $\left(\bigcup B_i\right)\backslash A_i$ for some $i$, so we have
\begin{equation} \label{eqn:lem-extr-tricky-vtxs-few}
\left|V_0\cup\bigcup_{i\in[\ell]}\left(\bigcup B_i\right)\backslash A_i\right|\leq\eps n+\xi^2n<2\xi^2n.
\end{equation}
As such, \eqref{eqn:lem-extr-ext-min-deg} implies that
\begin{equation} \label{eqn:lem-extr-dense-ext-min-deg}
\delta(G[A_i])\geq|A_i|/2+48\xi^2n
\end{equation}
and since $|B_i|>\delta'-\left|I_{[k-1]}\right|\geq k\delta'-(k-1)m$, we have
\begin{equation} \label{eqn:lem-extr-dense-ext-size-min}
|A_i|\geq\left|\bigcup B_i\right|-\xi^2n\geq(1-\eps)\frac{n}{m}|B_i|-\xi^2n\geq k\delta-(k-1)n-2\xi^2n.
\end{equation}
for each $i\in[\ell]$, where we have used~\eqref{eqn:lem-extr-eta-d-bound}, \eqref{eqn:lem-extr-epsilon-bound}, \eqref{eqn:lem-extr-reduced-min-deg-bound} and the definition of $\xi$.

The following claim uses $A_1$ to obtain powers of cycles of all lengths up to near-extremal. 

\begin{claim} \label{claim:A_1-cycle-power}
$C^k_\ell\subseteq G$ for each $\ell\in[k+1,\frac{k+1}{2}|A_1|]$ such that $\chi(C^k_{\ell}) \le k+2$.
\end{claim}

\begin{claimproof}
By Lemma~\ref{lem:extr-ext-path-cycle} (with $B=\nth$) we find in $A_1$ a copy of $C_{2h'}$ for all $2h'\in[4,\min\{|A_1|,\frac{2n}{k+2}\}]$. We shall construct a copy of $C^k_{(k+1)h'}$ from this cycle by repeated application of Lemma~\ref{lem:power-step-up}. We have steps $j=1,\dots,k-1$. In step $j$ we start with a copy of $C^j_{(j+1)h'}$
\[T_j=q_{1,j-1}\dots q_{1,1}t_1t_2\dots q_{h',j-1}\dots q_{h',1}t_{2h'-1}t_{2h'}\]
in $G$, with $t_i\in A_1$ for $i\in[2h']$, $q_{f,g}\in W_g$ for each $f\in[h'],g\in[j-1]$, such that each vertex is adjacent to the immediately preceding $j$ vertices in cyclic order.

Any $2(j+1)$-tuple of consecutive vertices on $T_j$ comprises four vertices from $A_1$ and two vertices from $W_i$ for each $i\in[j-1]$. Each vertex in $A_1$ has at least $|W_j|-\xi^2 n$ neighbours in $W_j$, while for each $i\in[j-1]$ a vertex in $W_i$ has at least $|W_j|-(n-\delta)+\left|I^*_{\{i\}}\right|-\xi n$ neighbours in $W_j$. Applying Lemma~\ref{lem:common-nbrhood-size} and~\eqref{eqn:lem-extr-intr-ind-set-cluster-J-union-lower-bound}, we find that every $2(j+1)$-tuple of consecutive vertices on $T_j$ has at least
\begin{align*}
&|W_j|-4\xi^2 n-2(j-1)\xi n-2(j-1)(n-\delta)+2\left|I^*_{[j-1]}\right| \\
&\geq|W_j|-4\xi^2 n-2(j-1)\xi n-10k\eta n-2(k-j)(d+\eps)n-2\eps n
\end{align*}
common neighbours in $W_j$. Since $\delta<\frac{kn}{k+1}$ and by~\eqref{eqn:lem-extr-eta-d-bound}, \eqref{eqn:lem-extr-epsilon-bound}, \eqref{eqn:lem-extr-xi-bound} and \eqref{eqn:lem-extr-int-ind-set-J-union-bounds}, we have
\begin{equation*}
|W_j|-4\xi^2 n-2(j-1)\xi n-10k\eta n-2(k-j)(d+\eps)n-2\eps n\geq\frac{n}{k+2}.
\end{equation*}
This means that we can apply Lemma~\ref{lem:power-step-up} with $G$ and $W_j$ to obtain a copy of $C^{j+1}_{(j+2)h'}$
\[T_{j+1}=q_{1,j}\dots q_{1,1}t_1t_2\dots q_{h',j}\dots q_{h',1}t_{2h'-1}t_{2h'}\]
in $G$, with $t_i\in A_1$ for $i\in[2h']$, $q_{f,g}\in W_g$ for each $f\in[h'],g\in[j]$, such that each vertex is adjacent to the preceding $j$ vertices in cyclic order. Terminating after step $k-1$ gives us a copy of $C^k_{(k+1)h'}$. Hence, we are able to find copies of $C^k_h$ for $h\in[k+1,\frac{k+1}{2}\min\{|A_1|,\frac{2n}{k+2}\}]$ such that $h$ is divisible by $k+1$.

To obtain a copy of $C^k_h$ for $h$ not divisible by $k+1$, we perform a procedure which we will call parity correction. Fix $g\in[k]$. We seek a copy of $C^k_{(k+1)h'+g}$ with $h'\geq g$. Let $h'':=h'-g$. Pick (by Theorem~\ref{thm:erdos-stone}) vertices $a_{i,j}$ for $(i,j)\in[g]\times[3]$ in $A_1$ such that $a_{i,1}a_{i,2}a_{i,3}$ is a triangle for each $i\in[g]$ and $a_{i,3}a_{i+1,1}$ is an edge for $i\in[g-1]$. Let $A=\{a_{i,j}\mid(i,j)\in[g]\times[3]\}$. Apply Lemma~\ref{lem:extr-ext-path-cycle} to find a path $P'_1=a_{1,1}p_{2h''}\dots p_1a_{g,3}$ in $(A_1\backslash A)\cup\{a_{1,1},a_{g,3}\}$ on $2(h''+1)$ vertices whose end-vertices are $a_{1,1}$ and $a_{g,3}$. For each $a\in A$, insert a dummy vertex $a'$ into $G$ with the same adjacencies as $a$. Define $P^{(i)}_1:=a_{i+1,2}a_{i+1,1}a_{i,3}a'_{i,2}a'_{i,1}$ for $i\in[g-1]$ and $P_1:=a_{1,2}P'_1a'_{g,2}a'_{g,1}$.

We shall construct a copy of $C^k_{(k+1)h'+g}$ from these paths by repeatedly applying Lemma~\ref{lem:power-step-up} and suitably truncating and concatenating the resultant $k$th powers of paths. We have steps $j=1,\dots,k-1$. In step $1$ we start with the paths $P_1,P^{(1)}_1,\dots,P^{(g-1)}_1$. For $P_1$ set $Q_i=p_{2(h''-i+3)}p_{2(h''-i+3)-1}p_{2(h''-i+2)}p_{2(h''-i+2)-1}$ for $3\leq i\leq h''+1$ as well as $Q_1=a_{1,2}a_{1,1}$, $Q_2=a_{1,2}a_{1,1}p_{2h''}p_{2h''-1}$ and $Q_{h''+2}=p_2p_1a_{g,3}a'_{g,2}a'_{g,1}$. Then, apply Lemma~\ref{lem:power-step-up} with $W_1$ to obtain the squared path
\[qa_{1,2}a_{1,1}q_{h'',1}p_{2h''}p_{2h''-1}\dots q_{1,1}p_2p_1q^{(g)}_1a_{g,3}a'_{g,2}a'_{g,1},\]
with $q^{(g)}_1,q_{x,1}\in W_1$ for each $x\in[h'']$, such that each vertex is adjacent to the preceding 2 vertices in cyclic order and $q^{(g)}_1$ adjacent to $a'_{g,1}$. Let $P_2$ be the result of replacing $q$ in the above squared path with $a'_{1,3}$. For $P^{(i)}_1$ with $i\in[g-1]$, take $Q_1=a_{i+1,2}a_{i+1,1}$, $Q_2=a_{i+1,2}a_{i+1,1}a_{i,3}a'_{i,2}a'_{i,1}$, and apply Lemma~\ref{lem:power-step-up} with $W_1$ to obtain the squared path
\[qa_{i+1,2}a_{i+1,1}q^{(i)}_1a_{i,3}a'_{i,2}a'_{i,1},\]
such that $q^{(i)}_1\in W_1$ adjacent to $a'_{g,1}$ and each vertex is adjacent to the preceding 2 vertices in cyclic order. Let $P^{(i)}_2$ be the result of replacing $q$ in the above squared path with $a'_{i+1,3}$.

In step $j\geq2$ we start with $j$th powers of paths
\begin{align*}
P_j=&(q^{(1)}_{j-2})'\dots (q^{(1)}_{1})'a'_{1,3}a_{1,2}a_{1,1}q_{h'',j-1}\dots q_{h'',1}p_{2h''}p_{2h''-1} \\
&\dots q_{1,j-1}\dots q_{1,1}p_2p_1q^{(g)}_{j-1}\dots q^{(g)}_{1}a_{g,3}a'_{g,2}a'_{g,1}, \\
P^{(i)}_j=&(q^{(i+1)}_{j-2})'\dots (q^{(i+1)}_{1})'a'_{i+1,3}a_{i+1,2}a_{i+1,1}q^{(i)}_{j-1}\dots q^{(i)}_{1}a_{i,3}a'_{i,2}a'_{i,1}
\end{align*}
for each $i\in[g-1]$. We seek to apply Lemma~\ref{lem:power-step-up} with $W_j$ to each of them. For $P_j$ take 
$Q_1=(q^{(1)}_{j-2})'\dots (q^{(1)}_1)'a'_{1,3}a_{1,2}a_{1,1}$,
\begin{align*}
Q_2&=(q^{(1)}_{j-2})'\dots (q^{(1)}_1)'a'_{1,3}a_{1,2}a_{1,1}q_{h'',j-1}\dots q_{h'',1}p_{2h''}p_{2h''-1}, \\
Q_i&=q_{h''-i+3,j-1}\dots q_{h''-i+3,1}p_{2(h''-i+3)}p_{2(h''-i+3)-1} \\
&\quad\;q_{h''-i+2,j-1}\dots q_{h''-i+2,1}p_{2(h''-i+2)}p_{2(h''-i+2)-1}
\end{align*}
for each $3\leq i\leq h''+1$, and
\[Q_{h''+2}=q_{1,j-1}\dots q_{1,1}p_2p_1q^{(g)}_{j-1}\dots q^{(g)}_1a_{g,3}a'_{g,2}a'_{g,1}.\]
Applying Lemma~\ref{lem:power-step-up} with $W_j$ yields the $(j+1)$st power of a path
\begin{align*}
&q(q^{(1)}_{j-2})'\dots (q^{(1)}_{1})'a'_{1,3}a_{1,2}a_{1,1}q_{h'',j}\dots q_{h'',1}p_{2h''}p_{2h''-1} \\
&\dots q_{1,j}\dots q_{1,1}p_2p_1q^{(g)}_j\dots q^{(g)}_{1}a_{g,3}a'_{g,2}a'_{g,1}
\end{align*}
with $q^{(g)}_j,q_{x,j}\in W_1$ for each $x\in[h'']$, such that each vertex is adjacent to the preceding $j+1$ vertices in cyclic order and $q^{(g)}_j$ adjacent to $a'_{g,1}$. Insert a dummy vertex $(q^{(1)}_{j-1})'$ into $G$ with the same adjacencies as $q^{(1)}_{j-1}$. Define $P_{j+1}$ to be the above $(j+1)$st power of a path with $q$ replaced by $(q^{(1)}_{j-1})'$. For $P^{(i)}_j$ with $i\in[g-1]$, take
\begin{align*}
&Q_1=(q^{(i+1)}_{j-2})'\dots (q^{(i+1)}_1)'a'_{i+1,3}a_{i+1,2}a_{i+1,1}\textrm{ and} \\
&Q_2=(q^{(i+1)}_{j-2})'\dots (q^{(i+1)}_1)'a_{i+1,3}a_{i+1,2}a_{i+1,1}q^{(i)}_{j-1}\dots q^{(i)}_1a_{i,3}a'_{i,2}a'_{i,1}.
\end{align*}
Applying Lemma~\ref{lem:power-step-up} with $W_j$ yields the $(j+1)$st power of a path
\[q(q^{(i+1)}_{j-2})'\dots (q^{(i+1)}_1)'a'_{i+1,3}a_{i+1,2}a_{i+1,1}q^{(i)}_j\dots q^{(i)}_1a_{i,3}a'_{i,2}a'_{i,1}\]
such that $q^{(i)}_j\in W_j$ adjacent to $a'_{g,1}$ and each vertex is adjacent to the preceding 2 vertices in cyclic order. Insert a dummy vertex $(q^{(i+1)}_{j-1})'$ into $G$ with the same adjacencies as $q^{(i+1)}_{j-1}$. Define $P^{(i)}_{j+1}$ to be the above $(j+1)$st power of a path with $q$ replaced by $(q^{(i+1)}_{j-1})'$.

After step $k-1$, we have $k$th powers of paths $P_{k-1},P^{(1)}_{k-1},\dots,P^{(g-1)}_{k-1}$. We delete the cloned vertices from each of them and concatenate the resultant $k$th powers of paths to obtain the $k$th power of a cycle on $(k+1)h'+g$ vertices. Therefore, we can obtain $C^k_{\ell'}$ for every $\ell'\in[k+1,\frac{k+1}{2}\min\{|A_1|,\frac{2n}{k+2}\}]$ such that $\chi(C^k_{\ell'})\leq k+2$. Since $\pc_k(n,\delta)\leq\frac{(k+1)n}{k+2}$ by~\eqref{eqn:lem-extr-non-hamiltonian}, we obtain the desired result.
\end{claimproof}

It remains to show that we have $C^k_{\ell'}\subseteq G$ for every $\frac{k+1}{2}|A_1|\leq\ell'\leq\pc_k(n,\delta)$ and that in the case $\delta\in\left[\left(\frac{k-1}{k}+\nu\right)n,\frac{kn-1}{k+1}\right)$ we have $P^k_{\pp_k(n,\delta)}\subseteq G$. To do so, we need to incorporate vertices which are not `nice' enough to be included in the sets $A_i$. Define $X_i$ as $A_i$ together with all vertices in $V(G)\backslash W$ with at least $30\xi^2n$ neighbours in $A_i$. Every vertex of $V(G)\backslash W$ has at least $\delta-|W|$ neighbours outside $W$, so by~\eqref{eqn:lem-extr-int-ind-set-J-union-bounds} every vertex of $V(G)\backslash W$ is in $X_i$ for at least one $i$. Let $i,j\in[\ell]$ satisfy $i\neq j$. Since $A_h\subseteq\bigcup B_h$, we have $A_i\cap A_j=\nth$. By the definition of an $(\eps,d)$-regular partition and~\ref{item:extr-lem-ext}, vertices in $A_i$ have at most $(d+\eps)n$ neighbours outside of $\left(\bigcup B_i\right)\cup I^*_{[k-1]}$; by~\eqref{eqn:lem-extr-B_i-W} vertices in $A_i$ have at most $(d+\eps)n < 30\xi^2n$ neighbours in $A_j$. Hence, we have
\begin{equation}\label{eqn:lem-extr-expanded-ext-dense-other-ext-disjoint}
A_i\cap X_j=\nth.
\end{equation}
Then, it follows from~\eqref{eqn:lem-extr-tricky-vtxs-few} that
\begin{equation} \label{eqn:lem-extr-expanded-ext-bound}
|X_i|<|A_i|+2\xi^2n.
\end{equation}
We shall now show the desired outcome by considering three cases based on the values of $|X_i\cap X_j|$. The following claim deals with the case when $|X_i\cap X_j|\geq2$ for some $i\neq j$.

\begin{claim} \label{claim:X_h-two-links-cycle-path-power}
Suppose that $|X_i\cap X_j|\geq2$ for some $i\neq j$. Then we have $C^k_{\ell}\subseteq G$ for every $\frac{k+1}{2}|A_1|\leq\ell\leq\pc_k(n,\delta)$ and if further $\delta\in\left[\left(\frac{k-1}{k}+\nu\right)n,\frac{kn-1}{k+1}\right)$ we also have $P^k_{\pp_k(n,\delta)}\subseteq G$.
\end{claim}

\begin{claimproof}
Let $i\neq j$ such that $|X_i\cap X_j|\geq2$. Let $u_1$ and $u_2$ be distinct vertices of $X_i\cap X_j$. Let $v_1$ and $v_2$ be distinct neighbours in $A_i$ of $u_1$ and $u_2$ respectively, and similarly $w_1$ and $w_2$ in $A_j$. Applying Lemma~\ref{lem:extr-ext-path-cycle} in $A_i$, we can find a path from $v_1$ to $v_2$ of length $\alpha$ for any $4\leq \alpha\leq|A_i|-1$. We can find a similar path in $A_j$ from $w_1$ to $w_2$. Concatenating these paths with $u_1$ and $u_2$, we can find a cycle $S_{2h'}$ of length $2h'$ in $X_i\cup X_j$ for any $12\leq2h'\leq\min\{|A_i|+|A_j|+2,\frac{2n}{k+2}\}$. We shall construct the desired copy of $C^k_{(k+1)h'}$ from this cycle by repeated application of Lemma~\ref{lem:power-step-up}. We have steps $j=1,\dots,k-1$. In step $j$ we start with a copy of $C^j_{(j+1)h'}$
\[T_j=q_{1,j-1}\dots q_{1,1}t_1t_2\dots q_{h',j-1}\dots q_{h',1}t_{2h'-1}t_{2h'}\]
in $G$, with $t_p\in A_i\cup A_j\cup\{u_1,u_2\}$ for $p\in[2h']$, $q_{f,g}\in W'_g$ for each $f\in[h'],g\in[j-1]$, such that each vertex is adjacent to the immediately preceding $j$ vertices in cyclic order and no $2(j+1)$-tuple of consecutive vertices on $T_j$ uses both $u_1$ and $u_2$. 

Any $2(j+1)$-tuple of consecutive vertices on $T_j$ comprises four vertices from $A_i\cup A_j\cup\{u_1,u_2\}$ and two vertices from $W'_h$ for each $h\in[j-1]$. Each vertex in $A_i\cup A_j$ has at least $|W'_j|-\xi^2 n$ neighbours in $W'_j$, $u_1$ and $u_2$ each have at least $\xi n-2$ neighbours in $W'_j$, and for each $i\in[j-1]$ a vertex in $W'_i$ has at least $|W'_j|-(n-\delta)+\left|I^*_{\{i\}}\right|-2$ neighbours in $W'_j$. Hence, the four $2(j+1)$-tuples which use either $u_1$ or $u_2$ each have at least 
\begin{align*}
&\xi n-2-3\xi^2 n-2(j-1)(n-\delta)+2\left|I^*_{[j-1]}\right| \\
&\geq\xi n-2-3\xi^2 n-10k\eta n-2(k-j+1)(d+\eps)n>100\ell
\end{align*}
common neighbours in $W'_j$, with the first inequality following from~\eqref{eqn:lem-extr-int-ind-set-J-union-bounds} and the second inequality following from~\eqref{eqn:lem-extr-n-bound}, \eqref{eqn:lem-extr-xi-bound} and from
\begin{equation} \label{eqn:lem-extr-nu-l-bound}
\ell\leq\nu^{-1}.
\end{equation}
Every other $2(j+1)$-tuple of consecutive vertices on $T_j$ has at least
\begin{align*}
&|W'_j|-4\xi^2 n-2(j-1)(n-\delta)+2\left|I^*_{[j-1]}\right| \\
&\geq|W'_j|-4\xi^2 n-10k\eta n-2(k-j+1)(d+\eps)n
\end{align*}
common neighbours in $W'_j$. By the definition of $\xi$, \eqref{eqn:lem-extr-eta-d-bound}, \eqref{eqn:lem-extr-epsilon-bound} and \eqref{eqn:lem-extr-int-ind-set-J-union-bounds}, we have
\begin{equation*}
|W'_j|-4\xi^2 n-10k\eta n-2(k-j+1)(d+\eps)n\geq\frac{n}{k+2}.
\end{equation*}
This means that we can apply Lemma~\ref{lem:power-step-up}, with $G$, $W'_j$, and an ordering $\sigma$ of the relevant $2(j+1)$-tuples which has all the $2(j+1)$-tuples containing $u_1$ or $u_2$ coming first, to obtain a copy of $C^{j+1}_{(j+2)h'}$
\[T_{j+1}=q_{1,j}\dots q_{1,1}t_1t_2\dots q_{h',j}\dots q_{h',1}t_{2h'-1}t_{2h'}\]
in $G$, with $t_p\in A_i\cup A_j\cup\{u_1,u_2\}$ for $p\in[2h']$, $q_{f,g}\in W'_g$ for each $f\in[h'],g\in[j]$, such that each vertex is adjacent to the immediately preceding $j$ vertices in cyclic order and no $2(j+2)$-tuple of consecutive vertices on $T_{j+1}$ uses both $u_1$ and $u_2$. Terminating after step $k-1$ gives us a copy of $C^k_{(k+1)h'}$. Hence, we are able to find copies of $C^k_h$ for $h\in[k+1,\frac{k+1}{2}\min\{|A_i|+|A_j|+2,\frac{2n}{k+2}\}]$ such that $h$ is divisible by $k+1$. 

To obtain a copy of $C^k_h$ for $h$ not divisible by $k+1$, we perform a parity correction procedure. Fix $g\in[k]$. We seek a copy of $C^k_{(k+1)h'+g}$ with $h'\geq g+7$. Let $h'':=h'-g\geq7$. Let $u_1,u_2,v_1,v_2,w_1,w_2$ be the vertices previously picked. For the purpose of parity correction, pick (by Theorem~\ref{thm:erdos-stone}) vertices $a_{x,y}$ for $(x,y)\in[g]\times[3]$ in $A_i$ such that $a_{x,1}a_{x,2}a_{x,3}$ is a triangle for each $x\in[g]$ and $a_{x,3}a_{x+1,1}$ is an edge for $x\in[g-1]$. Let $A'=\{a_{x,y}\mid(x,y)\in[g]\times[3]\}$. Pick a common neighbour $v$ of $v_1$ and $a_{1,1}$ in $A_i$ which is not in $A'\cup\{v_2\}$. Applying Lemma~\ref{lem:power-step-up} suitably, we can find a path in $A_i\backslash(A\cup\{v,v_1\})$ from $a_{g,3}$ to $v_2$ of length $h$ for any $4\leq h\leq|A_i|-3g-2$ and a path in $A_j$ from $w_1$ to $w_2$ of length $h$ for any $4\leq h\leq|A_j|-1$. Concatenating these paths with $u_1,u_2,v_1,v,a_{1,1},a_{g,3}$, we can find a path of length $2h''+1$ in $A_i\cup A_j\cup\{u_1,u_2\}$ for any $15\leq2h''+1\leq\min\{|A_i|+|A_j|-3g+3,\frac{2n}{k+2}\}$. This allows us to construct a copy of $C^k_{(k+1)h'+g}$ whenever $h'\geq g+7$ by applying the method used previously. Therefore, we can obtain $C^k_{\ell'}\subseteq G$ for every $\ell'\in[k^2+9k+7,\frac{k+1}{2}\min\{|A_i|+|A_j|-3k,\frac{2n}{k+2}\}]$ such that $\chi(C^k_{\ell'})\leq k+2$. By~\eqref{eqn:lem-extr-dense-ext-size-min}, \eqref{eqn:spacious-bound-threshold-hug} and~\eqref{eqn:cycle-half-bound} we have $\pc_k(n,\delta)\leq\frac{k+1}{2}\min\{|A_i|+|A_j|-3k,\frac{2n}{k+2}\}$, so $G$ contains $C^k_{\ell'}$ for every $\frac{k+1}{2}|A_1|\leq\ell'\leq\pc_k(n,\delta)$. For the case $\delta\in\left[\left(\frac{k-1}{k}+\nu\right)n,\frac{kn-1}{k+1}\right)$ we note that $P^k_\ell\subseteq C^k_\ell$ and by~\eqref{eqn:lem-extr-dense-ext-size-min}, \eqref{eqn:spacious-bound-threshold-hug} and~\eqref{eqn:lem-extr-non-hamiltonian} we have $\pp_k(n,\delta)\leq\frac{k+1}{2}\min\{|A_i|+|A_j|-3k,\frac{2n}{k+2}\}$, so $G$ contains $P^k_{\pp_k(n,\delta)}$. This completes the proof.
\end{claimproof}

The following claim deals with the case when there exists $i\in[\ell]$ such that every vertex of $A_i$ is adjacent to some vertex outside $X_i\cup W$.

\begin{claim} \label{claim:X_i-external-links-cycle-path-power}
Suppose that there exists $i\in[\ell]$ such that every vertex of $A_i$ is adjacent to some vertex outside $X_i\cup W$. Then we have $C^k_{\ell}\subseteq G$ for every $\frac{k+1}{2}|A_1|\leq\ell\leq\pc_k(n,\delta)$ and if further $\delta\in\left[\left(\frac{k-1}{k}+\nu\right)n,\frac{kn-1}{k+1}\right)$ we also have $P^k_{\pp_k(n,\delta)}\subseteq G$.
\end{claim}

\begin{claimproof}
Since we have
\begin{equation*}
|A_i|
\overset{\eqref{eqn:lem-extr-dense-ext-size-min}}{\geq}\left|\bigcup B_i\right|-\xi^2n
\overset{\eqref{eqn:lem-extr-reduced-ext-lower-bound}}{\geq}\frac{38}{39}\nu(1-\eps)n-\xi^2n
\overset{\eqref{eqn:lem-extr-xi-bound}}{\geq}25\xi n
\overset{\eqref{eqn:lem-extr-xi-bound},\eqref{eqn:lem-extr-nu-l-bound}}{>}50\ell\xi^2n
\end{equation*}
there exists $j\neq i$ such that there are $50\xi^2n$ vertices in $A_i$ all adjacent to vertices of $X_j\setminus X_i$. No vertex of $X_j\setminus X_i$ is adjacent to $30\xi^2n$ vertices of $A_i$ (by definition of $X_i$), so there are two disjoint edges $u_1v_1$ and $u_2v_2$ from $u_1,u_2\in A_i$ to $v_1,v_2\in X_j$. Then, choosing distinct neighbours $w_1$ of $v_1$ and $w_2$ of $v_2$ in $A_j$ and applying the same reasoning as in Claim~\ref{claim:X_h-two-links-cycle-path-power} completes the proof.
\end{claimproof}

Now we deal with the remainder case. The following claim deals with finding the $k$th power of a path of the desired length in this case.

\begin{claim} \label{claim:remainder-path-power}
Suppose that $\delta\in\left[\left(\frac{k-1}{k}+\nu\right)n,\frac{kn-1}{k+1}\right)$, for each $i\neq j$ we have $|X_i\cap X_j|\leq1$ and for each $i$ there is a vertex of $A_i$ adjacent only to vertices in $X_i\cup W$. Then we have $P^k_{\pp_k(n,\delta)}\subseteq G$.
\end{claim}

\begin{claimproof}
In this case we have $|X_i|\geq\delta-|W|+1$ for each $i\in[\ell]$. We first focus on finding the $k$th power of a path on $\pp_k(n,\delta)$ vertices when $\delta\in\left[\left(\frac{k-1}{k}+\nu\right)n,\frac{kn-1}{k+1}\right)$. Note that if $|X_i\cap X_j|=1$ for some $i\neq j$, then we obtain the $k$th power of a path of the desired length as in Claim~\ref{claim:X_h-two-links-cycle-path-power}. We required two vertices in $|X_i\cap X_j|$ previously for a cycle to cross from $X_i$ to $X_j$ and back to $X_i$, whereas here we only need one vertex for a path to cross from $X_i$ to $X_j$.

Hence, assume that the sets $X_i$ are all disjoint. This implies that $\ell\leq\frac{n-|W|}{\delta-|W|+1}$. Note that $|W|\leq(k-1)(n-\delta)$ by~\eqref{eqn:lem-extr-int-ind-set-J-union-bounds}, so we have 
\[\ell\leq\frac{n-(k-1)(n-\delta)}{\delta-(k-1)(n-\delta)+1}=\frac{(k-1)\delta-(k-2)n}{k\delta-(k-1)n+1}.\]
Now if $\ell\geq r_p(n,\delta)+1$, we would have $r_p(n,\delta)+1\leq\ell\leq\frac{(k-1)\delta-(k-2)n}{k\delta-(k-1)n+1}$, and so $r_p(n,\delta)\leq\frac{n-\delta-1}{k\delta-(k-1)n+1}$, but by~\eqref{eqn:r-delta-ineq1} we have $r_p(n,\delta)\geq\frac{n-\delta}{k\delta-(k-1)n+1}$, so we have $\ell\leq r_p(n,\delta)$. Therefore, the largest of the sets $X_i$, say $X_1$, has at least
\begin{equation} \label{eqn:lem-extr-largest-expanded-ext-path}
|X_1|\geq\frac{n-|W|}{\ell}
\overset{\eqref{eqn:lem-extr-int-ind-set-J-union-bounds}}{\geq}\frac{(k-1)\delta-(k-2)n}{\ell}\geq\frac{(k-1)\delta-(k-2)n}{r_p(n,\delta)}
\end{equation}
vertices.

We wish to apply Lemma~\ref{lem:extr-ext-path-cycle} with $H=G[X_1]$ and `bad' vertices $B=X_1\setminus A_1$. Note that by~\eqref{eqn:lem-extr-expanded-ext-bound} $B$ contains at most $2\xi^2n$ vertices, so we have
\[|B|
\overset{\eqref{eqn:lem-extr-expanded-ext-bound}}{\leq}2\xi^2n
\overset{\eqref{eqn:lem-extr-xi-bound}}{\leq}\frac{\nu[(k-1)\delta-(k-2)n]}{100}
\overset{\eqref{eqn:lem-extr-nu-l-bound}}{\leq}\frac{(k-1)\delta-(k-2)n}{100\ell}
\overset{\eqref{eqn:lem-extr-largest-expanded-ext-path}} {\leq}\frac{|H|}{100}.\]
Moreover, we have $\delta(H)\geq\delta(G[X_1])\geq30\xi^2n$
by definition of $X_1$, so every vertex of $B$ has at least $30\xi^2n\geq9\cdot2\xi^2n\geq9|B|$ neighbours in $H$. For $v\in X_1\setminus B=A_1$, we have
\begin{align*}
\deg(v;X_1)&\overset{\eqref{eqn:lem-extr-dense-ext-min-deg}}{\geq}\frac{|A_1|}{2}+48\xi^2n
\overset{\eqref{eqn:lem-extr-expanded-ext-bound}}{>}\frac{|X_1|}{2}+47\xi^2n \\
&=\frac{|H|}{2}+47\xi^2n
\overset{\eqref{eqn:lem-extr-n-bound}}{\geq}\frac{|H|}{2}+9|B|+10.
\end{align*}
Hence, we may apply Lemma~\ref{lem:extr-ext-path-cycle} to obtain a path $P$ in $X_1$ with $\alpha:=\min\left\{|X_1|,\frac{2n}{k+2}\right\}$ vertices, on which no four consecutive vertices contain more than one vertex of $B$. Define $h':=\left\lfloor\frac{\alpha}{2}\right\rfloor$ and $\beta:=\alpha-2h'\in\{0,1\}$. We shall construct the desired copy of $P^k_{\pp_k(n,\delta)}$ from $P$ by repeated application of Lemma~\ref{lem:power-step-up}. We have steps $j=1,\dots,k-1$. In step $j$ we start with a copy of $P^j_{(j+1)h'+j-1+\beta}$
\begin{align*}
T_j&=q_{1,j-1}\dots q_{1,1}t_1t_2\dots q_{h',j-1}\dots q_{h',1}t_{2h'-1}t_{2h'} \\
&\qquad q_{h'+1,j-1}\dots q_{h'+1,1}t_{2h'+1}\dots t_{2h'+\beta}
\end{align*}
in $G$, with $t_p\in X_1$ for $p\in[\alpha]$, $q_{f,g}\in W'_g$ for each $f\in[h'+1],g\in[j-1]$, such that each vertex is adjacent to the preceding $j$ vertices and no $2(j+1)$-tuple of consecutive vertices on $T_j$ contains more than one vertex of $B$.

There are at most $2|B|\leq4\xi^2n$ $2(j+1)$-tuples containing vertices of $B$. Any $2(j+1)$-tuple of consecutive vertices on $T_j$ comprises four vertices from $X_1$ and two vertices from $W'_i$ for each $i\in[j-1]$. Each vertex in $A_1$ has at least $|W'_j|-\xi^2 n$ neighbours in $W'_j$, each vertex in $B$ has at least $\xi n-2$ neighbours in $W'_j$, and for each $i\in[j-1]$ a vertex in $W'_i$ has at least $|W'_j|-(n-\delta)+\left|I^*_{\{i\}}\right|-2$ neighbours in $W'_j$. Hence, the $2(j+1)$-tuples which contain a vertex of $B$ each have at least
\begin{align*}
&\xi n-2-3\xi^2 n-2(j-1)(n-\delta)+2\left|I^*_{[j-1]}\right| \\
&\geq\xi n-2-3\xi^2 n-10k\eta n-2(k-j+1)(d+\eps)n>100\ell
\end{align*}
common neighbours in $W'_j$, with the first inequality following from~\eqref{eqn:lem-extr-int-ind-set-J-union-bounds} and the second inequality following from~\eqref{eqn:lem-extr-n-bound}, \eqref{eqn:lem-extr-xi-bound} and~\eqref{eqn:lem-extr-nu-l-bound}. Every other $2(j+1)$-tuple of consecutive vertices on $T_j$ has at least
\begin{align*}
&|W'_j|-4\xi^2 n-2(j-1)(n-\delta)+2\left|I^*_{[j-1]}\right| \\
&\geq|W'_j|-4\xi^2 n-10k\eta n-2(k-j+1)(d+\eps)n
\end{align*}
common neighbours in $W'_j$. By the definition of $\xi$, \eqref{eqn:lem-extr-eta-d-bound}, \eqref{eqn:lem-extr-epsilon-bound} and \eqref{eqn:lem-extr-int-ind-set-J-union-bounds}, we have
\begin{equation*}
|W'_j|-4\xi^2 n-10k\eta n-2(k-j+1)(d+\eps)n\geq\frac{n}{k+2}.
\end{equation*}
This means that we can apply Lemma~\ref{lem:power-step-up}, with an ordering $\sigma$ of the relevant $2(j+1)$-tuples which has all the $2(j+1)$-tuples containing vertices of $B$ coming first, to obtain a copy of $P^{j+1}_{(j+2)h'+j+\beta}$
\begin{align*}
T_{j+1}&=q_{1,j}\dots q_{1,1}t_1t_2\dots q_{h',j}\dots q_{h',1}t_{2h'-1}t_{2h'} \\
&\qquad q_{h'+1,j}\dots q_{h'+1,1}t_{2h'+1}\dots t_{2h'+\beta}
\end{align*}
in $G$, with $t_p\in X_1$ for $p\in[\alpha]$, $q_{f,g}\in W'_g$ for each $f\in[h'+1],g\in[j]$, such that each vertex is adjacent to the preceding $j$ vertices and no $2(j+2)$-tuple of consecutive vertices on $T_j$ contains more than one vertex of $B$. Terminating after step $k-1$ gives the $k$th power of a path on at least $(k+1)h'+k-1+\beta$ vertices. We consider two cases. First consider when $\alpha=\frac{2n}{k+2}$. In this case, we have the $k$th power of a path on at least 
\[(k+1)\left(\frac{n}{k+2}-\frac{k+1}{k+2}\right)+k-1\geq\frac{(k+1)n}{k+2}-2\geq\pp_k(n,\delta)\]
vertices, with the inequality following from~\eqref{eqn:lem-extr-non-hamiltonian}. Otherwise, we have $\alpha=|X_1|$. Define $h'':=\left\lfloor\frac{|X_1|}{2}\right\rfloor$ and $\beta':=|X_1|-2h''\in\{0,1\}$. In this case, we have the $k$th power of a path on at least
\[(k+1)h''+k-1+\beta'=(k-1)(h''+1)+|X_1|\geq\pp_k(n,\delta)\]
vertices, with the inequality following from~\eqref{eqn:lem-extr-largest-expanded-ext-path} and the definition of $\pp_k(n,\delta)$.
\end{claimproof}

Finally, the following claim deals with finding $k$th powers of cycles of the desired lengths in the remainder case.

\begin{claim} \label{claim:remainder-cycle-power}
Suppose that for each $i\neq j$ we have $|X_i\cap X_j|\leq1$ and for each $i$ there is a vertex of $A_i$ adjacent only to vertices in $X_i\cup W$. Then we have $C^k_{\ell}\subseteq G$ for every $\frac{k+1}{2}|A_1|\leq\ell\leq\pc_k(n,\delta)$.
\end{claim}

\begin{claimproof}
First consider when there is a cycle of sets (relabelling the indices if necessary) $X_1,\dots,X_s$ for some $3\leq s\leq\ell$ such that $X_i\cap X_{i+1}=\{u_i\}$ for each $i$ and the $u_i$ are all distinct. In this case for each $i$ we may choose neighbours $v_i\in A_i$ and $w_i\in A_{i+1}$ of $u_i$, and we may insist that these $3s$ vertices are distinct. Similarly as before, we may apply Lemma~\ref{lem:extr-ext-path-cycle} to each $G[A_i]$ in turn and concatenate the resulting paths, in order to find a cycle $T_{2h'}$ for every $6s\leq2h'\leq\min\{|A_i|+|A_j|,\frac{2n}{k+2}\}$ on which there are no quadruples using more than one vertex outside $\bigcup_{i\in[s]}A_i$. Arguing in a manner similar to Claim~\ref{claim:X_h-two-links-cycle-path-power}, we may repeatedly apply Lemma~\ref{lem:power-step-up} to obtain a copy of $C^k_{(k+1)h'}$. Hence, we are able to find copies of $C^k_h$ for $h\in[3s(k+1),\frac{k+1}{2}\min\{|A_i|+|A_j|,\frac{2n}{k+2}\}]$ such that $h$ is divisible by $k+1$. To obtain a copy of $C^k_h$ for $h$ not divisible by $k+1$, we use a parity correction procedure analogous to that in Claim~\ref{claim:X_h-two-links-cycle-path-power}. Therefore, we can find copies of $C^k_h$ for $h\in[k^2+3(s+1)k+(3s+1),\frac{k+1}{2}\min\{|A_i|+|A_j|-3k,\frac{2n}{k+2}\}]$. Hence, we have $C^k_{\ell'}\subseteq G$ for every $\ell'\in[k+1,\frac{k+1}{2}\min\{|A_i|+|A_j|-3k,\frac{2n}{k+2}\}]$ such that $\chi(C^k_{\ell'})\leq k+2$.

Otherwise, no such cycle of sets exists. In this case, we have $\sum_{i=1}^\ell|X_i|\leq n-|W|+\ell-1$. Note that $|X_i|\geq\delta-|W|+1$ for each $i\in[\ell]$, so this implies that $\ell\leq\frac{n-|W|-1}{\delta-|W|}$. Note that $|W|\leq(k-1)(n-\delta)$ by~\eqref{eqn:lem-extr-int-ind-set-J-union-bounds}, so we have
\[\ell\leq\frac{n-(k-1)(n-\delta)-1}{\delta-(k-1)(n-\delta)}=\frac{(k-1)\delta-(k-2)n-1}{k\delta-(k-1)n}.\]
Now if $\ell\geq r_c(n,\delta)+1$, we would have $r_c(n,\delta)+1\leq\ell\leq\frac{(k-1)\delta-(k-2)n}{k\delta-(k-1)n}$, and so $r_c(n,\delta)\leq\frac{n-\delta-1}{k\delta-(k-1)n}$, but we have $r_c(n,\delta)\geq\frac{n-\delta}{k\delta-(k-1)n}$, so we have $\ell\leq r_c(n,\delta)$. Therefore, the largest of the sets $X_i$, say $X_1$, has at least
\begin{equation*} \label{eqn:lem-extr-largest-expanded-ext-cycle}
|X_1|\geq\frac{n-|W|}{\ell}\geq\frac{(k-1)\delta-(k-2)n}{\ell}\geq\frac{(k-1)\delta-(k-2)n}{r_c(n,\delta)}
\end{equation*}
vertices.

As before, by Lemma~\ref{lem:extr-ext-path-cycle} for each $2h'\in[4,\min\{|X_1|,\frac{2n}{k+2}\}]$ we find in $X_1$ a copy of $C_{2h'}$, on which no four consecutive vertices contain more than one vertex of $B$, and by repeated application of Lemma~\ref{lem:power-step-up} we obtain the $k$th power of a cycle $C^k_{(k+1)h'}$ for each $(k+1)h'\in[2(k+1),\pc_k(n,\delta)]$. As before, we may apply a parity correction procedure for copies of $C^k_h$ where $h$ is not divisible by $k+1$. Therefore, we have copies of $C^k_h$ for $h\in[k+1,\pc_k(n,\delta)\}]$ such that $\chi(C^k_{\ell'})\leq k+2$.
\end{claimproof}

Claims~\ref{claim:X_h-two-links-cycle-path-power}, \ref{claim:X_i-external-links-cycle-path-power}, \ref{claim:remainder-path-power} and~\ref{claim:remainder-cycle-power} collectively yield the desired outcome, completing the proof.
\end{proof}

\section{Concluding remarks}

\paragraph{\bf Extremal graphs and minimum degree}

Our proofs provide a template for checking that $G_p(k,n,\delta)$ and $G_c(k,n,\delta)$ are the only extremal graphs up to some trivial modifications. We believe that the graph $G_p(k,n,\delta)$ remains extremal for $k$th powers of paths for all $\delta>\frac{(k-1)n}{k}$. However, the same is generally not true for $G_c(k,n,\delta)$ and $k$th powers of cycles: Allen, B\"ottcher and Hladk\'y~\cite{AllenBoettcherHladky} sketched a construction, for infinitely many values of $n$, of graphs $G$ on $n$ vertices with $\delta(G)\geq\frac{n}{2}+\frac{\sqrt{n}}{5}$ which do not contain a copy of $C^2_6$. Their construction can be generalised to one for general powers of cycles.

\paragraph{\bf Long $k$th powers of cycles}

Theorem~\ref{thm:cycle-path-power}\ref{item:cycle-power-parity cases} states that if $G$ does not contain any of various $k$th powers of cycles of lengths not divisible by $k+1$, then $G$ must contain $k$th powers of cycles of every length divisible by $k+1$ up to $(k+1)(k\delta-(k-1)n)-\nu n$. We believe that the error term of $\nu n$ can be removed, but it would involve significantly more technical work. This includes a new version of the stability lemma with more extremal cases and new extremal results corresponding to these additional extremal cases.


\bibliography{TuranPosa} 
\bibliographystyle{plain}

 
\appendix 

\section{Embedding Lemma} \label{section:embedding}

In this appendix we will provide our proof of Lemma~\ref{lem:embedding}. To do so, we shall apply a version of a graph Blow-up Lemma by Allen, B\"ottcher, H\`an, Kohayakawa and Person~\cite{AllenBoettcherHanKohayakawaPerson}. We remark that the Blow-up Lemma of Koml\'os, S\'ark\"ozy and Szemer\'edi~\cite{KomlosSarkozySzemeredi-blowup} is perfectly adequate for this proof; our choice of Blow-up Lemma is not driven by necessity, but rather a desire to reduce the technical complexity of our proof.

We will use the definition of $(\eps,d)$-regular as given in Section~\ref{subsection:regularity}; this involves both an upper bound and a lower bound on densities. We note that the corresponding graph Blow-up Lemma in~\cite{AllenBoettcherHanKohayakawaPerson} applies to a more general class of graphs; in particular, the regularity condition in~\cite{AllenBoettcherHanKohayakawaPerson} is weaker and involves only a lower bound.

We first introduce some terminology in order to formulate this version of the Blow-up Lemma. Let $\kappa\geq1$. Let $G$ and $H$ be two graphs, on the same number of vertices, with partitions $\mathcal{V}=\{V_i\}_{i\in[r]}$ and $\mathcal{X}=\{X_i\}_{i\in[r]}$ of their respective vertex sets. We say that $\mathcal{V}$ and $\mathcal{X}$ are \emph{size-compatible} if $|V_i|=|X_i|$ for all $i\in[r]$. Moreover, we say that $(G,\mathcal{V})$ is \emph{$\kappa$-balanced} if there exists $m\in\NN$ such that $m\leq|V_i|\leq\kappa m$ for all $i\in[r]$.

Let $R$ be a graph on $r$ vertices.
\begin{enumerate}[noitemsep,label=(\roman*)]
\item $(H,\mathcal{X})$ is an \emph{$R$-partition} if each part of $\mathcal{X}$ is nonempty, and whenever there are edges of $H$ between $X_i$ and $X_j$, the pair $ij$ is an edge of $R$,
\item $(G,\mathcal{V})$ is an \emph{$(\eps,d)$-regular $R$-partition} if for each edge $ij\in E(R)$ the pair $(V_i,V_j)$ is $(\eps,d)$-regular.
\end{enumerate}
In this case we say that $R$ is an \emph{$(\eps,d)$-full-reduced graph} of the partition~$\mathcal{V}$. We remark that the notion of an $(\eps,d)$-regular $R$-partition is distinct from that of an $(\eps,d)$-reduced partition defined in Section~\ref{section:lem-thmproof}: in an $(\eps,d)$-regular $R$-partition, the partition~$\mathcal{V}$ does not have an exceptional set and we allow each vertex of $G$ to be incident to possibly many edges which are not in $(\eps,d)$-regular pairs. The notion of an $(\eps,d)$-full-reduced graph is correspondingly distinct from that of an $(\eps,d)$-reduced graph.

Suppose $R$ is a graph on $r$ vertices, $(H,\mathcal{X})$ is an $R$-partition and $(G,\mathcal{V})$ is a size-compatible $(\eps,d)$-regular $R$-partition. Let $\alpha>0$. A family $\bar{\mathcal{X}}=\{\bar{X_i}\}_{i\in[r]}$ of subsets $\bar{X_i}\subseteq X_i$ is an \emph{$\alpha$-buffer} for $H$ if
\begin{enumerate}[noitemsep,label=(\roman*)]
\item the elements of $\bar{X_i}$ are isolated vertices in $H$,
\item $|\bar{X_i}|\geq\alpha|X_i|$ for all $i\in[r]$.
\end{enumerate}
Note that this corresponds to the notion of an $(\alpha,R')$-buffer for $H$ in~\cite{AllenBoettcherHanKohayakawaPerson} with $R'$ as the empty spanning subgraph of $R$.

Let $R$ be a graph on $r$ vertices, $(H,\mathcal{X})$ be an $R$-partition and $(G,\mathcal{V})$ be a size-compatible $(\eps,d)$-regular $R$-partition, with $G\subseteq K_n$. Let $\mathcal{I}=\{I_x\}_{x\in V(H)}$ be a collection of subsets of $V(G)$, called \emph{image restrictions}, and $\mathcal{J}=\{J_x\}_{x\in V(H)}$ be a collection of subsets of $V(K_n)\setminus V(G)$, called \emph{restricting vertices}. $\mathcal{I}$ and $\mathcal{J}$ are a \emph{$(\rho,\zeta,\Delta,\Delta_J)$-restriction pair} if the following properties hold for each $i\in[r]$ and $x\in X_i$.
\begin{enumerate}[label=(\alph*), noitemsep]
\item The set $X^*_i\subseteq X_i$ of \emph{image restricted} vertices in $X_i$, that is, vertices such that $I_x\neq V_i$, has size $|X^*_i|\leq\rho|X_i|$.
\item If $x\in X^*_i$, then $I_x\subseteq V_i$ is of size at least $\zeta d^{|J_x|}|V_i|$.
\item If $x\in X^*_i$, then $|J_x|+|\Gamma_H(x)|\leq\Delta$, and if $x\notin X^*_i$, then $J_x=\nth$.
\item Each vertex of $K_n$ appears in at most $\Delta_J$ of the sets of $\mathcal{J}$.
\item If $x\in X^*_i$, then for each $xy\in E(H)$ with $y\in X_j$ the pair $(V_i,V_j)$ is $(\eps,d)$-regular in $G$.
\end{enumerate}

\begin{lemma}[Allen, B\"ottcher,  H\`an, Kohayakawa and Person~\cite{AllenBoettcherHanKohayakawaPerson}] \label{lem:blow-up-dense}
For all $\Delta\geq2,\Delta_J,\alpha,\zeta,d>0$, $\kappa>1$ there exists $\eps,\rho>0$ such that for all $r_1$ there exists $n_{BL}\in\NN$ such that for all $n\geq n_{BL}$ the following holds. Let $R$ be a graph on $r\leq r_1$ vertices. Let $H$ and $G$ be $n$-vertex graphs with $\kappa$-balanced size-compatible vertex partitions $\mathcal{X}=\{X_i\}_{i\in[r]}$ and $\mathcal{V}=\{V_i\}_{i\in[r]}$, respectively, which have parts of size at least $m\geq n/(\kappa r_1)$. Let $\bar{\mathcal{X}}=\{\bar{X_i}\}_{i\in[r]}$ be a family of subsets of $V(H)$, $\mathcal{I}=\{I_x\}_{x\in V(H)}$ be a family of image restrictions, and $\mathcal{J}=\{J_x\}_{x\in V(H)}$ be a family of restricting vertices. Suppose that
\begin{enumerate}[label=(\roman*)]
\item $\Delta(H)\leq\Delta$, $(H,\mathcal{X})$ is an $R$-partition, and $\bar{\mathcal{X}}$ is an $\alpha$-buffer for~$H$,
\item $(G,\mathcal{V})$ is an $(\eps,d)$-regular $R$-partition,
\item $\mathcal{I}$ and $\mathcal{J}$ form a $(\rho,\zeta,\Delta,\Delta_J)$-restriction pair.
\end{enumerate}
Then there is an embedding $\psi:V(H)\to V(G)$ such that $\psi(x)\in I_x$ for each $x\in V(H)$.
\end{lemma}

\begin{proof}[Proof of Lemma~\ref{lem:embedding}]
We proceed by checking the conditions for a suitable application of Lemma \ref{lem:blow-up-dense} to embed a relevant graph $H$ into $G$. We first prove~\ref{item:embed-divisible},\ref{item:embed-(k+2)-clique}. Fix $k\geq2$, $d>0$ and set $\Delta=2k,\Delta_J=k,\kappa=2, \alpha=\frac{d}{2},\zeta=1$. Now Lemma~\ref{lem:blow-up-dense} outputs $\eps_0,\rho_0>0$. We choose 
\[\eps_{EL}=\min\left\{\frac{\eps_0}{k+1},\frac{d^2}{8(k+1)}\right\}.\]
Given $0<\eps<\eps_{EL},r_{EL}\in\NN$, Lemma~\ref{lem:blow-up-dense} outputs $n_{BL}\in\NN$. We choose 
\[n_{EL}=\max\left\{n_{BL},\frac{6r_{EL}^{k+2}}{\eps},\frac{4r_{EL}}{\rho_0}\right\}.\]
Let $n\geq n_{EL}$, let $G$ be a graph on $n$ vertices and let $R$ be an $(\eps,d)$-reduced graph of $G$ on $r\leq r_{EL}$ vertices. Let $V_0,V_1,\dots,V_r$ be the vertex classes of the $(\eps,d)$-regular partition of $G$ which gives rise to $R$. Fix a connected $K_{k+1}$-factor $\mathcal{F}$ in $R$ which contains $c:=\frac{\ckf_{k+1}(R)}{k+1}$ copies of $K_{k+1}$. Let $T_1,\dots,T_c$ be the copies of $K_{k+1}$ in $\mathcal{F}$. Let $\mathcal{V}:=\{V_1,\dots,V_r\}$. Let $R'$ be the empty spanning subgraph of $R$. Let $G^*$ be the subgraph of $G$ induced on $\mathcal{V}$ and set $n^*:=|V(G^*)|$. Note that $(G^*,\mathcal{V})$ is an $(\eps,d)$-regular $R$-partition. Note that $|V_i|\geq(1-\eps)\frac{n}{r}\geq\frac{n}{2r_{EL}}$ for all $i\in[r]$, so $\mathcal{V}$ is 2-balanced.

Let $H$ be a copy of $C^k_\ell$ together with additional isolated vertices so that it has $n^*$ vertices. Let $v_1,\dots,v_{n^*}$ be its vertices, with $v_1,\dots,v_\ell$ being the vertices of the copy of $C^k_\ell$ in an arbitrary cyclic order. Let $C:=\{v_i: i\in[\ell]\}$. Suppose that we have a vertex partition $\mathcal{X}:=\{X_i\}_{i\in[r]}$ of $H$ and a family $\bar{\mathcal{X}}:=\{\bar{X}_i\}_{i\in[r]}$ of subsets of $V(H)$ such that $\mathcal{X}$ is size-compatible with $\mathcal{V}$, $(H,\mathcal{X})$ is an $R$-partition, and $\bar{\mathcal{X}}$ is an $\alpha$-buffer for $H$. Define $\mathcal{I}:=\{I_x\}_{x\in V(H)}$ and $\mathcal{J}:=\{J_x\}_{x\in V(H)}$ by $I_x=V_i$ for $x\in X_i$ and $J_x=\nth$ for $x\in V(H)$. Note that $\mathcal{I}$ and $\mathcal{J}$ form a $(\rho_0,\zeta,\Delta,\Delta_J)$-restriction pair. Then, by Lemma~\ref{lem:blow-up-dense} we will have an embedding $\phi:V(H)\to V(G^*)$, which will then complete our proof of~\ref{item:embed-divisible} and~\ref{item:embed-(k+2)-clique}. Therefore, for suitable values of $\ell$ it remains to find a vertex partition $\mathcal{X}$ of $H$ and a family $\bar{\mathcal{X}}$ of subsets of $V(H)$ such that $\mathcal{X}$ is size-compatible with $\mathcal{V}$, $(H,\mathcal{X})$ is an $R$-partition and $\bar{\mathcal{X}}$ is an $\alpha$-buffer for $H$.

We start with~\ref{item:embed-divisible} and we will consider $\ell\leq\frac{(1-d)(k+1)cn}{r}$ divisible by $k+1$. We first consider the case $\ell\leq\frac{(k+1)(1-d)n}{r}$ divisible by $k+1$, that is, when $c=1$. Let $Y_1,\dots,Y_{k+1}$ be the vertices of $T_1$. Define $\phi:V(H)\to V(R)$ as follows. For $i\leq\ell$, set $\phi(v_i)=Y_j$ with $j\equiv i\mod k+1$. For $i>\ell$, given $\phi(v_1),\dots,\phi(v_{i-1})$, set $\phi(v_i)=V_j$ with $j=\min\{h:|\{b<i:\phi(v_b)=V_h\}|<|V_h|\}$. Set $X_i:=\phi^{-1}(V_i)$, $\bar{X}_i:=\phi^{-1}(V_i)\backslash C$ for $i\in[r]$. Define $\mathcal{X}:=\{X_i\}_{i\in[r]}$ and $\bar{\mathcal{X}}:=\{\bar{X}_i\}_{i\in[r]}$. Since all edges in $H$ have pairs of vertices in $C$ at most $k$ apart in the cyclic order as endpoints and any $k+1$ consecutive vertices in the cyclic order are mapped to a copy of $K_{k+1}$ in $R$, it follows that $(H,\mathcal{X})$ is an $R$-partition. Furthermore, for each $i\in[r]$ at most $\frac{\ell}{k+1}\leq(1-d)\frac{n}{r}\leq|V_i|$ vertices in $C$ are mapped to $V_i$, so $\mathcal{X}$ is a vertex partition of $H$ which is size-compatible with~$\mathcal{V}$. Finally, $\bar{X}_i$ is a set of isolated vertices in $H$ by definition and
\[|\bar{X}_i|=|X_i|-|C\cap X_i|\geq\left(1-\frac{1-d}{1-\eps}\right)|X_i|\geq\alpha|X_i|\]
for each $i\in[r]$, so $\bar{\mathcal{X}}$ is an $\alpha$-buffer for $H$. This completes the proof in this case. We are done if $c=1$, so we can assume $c\geq2$ for the remainder of~\ref{item:embed-divisible}.

Next, we consider the case $\ell\in\left(\frac{(1-d)(k+1)n}{r},\frac{(1-d)(k+1)cn}{r}\right]$ divisible by $k+1$. For each $i\in[c-1]$, fix a $K_{k+1}$-walk $W_i$ whose first copy of $K_k$ is in $T_i$ and whose last is in $T_{i+1}$, which is of minimal length. We have $|W_i|\leq\binom{r}{k}$ for each $i\in[c-1]$. Let $W'$ be the $K_{k+1}$-walk obtained by concatenating $W_1,\dots,W_{c-1}$.

We shall now describe how to construct the sequence $Q(W,\overrightarrow{U_{11}\dots U_{1k}})$ for any $K_{k+1}$-walk $W=(E_1,E_2,\dots)$ in $R$ and any orientation $\overrightarrow{U_{11}\dots U_{1k}}$ of $E_1$, its first copy of $K_k$. We construct $Q(W,\overrightarrow{U_{11}\dots U_{1k}})$ iteratively as follows. Let $Q_1=(U_{11},\dots,U_{1k})$. Now for $2\leq i\leq|W|$ successively, we define $Q_i$ as follows. The last $k$ vertices $U_{(i-1)1},\dots,U_{(i-1)k}$ of $Q_{i-1}$ are an orientation of $E_{i-1}$. For some $j\in[k]$ we have $E_i=U_{(i-1)1}\dots U_{(i-1)(j-1)}U_{(i-1)(j+1)\dots U_{(i-1)k}}U_{ik}$. Append $(U_{ik},U_{(i-1)1},\dots U_{(i-1)(j-1)})$ to $Q_{i-1}$ to create $Q_i$. At each step the last $k$ vertices of $Q_i$ are an orientation of $E_i$ and every vertex of $Q_i$ is adjacent in $R$ to the $k$ vertices preceding it in $Q_i$. Finally we let $Q(W,\overrightarrow{U_{11}\dots U_{1k}}):=Q_{|W|}$.

It is easy to check by induction that for any $K_{k+1}$-walk $W$ whose first edge is $U_{11}\dots U_{1k}$, we have
\begin{equation} \label{eqn:seq-forward-backward-mod}
|Q(W,\overrightarrow{U_{11}\dots U_{1k}})|+|Q(W,\overrightarrow{U_{1k}\dots U_{11}})|\equiv-2\mod k+1.
\end{equation}
Now consider the concatenation $W'$ of the walks $W_i$. Let $U_{11}\dots U_{1k}$ be the first copy of $K_k$ of $W_1$. If we construct $Q(W',\overrightarrow{U_{11}\dots U_{1k}})$ then the first copy of $K_k$ $U_{i1}\dots U_{ik}$ and the last copy of $K_k$ $U'_{i1}\dots U'_{ik}$ of each $W_i$ obtain orientations, say $\overrightarrow{U_{i1}\dots U_{ik}}$ and $\overrightarrow{U'_{i1}\dots U'_{ik}}$. Clearly, there are sequences $\bar{Q}_i$ of vertices in $T_i$ for $1<i<c$, such that $Q(W',\overrightarrow{U_{11}\dots U_{1k}})$ is the concatenation of
\[Q(W_1,\overrightarrow{U_{11}\dots U_{1k}}),\bar{Q}_2,Q(W_2,\overrightarrow{U_{21}\dots U_{2k}}),\dots,\bar{Q}_{c-1},Q(W_{c-1},\overrightarrow{U_{(c-1)1}\dots U_{(c-1)k}}).\]
Let $\bar{Q}_1:=T_1-U_{11}\dots U_{1k}$ and $\bar{Q}_c:=T_c-U'_{c1}\dots U'_{ck}$. Define $f_i\equiv|\bar{Q}_i|\mod k+1$ for $i\in[c]$. Together with~\eqref{eqn:seq-forward-backward-mod}, we obtain
\begin{equation*}
\begin{split}
& |Q(W',\overrightarrow{U_{1k}\dots U_{11}})| + |Q(W_1,\overrightarrow{U_{11}\dots U_{1k}})| \\
& \quad+ \sum_{1<i<c}(|Q(W_i,\overrightarrow{U_{i1}\dots U_{ik}})|+f_i)\equiv-2\mod k+1
\end{split}
\end{equation*}
and hence
\begin{equation} \label{eqn:seq-forward-backward-mod-divisible}
|Q(W',\overrightarrow{U_{1k}\dots U_{11}})|+\sum_{i\in[c-1]}(|Q(W_i,\overrightarrow{U_{i1}\dots U_{ik}})|+f_i)+f_c\equiv0\mod k+1.
\end{equation}
Let $Q'$ denote $Q(W',\overrightarrow{U_{1k}\dots U_{11}})$ and let $Q^*_i$ denote $Q(W_i,\overrightarrow{U_{i1}\dots U_{ik}})$ for each $i\in[c-1]$. Define $q':=|Q'|$ and $q_i:=|Q^*_i|$ for each $i\in[c-1]$. For a sequence $Q$ of vertices of $R$, let $(Q)_h$ denote the $h$th term of $Q$.

Let $U_{11}\dots U_{1k}$ be the first copy of $K_k$ in $W_1$. Orient it as $\overrightarrow{U_{11}\dots U_{1k}}$. Construct $Q^*_i$ for $i\in[c-1]$ and $Q'$ as described before, and define $q_i,f_i$ for $i\in[c-1]$ and $q',f_c$ as before. Let $T_i=Y_{i1}\dots Y_{i(k+1)}$ for $i\in[c]$ be such that $\overrightarrow{Y_{i2}\dots Y_{i(k+1)}}$ is the oriented last copy of $K_k$ of $W_{i-1}$ in $Q^*_{i-1}$ for $2\leq i\leq c$ and $\overrightarrow{Y_{1(k+1)}\dots Y_{12}}$ is the oriented first copy of $K_k$ of $W'$ in $Q'$. Define the following. Let $\alpha:=\sum_{i=1}^{c-1}(q_i+f_i)+f_c+q'$.
\begin{gather*}
p_0:=\max\left\{p\in\mathbb{Z}\mid\ell\geq p\dot(1-d)(k+1)\frac{n}{r}+\alpha\right\}; \\
t_i=
\begin{cases}
(1-d)\frac{n}{r}&\textrm{ if }i\in[p_0]\\
\frac{\ell-\alpha}{k+1}-p_0(1-d)\frac{n}{r}&\textrm{ if }i=p_0+1\\
0&\textrm{ if }i>p_0+1;
\end{cases} \\
L_0=0, L_j=\sum_{i=1}^j[t_i(k+1)+q_i+f_i]\textrm{ for }j\in[c-1]; \\
M_j=L_{j-1}+t_j(k+1)+f_j\textrm{ for }j\in[c].
\end{gather*}
Define $\phi:V(H)\to V(R)$ as follows. For $i\leq\ell$, set
\[
\phi(v_i)=
\begin{cases}
Y_{jh}&\textrm{ if }L_{j-1}<i\leq M_j,\textrm{ with }h\equiv i-L_{j-1}\mod k+1 \\
(Q^*_j)_{i-M_j}&\textrm{ if }M_j<i\leq L_j \\
(Q')_{M_c+q'+1-i}&\textrm{ if }M_c<i\leq\ell.
\end{cases}
\]
For $i>\ell$, given $\phi(v_1),\dots,\phi(v_{i-1})$, set $\phi(v_i)=V_j$ with $j=\min\{h:|\{b<i:\phi(v_b)=V_h\}|<|V_h|\}$.

Set $X_i:=\phi^{-1}(V_i)$, $\bar{X}_i:=\phi^{-1}(V_i)\backslash C$ for $i\in[r]$. Define $\mathcal{X}:=\{X_i\}_{i\in[r]}$ and $\bar{\mathcal{X}}:=\{\bar{X}_i\}_{i\in[r]}$. Since all edges in $H$ have pairs of vertices in $C$ at most $k$ apart in the cyclic order as endpoints and any $k+1$ consecutive vertices in the cyclic order are mapped to a copy of $K_{k+1}$ in $R$, it follows that $(H,\mathcal{X})$ is an $R$-partition. Furthermore, for each $i\in[r]$ at most $(1-d)\frac{n}{r}+\frac{2r\binom{r}{k}}{k+1}\leq(1-d+\eps)\frac{n}{r}\leq(1-\eps)\frac{n}{r}\leq|V_i|$ vertices in $C$ are mapped to $V_i$, so $\mathcal{X}$ is a vertex partition of $H$ which is size-compatible with~$\mathcal{V}$. Finally, $\bar{X}_i$ is a set of isolated vertices in $H$ by definition and
\[|\bar{X}_i|=|X_i|-|C\cap X_i|\geq\left(1-\frac{1-d+\eps}{1-\eps}\right)|X_i|\geq\alpha|X_i|\]
for each $i\in[r]$, so $\bar{\mathcal{X}}$ is an $\alpha$-buffer for $H$. This completes the proof in this case and for~\ref{item:embed-divisible}.

We continue with~\ref{item:embed-(k+2)-clique} and we will consider $\ell\leq\frac{(1-d)(k+1)cn}{r}$ satisfying $\chi(C_{\ell}^k)\leq k+2$. Pick $y\in[k]\cup\{0\}$ such that $\ell\equiv y\mod k+1$. In particular, we have $\ell\geq y(k+2)$. Let $S$ be a copy of $K_{k+2}$ in the same $K_{k+1}$-component as the copies of $K_{k+1}$ in $\mathcal{F}$ and let $Z_1,\dots,Z_{k+2}$ be the vertices of $S$. We first consider the case $\ell\leq\frac{(k+1)(1-d)n}{r}$ satisfying $\chi(C_{\ell}^k)\leq k+2$. Define $\phi:V(H)\to V(R)$ as follows. For $i\leq\ell$, set 
\[
\phi(v_i)=
\begin{cases}
Z_j&\textrm{ if }i\leq\ell-y(k+2),\textrm{ with }j\equiv i\mod k+1, \\
Z_j&\textrm{ if }\ell-y(k+2)<i\leq\ell,\textrm{ with }j\equiv i\mod k+2.
\end{cases}
\]
For $i>\ell$, given $\phi(v_1),\dots,\phi(v_{i-1})$, set $\phi(v_i)=V_j$ with $j=\min\{h:|\{b<i:\phi(v_b)=V_h\}|<|V_h|\}$. Set $X_i:=\phi^{-1}(V_i)$, $\bar{X}_i:=\phi^{-1}(V_i)\backslash C$ for $i\in[r]$ and define $\mathcal{X}:=\{X_i\}_{i\in[r]}$ and $\bar{\mathcal{X}}:=\{\bar{X}_i\}_{i\in[r]}$. Since all edges in $H$ have pairs of vertices in $C$ at most $k$ apart in the cyclic order as endpoints and any $k+1$ consecutive vertices in the cyclic order are mapped to a copy of $K_{k+1}$ in $R$, it follows that $(H,\mathcal{X})$ is an $R$-partition. Furthermore, for each $i\in[r]$ at most $\frac{\ell}{k+1}\leq(1-d)\frac{n}{r}\leq|V_i|$ vertices in $C$ are mapped to $V_i$, so $\mathcal{X}$ is a vertex partition of $H$ which is size-compatible with~$\mathcal{V}$. Finally, $\bar{X}_i$ is a set of isolated vertices in $H$ by definition and
\[|\bar{X}_i|=|X_i|-|C\cap X_i|\geq\left(1-\frac{1-d}{1-\eps}\right)|X_i|\geq\alpha|X_i|\]
for each $i\in[r]$, so $\bar{\mathcal{X}}$ is an $\alpha$-buffer for $H$. This completes the proof in this case. We are done if $c=1$, so we can assume $c\geq2$ for the remainder of~\ref{item:embed-(k+2)-clique}.

Next, we consider $\ell\in\left(\frac{(1-d)(k+1)n}{r},\frac{(1-d)(k+1)cn}{r}\right]$ satisfying $\chi(C_{\ell}^k)\leq k+2$. For each $i\in[c-1]$, fix a $K_{k+1}$-walk $W_i$ whose first copy of $K_k$ is in $T_i$ and whose last is in $T_{i+1}$, which is of minimal length. We have $|W_i|\leq\binom{r}{k}$ for each $i\in[c-1]$. Let $W'$ be the $K_{k+1}$-walk obtained by concatenating $W_1,\dots,W_{c-1}$. Fix a $K_{k+1}$-walk $W''$ whose first copy of $K_k$ is that of $W_1$, whose last is that of $W_{c-1}$, which includes a copy of $K_k$ from $S$ and is one of minimal length satisfying these conditions. We have $|W''|\leq2\binom{r}{k}$.

We construct the sequence $Q(W,\overrightarrow{U_{11}\dots U_{1k}})$ for any $K_{k+1}$-walk $W=(E_1,E_2,\dots)$ in $R$ and any orientation $\overrightarrow{U_{11}\dots U_{1k}}$ of $E_1$, its first copy of $K_k$, identically to that in~\ref{item:embed-divisible}. Let $U_{11}\dots U_{1k}$ be the first copy of $K_k$ in $W_1$ and orient it as $\overrightarrow{U_{11}\dots U_{1k}}$. Construct $Q(W',\overrightarrow{U_{11}\dots U_{1k}})$. Then, the first copy of $K_k$ $U_{i1}\dots U_{ik}$ and the last copy of $K_k$ $U'_{i1}\dots U'_{ik}$ of each $W_i$ obtain orientations, say $\overrightarrow{U_{i1}\dots U_{ik}}$ and $\overrightarrow{U'_{i1}\dots U'_{ik}}$. Construct $Q(W_i,\overrightarrow{U_{i1}\dots U_{ik}})$ for $i\in[c]$. Clearly, there are sequences $\bar{Q}_i$ of vertices in $T_i$ for $1<i<c$, such that $Q(W',\overrightarrow{U_{11}\dots U_{1k}})$ is the concatenation of
\[Q(W_1,\overrightarrow{U_{11}\dots U_{1k}}),\bar{Q}_2,Q(W_2,\overrightarrow{U_{21}\dots U_{2k}}),\dots,\bar{Q}_{c-1},Q(W_{c-1},\overrightarrow{U_{(c-1)1}\dots U_{(c-1)k}}).\]
Let $\bar{Q}_1:=T_1-U_{11}\dots U_{1k}$ and $\bar{Q}_c:=T_c-U_{c1}\dots U_{ck}$. Define $f_i\equiv|\bar{Q}_i|\mod k+1$ for $i\in[c]$. Let $Q^*_i$ denote $Q(W_i,\overrightarrow{U_{i1}\dots U_{ik}})$ for $i\in[c-1]$ and let $Q''$ denote $Q(W'',\overrightarrow{U_{1k}\dots U_{11}})$. Define $q_i:=|Q^*_i|$ for $i\in[c-1]$ and $q'':=|Q''|$. For a sequence $Q$ of vertices of $R$, let $(Q)_h$ denote the $h$th term of $Q$. Let $T_i=Y_{i1}\dots Y_{i(k+1)}$ for all $i\in[c]$ be such that $\overrightarrow{Y_{i2}\dots Y_{i(k+1)}}$ is the oriented last copy of $K_k$ of $W_{i-1}$ in $Q_{i-1}$ for $2\leq i\leq c$ and $\overrightarrow{Y_{1(k+1)}\dots Y_{12}}$ is the oriented first copy of $K_k$ of $W''$ in $Q''$.

Let $\alpha:=\sum_{i=1}^{c-1}(q_i+f_i)+f_c+q''$. Pick $x\in[k]\cup\{0\}$ such that $\ell-\alpha\equiv x\mod k+1$. Let $Z_{k+2},\dots,Z_3$ be the last $k$ consecutive terms of $Q''$ which correspond to a copy of $K_k$ in $S$. Define $Q'''$ as the result of inserting $x$ copies of $Z_{k+2},\dots,Z_1$ into $Q''$ right before the last occurrence of $Z_{k+2},\dots,Z_3$ in $Q''$. Let $q''':=|Q'''|$. Define $\alpha_x:=\sum_{i=1}^{c-1}(q_i+f_i)+f_c+q'''=\alpha+x(k+2)$. Define the following.
\begin{gather*}
p_0:=\max\left\{p\in\mathbb{Z}\mid\ell\geq p\dot(1-d)(k+1)\frac{n}{r}+\alpha_x\right\}; \\
t_i=
\begin{cases}
(1-d)\frac{n}{r}&\textrm{ if }i\in[p_0]\\
\frac{\ell-\alpha_x}{k+1}-p_0(1-d)\frac{n}{r}&\textrm{ if }i=p_0+1\\
0&\textrm{ if }i>p_0+1;
\end{cases} \\
L_0=0, L_j=\sum_{i=1}^j[t_i(k+1)+q_i+f_i]\textrm{ for }j\in[c-1]; \\
M_j=L_{j-1}+t_j(k+1)+f_j\textrm{ for }j\in[c].
\end{gather*}
Define $\phi:V(H)\to V(R)$ as follows. For $i\leq\ell$, set
\[
\phi(v_i)=
\begin{cases}
Y_{jh}&\textrm{ if }L_{j-1}<i\leq M_j,\textrm{ with }h\equiv i-L_{j-1}\mod k+1 \\
(Q^*_j)_{i-M_j}&\textrm{ if }M_j<i\leq L_j \\
(Q''')_{M_c+q'''+1-i}&\textrm{ if }M_c<i\leq\ell.
\end{cases}
\]
For $i>\ell$, given $\phi(v_1),\dots,\phi(v_{i-1})$, set $\phi(v_i)=V_j$ with $j=\min\{h:|\{b<i:\phi(v_b)=V_h\}|<|V_h|\}$.

Set $X_i:=\phi^{-1}(V_i)$, $\bar{X}_i:=\phi^{-1}(V_i)\backslash C$ for $i\in[r]$. Define $\mathcal{X}:=\{X_i\}_{i\in[r]}$ and $\bar{\mathcal{X}}:=\{\bar{X}_i\}_{i\in[r]}$. Since all edges in $H$ have pairs of vertices in $C$ at most $k$ apart in the cyclic order as endpoints and any $k+1$ consecutive vertices in the cyclic order are mapped to a copy of $K_{k+1}$ in $R$, it follows that $(H,\mathcal{X})$ is an $R$-partition. Furthermore, for each $i\in[r]$ at most $(1-d)\frac{n}{r}+\frac{3r\binom{r}{k}+k(k+2)}{k+1}\leq(1-d+\eps)\frac{n}{r}\leq(1-\eps)\frac{n}{r}\leq|V_i|$ vertices in $C$ are mapped to $V_i$, so $\mathcal{X}$ is a vertex partition of $H$ which is size-compatible with~$\mathcal{V}$. Finally, $\bar{X}_i$ is a set of isolated vertices in $H$ by definition and
\[|\bar{X}_i|=|X_i|-|C\cap X_i|\geq\left(1-\frac{1-d+\eps}{1-\eps}\right)|X_i|\geq\alpha|X_i|\]
for each $i\in[r]$, so $\bar{\mathcal{X}}$ is an $\alpha$-buffer for $H$. This completes the proof in this case and for~\ref{item:embed-(k+2)-clique}.

Now we prove~\ref{item:embed-path-fixed-ends}. Fix $k\geq3$, $d>0$ and let $\Delta=2k, \Delta_J=k, \kappa=2, \alpha=\frac{d}{2},\zeta=1$. Now Lemma~\ref{lem:blow-up-dense} outputs $\eps_0,\rho_0>0$. We choose 
\[\eps_{EL}=\min\left\{\frac{\eps_0}{k+3},\frac{d^2}{8(k+1)}\right\}.\]
Given $0<\eps<\eps_{EL},r_{EL}\in\NN$, Lemma~\ref{lem:blow-up-dense} outputs $n_{BL}\in\NN$. We choose 
\[n_{EL}=\max\left\{n_{BL},\frac{6r_{EL}^{k+2}}{\eps},\frac{4r_{EL}}{\rho_0}\right\}.\]
Let $n\geq n_{EL}$, let $G$ be a graph on $n$ vertices and let $R^*$ be an $(\eps,d)$-reduced graph of $G$ on $r\leq r_{EL}$ vertices. Let $V'_0,V'_1,\dots,V'_r$ be the vertex classes of the $(\eps,d)$-regular partition of $G$ which gives rise to $R^*$. Let $\mathcal{T}'$ be the given connected $K_{k+1}$-factor in $R^*$ with $t:=|\mathcal{T}'|$ copies of $K_{k+1}$. Let $T'_1,\dots,T'_t$ be the copies of $K_{k+1}$ of $\mathcal{T}'$. Let $A':=\{u_{i,j}\mid (i,j)\in[2]\times[k]\}$.

Consider $T'_i=X'_{i,1}\dots X'_{i,(k+1)}$ for $i\in[t]$. Let $j\in[k+1]$. Remove the vertices of $A\cup A'$ from $X'_{i,j}$ to obtain $X_{i,j}$.
We have $|X_{i,j}|\geq\eps|X'_{i,j}|$ and $|X_{i,h}|\geq\eps|X'_{i,h}|$, so the $(\eps,d)$-regularity of $(X'_{i,j},X'_{i,h})$ implies that $(X_{i,j},X_{i,h})$ is $(2\eps,d-\eps)$-regular.

Let $\{V_0,\dots,V_r\}$ be the new vertex partition obtained by replacing each $X'_{i,j}$ with $X_{i,j}$ and let $\mathcal{V}:=\{V_1,\dots,V_r\}$. Let $R$ be the $(2\eps,d-\eps)$-full-reduced graph of the partition $\mathcal{V}$. Every edge of $R^*$ carries over to $R$, and let $V_i$ be the vertex of $R$ corresponding to $V'_i$ in $R^*$. Let $\mathcal{T}$ be the connected $K_{k+1}$-factor in $R$ corresponding to $\mathcal{T}'$. Let $T_1,\dots,T_t$ be the copies of $K_{k+1}$ in $\mathcal{T}$, with $T_i$ corresponding to $T'_i$ for all $i\in[t]$. Let $G^*$ be the subgraph of $G$ induced on $\mathcal{V}$. Let $n^*:=|V(G^*)|$. Here $(G^*,\mathcal{V})$ is a $(2\eps,d-\eps)$-regular $R$-partition. Note that $|V_i|\geq(1-3\eps)\frac{n}{r}\geq\frac{n}{2r_{EL}}$ for all $i\in[r]$, so $\mathcal{V}$ is 2-balanced.

Let $\ell'=\ell-2k$. Let $H$ be a copy of $P^k_{\ell'}$ together with additional isolated vertices so that it has $n^*$ vertices. Let $w_1,\dots,w_{n^*}$ be its vertices, with $w_1,\dots,w_{\ell'}$ being the vertices of the copy of $P^k_{\ell'}$ in a path order. Let $P:=\{w_i: i\in[\ell']\}$. Suppose that we have a vertex partition $\mathcal{X}:=\{X_i\}_{i\in[r]}$ of $H$ and a family $\bar{\mathcal{X}}:=\{\bar{X}_i\}_{i\in[r]}$ of subsets of $V(H)$ such that $\mathcal{X}$ is size-compatible with $\mathcal{V}$, $(H,\mathcal{X})$ is an $R$-partition, and $\bar{\mathcal{X}}$ is an $\alpha$-buffer for $H$. Suppose further that for each $j\in[k]$ we have $X_{1,j}=V_i$ and $X_{2,j}=V_h$ with $i,h$ such that $w_j\in X_i$ and $w_{\ell'-j+1}\in X_h$. Define $\mathcal{I}:=\{I_x\}_{x\in V(H)}$ and $\mathcal{J}:=\{J_x\}_{x\in V(H)}$ as follows.
\[
I_{w_j}=
\begin{cases}
V_i\cap\Gamma(u_{1,j},\dots,u_{1,k}) &\textrm{ for }j\in[k],\textrm{ with }w_j\in X_i  \\
V_i&\textrm{ for }k<j\leq\ell'-k,\textrm{ with }w_j\in X_i \\
V_i\cap\Gamma(u_{2,k},\dots,u_{2,(\ell'-j+1)}) &\textrm{ for }\ell'-k<j\leq\ell',\textrm{ with }w_j\in X_i,
\end{cases}
\]
\[
J_{w_j}=
\begin{cases}
\{u_{1,j},\dots,u_{1,k}\} &\textrm{ for }j\in[k]  \\
\nth &\textrm{ for }k<j\leq\ell'-k \\
\{u_{2,k},\dots,u_{2,(\ell'-j+1)}\} &\textrm{ for }\ell'-k<j\leq\ell'.
\end{cases}
\]
Since $|\Gamma(u_{i,j},\dots,u_{i,k})\cap X_{i,j}|\geq\frac{2dn}{r}-\frac{2\eps n}{r}\geq\frac{3dn}{2r}$ for each pair $(i,j)\in[2]\times[k]$ and $|V_i|\geq(1-2\eps)\frac{n}{r}\geq\frac{2}{\rho_0}$, $\mathcal{I}$ and $\mathcal{J}$ form a $(\rho_0,\zeta,\Delta,\Delta_J)$-restriction pair.

Then, by Lemma~\ref{lem:blow-up-dense} we will have an embedding $\phi:V(H)\to V(G^*)$ such that $w_j$ is adjacent to $u_{1,j},\dots,u_{1,k}$ for $j\in[k]$ and $w_j$ is adjacent to $u_{2,k},\dots,u_{2,(\ell-j+1)}$ for $\ell'-k<j\leq\ell'$. Together with $u_1,\dots,u_k,v_1,\dots,v_k$, this will yield a copy of $P^k_\ell$ which starts in $u_1,\dots,u_k$ and ends in $v_1,\dots,v_k$ (in those orders), contains no element of $A$ and has at most $(d+\eps)n$ vertices not in $\bigcup\mathcal{T'}$, which will then complete our proof of~\ref{item:embed-path-fixed-ends}. Therefore, for suitable values of $\ell'$ it remains to find a vertex partition $\mathcal{X}$ of $H$ and a family $\bar{\mathcal{X}}$ of subsets of $V(H)$ such that $\mathcal{X}$ is size-compatible with $\mathcal{V}$, $(H,\mathcal{X})$ is an $R$-partition, $\bar{\mathcal{X}}$ is an $\alpha$-buffer for $H$ and for each $j\in[k]$ we have $X_{1,j}=V_i$ and $X_{2,j}=V_h$ with $i,h$ such that $w_j\in X_i$ and $w_{\ell'-j+1}\in X_h$.

We consider $\ell\in\left(3r^{k+1},\frac{(1-d)(k+1)tn}{r}\right]$. Let $S$ be a copy of $K_{k+2}$ in the same $K_{k+1}$-component of $R$ as the copies of $K_{k+1}$ in $\mathcal{T}$ and let $Z_1,\dots,Z_{k+2}$ be the vertices of $S$. Fix a $K_{k+1}$-walk $W_0$ whose first copy of $K_k$ is $X_{1,1}\dots X_{1,k}$ and whose last is in $T_1$, which is of minimal length. For each $i\in[t-1]$, fix a $K_{k+1}$-walk $W_i$ whose first copy of $K_k$ is in $T_i$ and whose last is in $T_{i+1}$, which is of minimal length. Fix a $K_{k+1}$-walk $W_t$ whose first copy of $K_k$ is in $T_t$, whose last is $X_{2,1}\dots X_{2,k}$, which includes a copy of $K_k$ from $S$ and is one of minimal length satisfying these conditions. We have $|W_i|\leq\binom{r}{k}$ for $i\in[t-1]\cup\{0\}$ and $|W_t|\leq2\binom{r}{k}$. Let $W'$ be the $K_{k+1}$-walk obtained by concatenating $W_0,\dots,W_t$.

We construct the sequence $Q(W,\overrightarrow{U_{11}\dots U_{1k}})$ for any $K_{k+1}$-walk $W=(E_1,E_2,\dots)$ in $R$ and any orientation $\overrightarrow{U_{11}\dots U_{1k}}$ of $E_1$, its first copy of $K_k$, identically to that in~\ref{item:embed-divisible}. Orient $X_{1,1}\dots X_{1,k}$ as $\overrightarrow{X_{1,1}\dots X_{1,k}}$. Construct $Q(W',\overrightarrow{X_{1,1}\dots X_{1,k}})$. Then, the first copy of $K_k$ $U_{i1}\dots U_{ik}$ and the last copy of $K_k$ $U'_{i1}\dots U'_{ik}$ of each $W_i$ obtain orientations, say $\overrightarrow{U_{i1}\dots U_{ik}}$ and $\overrightarrow{U'_{i1}\dots U'_{ik}}$. Construct $Q(W_i,\overrightarrow{U_{i1}\dots U_{ik}})$ for $i\in[t]\cup\{0\}$. Clearly, there are sequences $\bar{Q}_i$ of vertices in $T_i$ for $i\in[t]$, such that $Q(W',\overrightarrow{X_{1,1}\dots X_{1,k}})$ is the concatenation of
\[Q(W_0,\overrightarrow{X_{1,1}\dots X_{1,k}}),\bar{Q}_1,Q(W_1,\overrightarrow{U_{11}\dots U_{1k}}),\dots,\bar{Q}_t,Q(W_t,\overrightarrow{U_{t1}\dots U_{tk}}).\]
Define $f_i\equiv|\bar{Q}_i|\mod k+1$ for $i\in[t]$. Let $Q^*_0$ denote $Q(W_0,\overrightarrow{X_{1,1}\dots X_{1,k}})$ and let $Q^*_i$ denote $Q(W_i,\overrightarrow{U_{i1}\dots U_{ik}})$ for $i\in[t]$. Define $q_i:=|Q^*_i|$ for $i\in[t]\cup\{0\}$. For a sequence $Q$ of vertices of $R$, let $(Q)_h$ denote the $h$th term of $Q$. For each $i\in[t]$ let $T_i=Y_{i1}\dots Y_{i(k+1)}$ be such that $\overrightarrow{Y_{i2}\dots Y_{i(k+1)}}$ is the oriented last copy of $K_k$ of $W_{i-1}$ in $Q^*_{i-1}$.

Let $\alpha:=q_0+\sum_{i=1}^t(q_i+f_i)$. Pick $x\in[k]\cup\{0\}$ such that $\ell'-\alpha\equiv x\mod k+1$. Let $Z_3,\dots,Z_{k+2}$ be the first $k$ consecutive terms of $Q^*_t$ which correspond to a copy of $K_k$ in $S$. Define $Q'$ as the result of inserting $x$ copies of $Z_1,\dots,Z_{k+2}$ into $Q^*_t$ right after the first occurrence of $Z_3,\dots,Z_{k+2}$ in $Q^*_t$. Let $q':=|Q'|$. Define $\alpha_x:=q_0+\sum_{i=1}^{t-1}(q_i+f_i)+f_t+q'=\alpha+x(k+2)$. Define the following.
\begin{gather*}
p_0:=\max\left\{p\in\mathbb{Z}\mid\ell'\geq p\dot(1-d)(k+1)\frac{n}{r}+\alpha_x\right\}; \\
t_i=
\begin{cases}
(1-d)\frac{n}{r}&\textrm{ if }i\in[p_0]\\
\frac{\ell'-\alpha_x}{k+1}-p_0(1-d)\frac{n}{r}&\textrm{ if }i=p_0+1\\
0&\textrm{ if }i>p_0+1;
\end{cases} \\
L_0=q_0, L_j=q_0+\sum_{i=1}^j[t_i(k+1)+q_i+f_i]\textrm{ for }j\in[t-1]; \\
M_0=0, M_j=L_{j-1}+t_j(k+1)+f_j\textrm{ for }j\in[t].
\end{gather*}
Define $\phi:V(H)\to V(R)$ as follows. For $i\leq\ell'$, set
\[
\phi(v_i)=
\begin{cases}
Y_{jh}&\textrm{ if }L_{j-1}<i\leq M_j,\textrm{ with }h\equiv i-L_{j-1}\mod k+1 \\
(Q^*_j)_{i-M_j}&\textrm{ if }M_j<i\leq L_j \\
(Q')_{i-M_t}&\textrm{ if }M_t<i\leq\ell'.
\end{cases}
\]
For $i>\ell'$, given $\phi(v_1),\dots,\phi(v_{i-1})$, set $\phi(v_i)=V_j$ with $j=\min\{h:|\{b<i:\phi(v_b)=V_h\}|<|V_h|\}$.

Set $X_i:=\phi^{-1}(V_i)$, $\bar{X}_i:=\phi^{-1}(V_i)\backslash P$ for $i\in[r]$. Define $\mathcal{X}:=\{X_i\}_{i\in[r]}$ and $\bar{\mathcal{X}}:=\{\bar{X}_i\}_{i\in[r]}$. Since all edges in $H$ have pairs of vertices in $P$ at most $k$ apart in the path order as endpoints and any $k+1$ consecutive vertices in the cyclic order are mapped to a copy of $K_{k+1}$ in $R$, it follows that $(H,\mathcal{X})$ is an $R$-partition. Furthermore, for each $i\in[r]$ at most $(1-d)\frac{n}{r}+\frac{3r\binom{r}{k}+k(k+2)}{k+1}\leq(1-d+\eps)\frac{n}{r}\leq(1-2\eps)\frac{n}{r}\leq|V_i|$ vertices in $P$ are mapped to $V_i$, so $\mathcal{X}$ is a vertex partition of $H$ which is size-compatible with~$\mathcal{V}$. Finally, $\bar{X}_i$ is a set of isolated vertices in $H$ by definition and
\[|\bar{X}_i|=|X_i|-|P\cap X_i|\geq\left(1-\frac{1-d+\eps}{1-3\eps}\right)|X_i|\geq\alpha|X_i|\]
for each $i\in[r]$, so $\bar{\mathcal{X}}$ is an $\alpha$-buffer for $H$. This completes the proof for~\ref{item:embed-path-fixed-ends}.
\end{proof}

\end{document}